\documentclass[leqno]{amsart}
\usepackage{amsmath,amssymb,amsthm,amsfonts, amscd}
\usepackage[cm]{fullpage}
\usepackage{color}
\numberwithin{equation}{section}
\newtheorem{lemma}{Lemma}[section]

\newtheorem{proposition}{Proposition}[section]

\newtheorem{theorem}{Theorem}[section]

\newtheorem{corollary}{Corollary}[section]

\newtheorem{Definition}{Definition}[section]
\newtheorem{Remark}{Remark}[section]
\newtheorem{Conjecture}{Conjecture}
\newtheorem{definition}{Definition}[section]

\newcommand{\mM}{\mathbb{M}}
\newcommand{\mn}{\mathbf{n}}
\newcommand{\mQ}{\mathbb{Q}}
\newcommand{\mR}{\mathbb{R}}
\newcommand{\mN}{\mathbb{N}}
\newcommand{\mZ}{\mathbb{Z}}
\newcommand{\mA}{\mathbb{A}}
\newcommand{\mC}{\mathbb{C}}
\newcommand{\mO}{\mathcal{O}}

\newcommand{\mc}{\mathfrak{c}}
\newcommand{\mw}{\mathbf{w}}
\title{On a certain local identity for Lapid--Mao's conjecture and formal degree conjecture : even unitary group case}
\author{Kazuki Morimoto}
\begin{document}
\begin{abstract}
Lapid and Mao formulated a conjecture on an explicit formula of Whittaker Fourier coefficients of automorphic forms
on quasi-split classical groups and metaplectic groups as an analogue of Ichino-Ikeda conjecture.
They also showed that this conjecture is reduced to a certain local identity in the case of unitary groups.
In this paper, we study even unitary group case.
Indeed, we prove this local identity over $p$-adic fields.
Further, we prove an equivalence between this local identity and a refined formal degree conjecture over any local field of characteristic zero.
As a consequence, we prove a refined formal degree conjecture over $p$-adic fields
and we get an explicit formula of Whittaker Fourier coefficients under certain assumptions.
\end{abstract}
\maketitle
%%%%%%%%%%%%%%%%%%%%%%%%%%%%%%%%%%%%%%%%%%%%%%%
%
%
%
%
%
%
%
%
%
%
%
%
%
%
%%%%%%%%%%%%%%%%%%%%%%%%%%%%%%%%%%%%%%%%%%%%%%%
\section{Introduction}
%Let $G$ be a reductive group over a number field $F$ with adele ring $\mA$.
A relationship between periods of automorphic forms and special values of 
certain $L$-functions has been studied in several situations.
For example, Gross and Prasad~\cite{GP} conjectured a relationship between non-vanishing of Bessel period on 
special orthogonal groups and non-vanishing of central values of $L$-functions.
Moreover, Ichino and Ikeda~\cite{II} refined the Gross--Prasad conjecture and conjectured
an explicit formula between Bessel periods and central values.
%$L$-functions for $\mathrm{SO}(n) \times \mathrm{SO}(n+1)$.
%In their conjecture, the central values appear as a proportionally constant of two different constructions 
%of Bessel periods.
As an analogue of Ichino--Ikeda conjecture,  
Lapid and Mao~\cite{LMe} formulated a similar conjecture on an explicit formula of Whittaker periods
for quasi-split reductive groups and metaplectic groups,
which in the spirit of Sakellaridis and Venkatesh~\cite{SV}.

In this paper, we consider Lapid-Mao conjecture in the case of quasi-split skew-hermitian even unitary group $\mathrm{U}_{2n}$
associated to a quartic extension $E$ of $F$.
Let us recall their conjecture in our case.
For an automorphic form $\varphi$ on $\mathrm{U}_{2n}(\mA)$, we consider the Whittaker coefficient
\[
\mathcal{W}^{\psi_N}(\varphi) := \left(\mathrm{vol}(N(F) \backslash N(\mA)) \right)^{-1}
\int_{N(F) \backslash N(\mA)} \varphi(n) \psi_N(n)^{-1} \, dn.
\]
Here, $N$ is a maximal unipotent subgroup of $\mathrm{U}_{2n}$ and $\psi_N$ is a non-degenerate character of $N(\mA)$.
We also consider the $\mathrm{U}_{2n}(\mA)$-invariant pairing
\[
(\varphi_1, \varphi_2) := \left(\mathrm{vol}(\mathrm{U}_{2n}(F) \backslash \mathrm{U}_{2n}(\mA)) \right)^{-1}
\int_{\mathrm{U}_{2n}(F) \backslash \mathrm{U}_{2n}(\mA)} \varphi_1(g) \varphi_2(g) \, dg
\]
of two square-integrable automorphic forms $\varphi_1, \varphi_2$ on $\mathrm{U}_{2n}(\mA_F)$.
Then we would like to construct Whittaker functional using pairing $(-, -)$.
%For it, we should recall an important notion, that is stable integral.
Given a finite set of places $S$ of $F$, Lapid and Mao~\cite{LMe} defined a stable integral
\[
\int_{N(\mA^S)}^{st} f(n) \, dn
\]
for a suitable class of smooth function $f$ on $N(\mA^S)$ with $\mA^S = \prod_{v \in S} F_v$.
In particular, when $S$ consists of finite places, there is a sufficiently large compact open subgroup $N_0$ of $N(\mA^S)$
such that the above integral is equal to 
\[
\int_{N_0^\prime} f(n) \, dn
\]
for any compact open subgroup $N_0^\prime \supset N_0$.
Let $\sigma$ be an irreducible $\psi_{N}$-generic cuspidal automorphic representation of $\mathrm{U}_{2n}(\mA_F)$.
Lapid and Mao~\cite{LMe} proved that for a $\varphi \in \sigma, \varphi^\vee \in \sigma^\vee$ and for 
a sufficiently large finite set $S$ of places, 
\[
I(\varphi, \varphi^\vee) := \int_{N(\mA^S)}^{st} (\sigma(n)\varphi, \varphi^\vee) \psi_{N}^{-1}(n) \, dn
\]
is defined, and $I(\varphi, \varphi^\vee)$ satisfies $I(\sigma(n)\varphi, \sigma^\vee(n^{\prime-1})\varphi^\vee) = \psi_N(n n^\prime) I(\varphi, \varphi^\prime)$.
Then they formulated a conjecture on an explicit formula between these two Whittaker functional 
in the sprit of Ichino-Ikeda conjecture.

Let $\pi$ be the base change lift of $\sigma$ to $\mathrm{GL}_{2n}(\mA_E)$ constructed by 
Kim and Krishnamurthy~\cite{KK}. It is an isobaric sum $\pi_1 \boxplus \cdots \boxplus \pi_k$
where $\pi_i$ is an irreducible cuspidal automorphic representation of $\mathrm{GL}_{n_i}(\mA_E)$
with $n_1 + \cdots n_k = 2n$, such that $L^S(s, \pi_i, \mathrm{As}^{-})$ has a simple pole at $s =1$.
Here, $L^S(s, \pi_i, \mathrm{As}^{-})$ is the partial Asai $L$-function of $\pi_i$ defined  in \cite{GRS}.
Also, we have another Asai $L$-function $L^S(s, \pi_i, \mathrm{As}^{+})$ of $\pi_i$, which satisfies
$L^S(s, \pi_i, \mathrm{As}^-) = L^S(s, \pi_i \otimes \Upsilon, \mathrm{As}^+)$
where $\Upsilon$ is a character of $\mA_E^\times \slash E^\times$ whose restriction
to $\mA^\times$ is the quadratic character $\eta_{E \slash F}$ corresponding to the quadratic extension $E \slash F$.
\begin{Conjecture}[Conjecture~1.2, 5.1 in \cite{LMe}]
\label{main conj}
Let $\sigma$ and $\pi$ be as above.
Then for any $\varphi \in \sigma$ and $\varphi^\vee \in \sigma^\vee$ 
and for any sufficiently large $S$ finite set of places, we have
\[
\mathcal{W}^{\psi_{N}}(\varphi)\mathcal{W}^{\psi_{N}^{-1}}(\varphi^\vee)
=2^{1-k}  \frac{\prod_{j=1}^{2n} L^{S}(j, \eta_{E \slash F}^j)}{L^S(1, \pi, \mathrm{As}^+)}
(\mathrm{vol}(N(\mathcal{O}_S) \backslash N(\mA^S) ))^{-1}
\int_{N(F_S)}^{st} (\sigma(u)\varphi, \varphi^\prime) \psi_{N}(u)^{-1} \, du.
\]
Here, $\mathcal{O}_S$ is the ring of $S$-integers of $F$.
\end{Conjecture}
In \cite{LMa}, for any place $v$ of $F$, Lapid and Mao defined the constant $c_{\pi_v}$ depending on $\pi_v$ (also on $\sigma_v$),
which is given as a proportionality constant of two local Whittaker periods.
Then they proved the following theorem.
\begin{theorem}[Theorem~5.5 in Lapid--Mao \cite{LMa}]
Keep the above notation.
Then for any $\varphi \in \sigma$ and $\varphi^\vee \in \sigma^\vee$ 
and for any sufficiently large $S$ finite set of places, we have
\[
\mathcal{W}^{\psi_{N}}(\varphi)\mathcal{W}^{\psi_{N}^{-1}}(\varphi^\vee)
=2^{1-k} (\prod_{v} c_{\pi_v}^{-1}) \frac{\prod_{j=1}^{2n} L^{S}(j, \eta_{E \slash F}^j)}{L^S(1, \pi, \mathrm{As}^+)}
(\mathrm{vol}(N(\mathcal{O}_S) \backslash N(\mA^S) ))^{-1}
\int_{N(F_S)}^{st} (\sigma(u)\varphi, \varphi^\prime) \psi_{N}(u)^{-1} \, du.
\]
\end{theorem}
\begin{Remark}
They proved this reduction under two working assumptions; 
certain properties of certain local zeta integrals and the irreducibility of global descents for $\mathrm{U}_{2n}$.
The first assumption was proved in Ben-Artzi--Soudry~\cite{BAS} and the 
second assumption was proved by the author~\cite{Mo1}.
Hence, the above theorem holds without any assumption.
\end{Remark}
Because of this theorem, in order to prove Conjecture~\ref{main conj}, it suffices to show $\prod_{v} c_{\pi_v} =1$.
Indeed, they conjectured the following identity.
\begin{Conjecture}[Conjecture~5.8 in \cite{LMa}]
\label{LM local conjecture}
Let $v$ be a place of $F$.
Then 
\[
c_{\pi_v} = \omega_{\sigma_v}(-1).
\]
\end{Conjecture}
In particular, this conjecture concludes that $\prod_{v} c_{\pi_v} =1$.
We note that when $E_v:= E \otimes_F F_v$ is a quadratic extension, we have $\omega_{\sigma_v}(-1) = \omega_{\pi_v}(\tau)$
for $\tau \in E_v$ such that $\mc(\tau) = - \tau$ with $1 \ne \mc \in \mathrm{Gal}(E_v \slash F_v)$.
Then the above conjecture is equivalent to
\begin{equation}
\label{main thm eq}
c_{\pi_v} = \omega_{\pi_v}(\tau).
\end{equation}
One of our local main theorems of the present paper is the proof of this conjecture in an inert case.
\begin{theorem}
\label{main thm}
Suppose that $v$ is non-split finite place (i.e. $\mathrm{U}_{2n}(F_v)$ is quasi-split unitary group).
Then Conjecture~2 holds. 
\end{theorem}
A similar identity was proved by \cite{LMb} in the case of metaplectic groups.
In this paper, following their argument we shall prove this theorem.
%As in \cite{LMb}, studying several analytic properties of related integrals and using 
%certain functional equation, we shall reduce our theorem to certain identity (See).
%The novelty of this paper is t
%Further, in the above analysis, we need a stability of certain unipotent integral.
%In \cite{LMb}, they used their results .
%In our case, we need to generalize the stability to the case of quasi-split case.

This local identity is important not only for Lapid--Mao conjecture but also for formal degree conjecture in Hiraga--Ichino--Ikeda~\cite{HII}.
Using an observation by Gross-Reeder~\cite{GR}, Gan--Ichino~\cite[14.5]{GI} formulated a refinement of this conjecture for classical groups.
We call this refinement refined formal degree conjecture.
In \cite{ILM}, Ichino--Lapid--Mao showed that a similar local identity to Conjecture~\ref{LM local conjecture} proved in \cite{LMb} is equivalent to refined formal degree conjecture
for metaplectic groups. As a consequence of \cite{LMb}, they proved refined formal degree conjecture in this case.
In a similar argument as \cite{ILM}, we prove the equivalence between refined formal degree conjecture for $\mathrm{U}_{2n}$
and Theorem~\ref{main thm}. As a consequence, we can prove refined formal degree conjecture.
\begin{theorem}
\label{formal deg intro}
Let $F$ be a non-arhimedean local field of characteristic zero and $E$ a quadratic extension of $F$.
Let $\pi$ be an irreducible representation of $\mathrm{GL}_{2n}(E)$ of the form 
$\pi = \tau_1 \times \cdots \times \tau_k$ where $\tau_i$ are mutually inequivalent irreducible discrete series representations
of $\mathrm{GL}_{n_i}(E)$ such that $n=n_1+\cdots + n_k$ and $L(s, \tau_i, \mathrm{As}^+)$ has a pole at $s=0$.
Write $\sigma = \mathcal{D}_{\psi^{-1}}^{\Upsilon^{-1}}(\mc(\pi))$, which is an irreducible generic discrete series representation of $G^\prime$ 
Then we have
\[
d_\psi =2^k
\lambda(E \slash F)^n \omega_{\sigma}(-1) \gamma(1, \mathfrak{c}(\pi), \mathrm{As}^-, \psi) d_\sigma.
\]
Here, $d_\sigma$ is the formal degree (See Section~8) and $d_\psi$ is a certain measure on $\mathrm{U}_{2n}(F)$ (See Section~2.4).
\end{theorem}
\begin{Remark}
Recently, Beuzart-Plessis~\cite{BP} proved the original formal degree conjecture by \cite{HII}
for any unitary groups using a different method. 
%In our case. he proved  $d_\psi =|2|^k |\gamma(1, \mathfrak{c}(\pi), \mathrm{As}^-, \psi)| d_\sigma$.
\end{Remark}
On the other hand, we can show refined formal degree conjecture for $\mathrm{U}_{2n}$ over real field using computations in  \cite{HII}.
In a similar argument as the non-archimedean case, we can prove the equivalence between refined formal degree conjecture
and Conjecture~\ref{LM local conjecture}. Hence, we obtain the following result.
\begin{theorem}
\label{main thm archimedean}
Suppose that $v$ is a real place and that $E_v \simeq \mC$, i.e. $\mathrm{U}_{2n}(F_v) \simeq \mathrm{U}_{2n}(\mR)$.
Then Conjecture~2 holds for diecerete series representations of $\mathrm{U}_{2n}(\mR)$. 
\end{theorem}
As a corollary of Theorem~\ref{main thm}, Theorem~\ref{main thm archimedean} and \cite[Lemma~5.4]{LMa}, the following global formula follows.
%See also the remark below.
\begin{corollary}
\label{into cor 1}
Let $\sigma$ be as in Conjecture~\ref{main conj}.
Suppose that $F$ is totally real field.
Assume that for a places $v$ of $F$,
\[
\left\{
\begin{array}{ll}
\sigma_v\text{ is unramified} & \text{ if $v$ is split finite place or $v | 2$}
\\
&
\\
\sigma_v\text{ is discrete series} & \text{ if $v$ is real place}
\end{array}
\right.
\]
Then the formula in Conjecture~\ref{main conj} holds.
\end{corollary}
In \cite{FM}, we proved the refined Gross-Prasad conjecture for special 
Bessel periods for $\mathrm{SO}(2n+1)$ using the explicit formula of Whittaker periods of 
metaplectic groups by Lapid--Mao~\cite{LMb}.
In a similar computation as \cite{FM}, Corollary~\ref{into cor 1} should yield a formula of Bessel periods 
for $(\mathrm{U}_{2n}, \mathrm{U}_1)$. In our future work, we will study this problem.

This paper is organized as follows.
In Section~2, we define basic notations.
In Section~3, we formulate Lapid-Mao's local conjecture 
and reduce it to tempered case.
In Section~4--6, following the idea by Lapid and Mao~\cite{LMb},
we rewrite our required identity by another local identity using local analysis and certain functional equations.
(See \cite[Section~4]{LMb} for the idea of the proof).
Most parts are proved in a similar argument as \cite{LMb}.
Hence, we give a proof only when there are non-trivial difference,
and we only give statements if the proof is essentially same as the corresponding results in \cite{LMb}.
In Section~7, we complete the proof of Theorem~\ref{main thm}.
In Section~8, we prove Theorem~\ref{formal deg intro} and Theorem~\ref{main thm archimedean}.
\subsection*{Acknowledgements}
The author express his deep gratitude to Erez Lapid for suggesting him to the problems in the present paper
and for useful discussion.
%\begin{Remark}
%Let $v$ be a split finite place.
%As remarked in \cite[Remark~5.3]{LMa}, $c_{\pi_v}$ does not depend on a choice of character $\psi$.
%Further, in this case, the Whittaker model does not depend on a choice of a generic character.
%In the proof of \cite[Lemma~5.4]{LMa}, they used the assumption $p \ne 2$
%because of (2.2), last part of Working Assumption~5.2 and part 4 of Lemma~5.1 in \cite{LMa}.
%In (2.2), we do not need this assumption by \cite{KeSh} (see also \cite{Hei}).
%By normalizing characters suitably  we can remove a multiplication by a power of $2$ from the computation in the proof of \cite[(6.15)]{BAS}.
%Since the unramified computation in \cite{BAS2} does not need the assumption on the residual characteristic,
%we can prove last part of Working Assumption~5.2 without the assumption $p \ne 2$.
%Further, by using the same normalization of the characters and by normalizing the isomorphism between certain unipotent subgroup
%and Heisenberg group in Section~4.2, the same computation as part 4 of Lemma~5.1 holds without $p \ne 2$.
%Therefore, in the same argument as the proof of \cite[Proposition~5.4]{LMc}, we get $c_{\pi_v} = 1$.
%\end{Remark}
%%%%%%%%%%%%%%%%%%%%%%%%%%%%%%%%%%%%%%%%%%%%%%%
%
%
%
%
%
%
%
%
%
%
%
%
%
%
%%%%%%%%%%%%%%%%%%%%%%%%%%%%%%%%%%%%%%%%%%%%%%%
\section{Notation and Preliminaries}
\subsection{Groups, homomorphisms and group elements}
\label{notation 1}
\begin{itemize}
\item Let $F$ be a local field of characterstic zero and $E$ a quadratic extension of $F$.
Let $\mc$ be a non-trivial element of $\mathrm{Gal}(E \slash F)$. We take $\tau \in E$ such that $\mc(\tau) = -\tau$.
%%%%
\item $\eta_{E \slash F}$ denotes the quadratic character of $F^\times$ corresponding to $E \slash F$.
%%%%
\item Fix a character $\Upsilon$ of $E^\times$ such that $\Upsilon |_{F^\times} = \eta_{E \slash F}$.
%%%%
\item $I_m$ is the identity matrix in $\mathrm{GL}_m$, $w_m$is the $m \times m$-matrix with ones on the non- principal diagonal and zeros elsewhere.
%%%%
\item For any group $Q$, $Z_Q$ is the center of $Q$; $e$ is the identity element of $Q$. 
We denote the modulus function of $Q$ (i.e., the quotient of a right Haar measure by a left Haar measure) by $\delta_Q$.
%%%%
\item $\mathrm{Mat}_m$ is the vector space of $m \times m$ matrices over $F$.
%%%%
\item for $x = (x_ij) \in \mathrm{Mat}(E)$, $x^\mc$ denotes the matrix $(\mc(x_{ij}))$
%%%%
\item $x \mapsto {}^t x$ is the transpose on $\mathrm{Mat}_m$; $x \mapsto x^{\vee}$ is the twisted transpose map on $\mathrm{Mat}_m$
given by $x^{\vee}= w_m {}^{t}x^\mc w_m$; 
$g \mapsto g^\ast$ is the outer automorphism of $\mathrm{GL}_m$ given by $g^\ast= w_m^{-1} ({}^t g^\mc)^{-1} w_m$.
%%%%
\item $\mathfrak{u}_n = \left\{ x \in \mathrm{Mat}_{n}(E) : x^\vee = x \right\}$
%%%%
\item $\mM = \mathrm{GL}_{2n}(E), \mM = \mathrm{GL}_{n}(E)$
%%%%
\item $G = \mathrm{U}_{4n} = \left\{ g \in \mathrm{GL}_{4n}(E) : {}^{t} g \begin{pmatrix} &w_{2n}\\ -w_{2n}&\end{pmatrix}g = 
\begin{pmatrix} &w_{2n}\\ -w_{2n}&\end{pmatrix} \right\}$
%%%%
\item $G^\prime = \mathrm{U}_{2n} = \left\{ g \in \mathrm{GL}_{2n}(E) : {}^{t} g \begin{pmatrix} &w_{n}\\ -w_{n}&\end{pmatrix}g = 
\begin{pmatrix} &w_{n}\\ -w_{n}&\end{pmatrix} \right\}$
%%%%
\item $G^\prime$ is embedded as a subgroup of $G$ via $g \mapsto \eta(g) = \mathrm{diag}(I_n, g, I_n)$.
%%%%
\item $P=MU$ (resp., $P^\prime = M^\prime U^\prime$) is the Siegel parabolic subgroup of $G$ (resp., $G^\prime$),
with its standard Levi decomposition.
%%%%
\item $\overline{P}= {}^{t}P$ is the opposite parabolic of $P$, with unipotent radical $\overline{U} = {}^t U$.
%%%%
\item We use the isomorphism $\varrho(g) = \mathrm{diag}(g, g^\ast)$ to identify $\mM$ with $M \subset G$. 
Similarly for $\varrho^\prime : \mM^\prime \rightarrow M^\prime \subset G^\prime$.
%%%%
\item We use the embeddings $\eta_\mM(g) = \mathrm{diag}(g, I_n)$ and $\eta_\mM^\vee(g) = \mathrm{diag}(I_n, g)$ 
to identify $\mM^\prime$ with subgroups of $\mM$. We also set $\eta_M =\varrho \circ \eta_\mM$ 
and $\eta_\mM^\vee =\varrho \circ \eta_{\mM^\prime} =\eta \circ \varrho^\prime$
%%%%
\item $K$ is the standard maximal compact subgroup of $G$. (In the p-adic case it consists
of the matrices with integral entries.)
%%%%
\item $N$ is the standard maximal unipotent subgroup of $G$ consisting of upper unitriangular matrices; 
$T$ is the maximal torus of G consisting of diagonal matrices; $B=TN$ is the Borel subgroup of $G$.
%%%%
\item For any subgroup $X$ of $G$ we write $X^\prime = \eta^{-1}(X)$, $X_M = X \cap M$
and $X_\mM =\varrho^{-1}(X_M)$;
similarly $X_{M^\prime}^\prime =X^\prime \cap M^\prime$ and $X_{\mM^\prime}^\prime = (\varrho^\prime)^{-1}(X_{M^\prime}^\prime)$.
%%%%
\item $\ell_\mM : \mathrm{Mat}_n \rightarrow N_\mM$ is the group embedding given by 
$\ell_\mM(x) = \left( \begin{smallmatrix} I_n &x\\ &I_n\end{smallmatrix} \right)$ and $\ell_M = \varrho \circ \ell_{\mM}$.
%%%%
\item $\ell : \mathfrak{u}_{2n} \rightarrow U$ is the group isomorphism given by $\ell(x) = \left( \begin{smallmatrix} I_n &x\\ &I_n\end{smallmatrix} \right)$.
%%%%
\item $\xi_m = (0, \dots, 0, 1) \in F^m$
%%%%
\item $\mathcal{P}$ is the mirabolic subgroup of $\mM$ consisting of the elements $g$ such that $\xi_{2n} g = \xi_{2n}$
%%%%
\item Put $H_\mM = \mathrm{GL}_{2n}(F)$
%%%%
\item $\mathfrak{t} = \mathrm{diag}(1, -1, \dots, 1, -1) \in \mM$.
%%%%
\item $w_0^\prime =  \left( \begin{smallmatrix} &w_n\\ -w_n&\end{smallmatrix} \right) \in G^\prime$ represents the longest Weyl element of $G^\prime$.
%%%%
\item $w_U = \left( \begin{smallmatrix} &I_{2n}\\ -I_{2n}&\end{smallmatrix} \right) \in G$ represents the longest $M$-reduced Weyl element of $G^\prime$
%%%%
\item $w_{U^\prime}^\prime = \left( \begin{smallmatrix} &I_{n}\\ -I_{n}&\end{smallmatrix} \right) \in G^\prime$ represents the longest $M^\prime$-reduced 
Weyl element of $G^\prime$
%%%%
\item  $w_0^{\mM} = w_{2n} \in \mM$ represents the longest Weyl element of $\mM$; $w_0^M = \varrho(w_0^{\mM})$.
%%%%
\item $w_0^{\mM^\prime} = w_{n} \in \mM$ represents the longest Weyl element of $\mM$; $w_0^{M^\prime} = \varrho(w_0^{\mM^\prime})$.
%%%%
\item $w_{2n,  n} =  \left( \begin{smallmatrix} &I_n\\ I_{n}&\end{smallmatrix} \right) \in \mM$, $w_{2n,  n}^\prime =  
\left( \begin{smallmatrix} &I_n\\ w_0^{\mM^\prime}&\end{smallmatrix} \right) \in \mM$
%%%%
\item $\gamma = w_U \eta(w_{U^\prime}^\prime)^{-1} =  \left( \begin{smallmatrix} &I_n&&\\ &&&I_n\\ -I_n&&&\\ &&I_{n}&\end{smallmatrix} \right) \in G$
%%%%
\item $\mathfrak{d} = \mathrm{diag}(1, -1, \dots, (-1)^{n-1}) \in \mathrm{Mat}_n$, 
$\varepsilon_1 = (\hat{w} \gamma)^{-1} \varrho(\varepsilon_3) w_U = \ell_M((-1)^n \mathfrak{d})$,
$\varepsilon_2 = \ell_\mM(\mathfrak{d})$,
$\varepsilon_3 = w_{2n, n}^\prime \varepsilon_2$, 
$\varepsilon_4 = \ell_{\mM} (- \frac{1}{2} \mathfrak{d} w_0^{\mM^\prime})$.
%%%%
\item $V$ (resp. $V^{\#}$) is the unipotent radical of the standard parabolic subgroup of $G$ with Levi $\mathrm{GL}_1(E)^n \times \mathrm{U}_{2n}$
(resp. $\mathrm{GL}_1(E)^{n-1} \times \mathrm{U}_{2n+2}$).
Thus, $N = \eta(N^\prime) \ltimes V$, $V^{\#}$ is normal in $V$ and $V \slash V^{\#}$ is isomorphic to the Heisenberg group of dimension $2n+1$
over $E$. Also $V = V_M \ltimes V_{U}$ where $V_U = V\cap U = \left\{\ell \left(\begin{smallmatrix}x&y\\ &x^{\vee} \end{smallmatrix} \right) : 
x \in \mathrm{Mat}_n(E), y \in \mathfrak{u}_n \right\}$
%%%%
\item $V_- = V_M^{\#} \ltimes V_U$ (Recall $V_M^{\#} = V^{\#} \cap M$ by our convention)
%%%%
\item $V_\gamma = V \cap \gamma^{-1} N \gamma = \eta(w_{U^\prime}^\prime) V_M \eta(w_{U^\prime}^\prime) 
= \eta_M(N_{\mM^\prime}^\prime) \ltimes \left\{\ell \left(\begin{smallmatrix}x&\\ &x^{\vee} \end{smallmatrix} \right) : 
x \in \mathrm{Mat}_n(E) \right\} \subset V_-$
%%%%
\item $V_+ \subset V$ is the image under $\ell_M$ of the space of $n \times n$-matrices over $E$ whose rows are zero
except possibly for the last one. Thus, $V = V_+ \ltimes V_-$.
For $c = \ell_M(x) \in V_+$ we denote by $\underline{c}$ the last row of $x$.
%%%%
\item $N^{\#} = V_{-} \rtimes \eta(N^\prime)$.It is the stabilizer in $N$ of the character $\psi_U$ defined below.
%%%%
\item $N_{\mM}^{\flat} = (N_{\mM}^{\#})^\ast$
%%%%
\item $J$ is the subspace of $\mathrm{Mat}_n$ consisting of the matrices whose first column is zero.
%%%%
\item $\bar{R} =  \left\{\ell \left(\begin{smallmatrix}I_n&\\ x&{}^{t}n \end{smallmatrix} \right) : 
x \in J, n \in N_{\mM^\prime}^\prime \right\}$.
%%%%
%\item \color{red}$T^{\prime \prime} = Z_{\mM} \times \eta_{\mM}^\vee(T_{\mM^\prime}^\prime) = $ \color{black}
\end{itemize}
%%%%%%%%%%%%%%%%%%%%%%%%%%%%%%%%%%%%%%%%%%%%%%%
%
%
%
%
%
%
%
%
%
%
%
%
%
%
%%%%%%%%%%%%%%%%%%%%%%%%%%%%%%%%%%%%%%%%%%%%%%%
\subsection{Characters}
We fix a non-trivial additive character $\psi_F$ of $F$ and define
an additive character $\psi$ of $E$ by $\psi(x) = \psi_F (\frac{x +\mc(x)}{2})$ for $x \in E$.
\begin{itemize}
\item $\psi_{N_\mM}(u) = \psi(u_{1, 2} + \cdots + u_{2n-1, 2n})$
%%%%
\item $\psi_{N_M} \circ \varrho = \psi_{N_\mM}$
%%%%
\item $\psi_{N_{\mM^\prime}^\prime}(u) = \psi(u_{1, 2}^\prime + \cdots + u_{n-1, n}^\prime)$
%%%%
\item $\psi_{N_{M^\prime}^\prime} \circ \varrho^\prime = \psi_{N_{\mM^\prime}^\prime}$.
%%%%
\item $\psi_{N^\prime}(nu) = \psi_{N_{\mM^\prime}^\prime}(n) \psi(\frac{1}{2} u_{n, n+1})^{-1}$, $n \in N_{M^\prime}^\prime, u \in U^\prime$
%%%%
\item $\psi_N(nu) = \psi_{N_M}(n)$, $n \in N_M, u \in U$ (a degenerate character).
Then $\psi_{N_{M^\prime}^\prime}(u) = \psi_{N} (\gamma \eta(u) \gamma^{-1})$
%%%%
\item $\psi_{V}(vu) = \psi_{N_M}(w_U u w_U^{-1})$ where we write an element of $V$
by $vu$ so that $u$ fixes $e_1, \dots e_n$, $v$ fixes $e_{n+1}, \dots e_{n+1}, \dots e_{-1-n}$ 
%%%%
\item $\psi_{V_-}(vu) = \psi_{N_M}(v)^{-1} \psi_{U}(u)$, $v \in V_M^{\#}, u \in V_U$. (Note that this is not a restriction of $\psi_{V}$ to $V_-$.)
%%%%
\item $\psi_{U}(\ell(v)) = \psi \left(\frac{1}{2}(v_{n, n+1}-v_{2n, 1}) \right)$
%%%%
\item $\psi_{\bar{U}}(\bar{v}) = \psi(v_{2n+1,1})$, $\bar{v} \in \bar{U}$.
%%%%
\item $\Upsilon_M(\varrho(g)) = \Upsilon(\det g), g \in \mM$ and $\Upsilon_{M^\prime}(\varrho^\prime(g^\prime)) = \Upsilon(\det g^\prime), g^\prime \in \mM^\prime$ 
\end{itemize}
%%%%%%%%%%%%%%%%%%%%%%%%%%%%%%%%%%%%%%%%%%%%%%%
%
%
%
%
%
%
%
%
%
%
%
%
%
%
%%%%%%%%%%%%%%%%%%%%%%%%%%%%%%%%%%%%%%%%%%%%%%%
\subsection{Other notation}
\begin{itemize}
\item We use the notation $a \ll_d  b$ to mean that $a \leq cb$ with $c > 0$ a constant depending on $d$.
%%%%
\item For any $g \in G$ define $\nu(g) \in \mR >0$ by $\nu(u \varrho(m) k)= \|\det m \|_E$ for any $u \in U$, $m \in \mM$, $k \in K$. 
Let $\nu^\prime(g)= \nu(\eta(g))$ for $g \in G^\prime$.
%%%%
\item $\mathcal{CSGR}(Q)$ is the set of compact open subgroups of a topological group $Q$.
%%%%
\item For an $\ell$-group $Q$ let $C(Q)$ (resp., $\mathcal{S}(Q)$) be the space of continuous (resp., Schwartz)
functions on $Q$ respectively.
%%%%
\item When $F$ is $p$-adic, if $Q^\prime$ is a closed subgroup of $Q$ and $\chi$ is a character of $Q^\prime$, we
denote by $C(Q^\prime \backslash Q, \chi)$ (resp., $C^{\rm sm}(Q^\prime \backslash Q, \chi), C_c^\infty(Q^\prime \backslash Q, \chi))$) 
the spaces of continuous (resp. $Q$-smooth, smooth and compactly supported modulo $Q^\prime$) complex-valued 
left $(Q^\prime, \chi)$-equivariant functions on $Q$.
%%%%
\item For an $\ell$-group $Q$ we write $\mathrm{Irr} Q$ for the set of equivalence classes of irreducible representations of $Q$. 
If $Q$ is reductive we also write $\mathrm{Irr}_{\rm sqr} Q$ and $\mathrm{Irr}_{\rm temp} Q$ for the subsets of irreducible unitary 
square-integrable (modulo center) and tempered representations respectively. We write $\mathrm{Irr}_{\rm gen} \mM$ and 
$\mathrm{Irr}_{\rm ut} \mM$ for the subset of irreducible generic representations of $\mM$ and representations of unitary type (see below), respectively. 
For the set of irreducible generic representations of $G$ we use the notation $\mathrm{Irr}_{{\rm gen}, \psi_{N^\prime}} G^\prime$ to emphasize the dependence on the character $\psi_{N^\prime}$.
%%%%
\item For $\pi \in \mathrm{Irr} Q$, let $\pi^{\vee}$ be the contragredient of $\pi$.
%%%%
\item For $\pi \in \mathrm{Irr}_{\rm gen} \mM$, $\mathbb{W}^{\psi_{N_\mM}}(\pi)$ denotes the (uniquely determined) Whittaker space of
$\pi$ with respect to the character $\psi_{N_\mM}$. Similarly we use the notation $\mathbb{W}^{\psi_{N_\mM}^{-1}}$, $\mathbb{W}^{\psi_{N_M}}$,
$\mathbb{W}^{\psi_{N_M}^{-1}}$, $\mathbb{W}^{\psi_{N^\prime}}$, $\mathbb{W}^{\psi_{N^\prime}^{-1}}$.
%%%%
\item For $\pi \in \mathrm{Irr}_{\rm gen} \mM$ let $\mathrm{Ind}(\mathbb{W}^{\psi_{N_M}}(\pi))$ be the space of smooth left $U$-invariant functions 
$W : G \rightarrow \mC$ such that for all $g \in G$, the function $ m \mapsto \delta_P(m)^{-\frac{1}{2}} W(mg)$ on $M$
belongs to $\mathbb{W}^{\psi_{N_M}}(\pi)$. Similarly define $\mathrm{Ind}(\mathbb{W}^{\psi_{N_M}^{-1}}(\pi))$
%%%%
\item
If a group $G_0$ acts on a vector space $W$ and $H_0$ is a subgroup of $G_0$, we denote by $W^{H_0}$ the subspace of $H_0$-fixed points.
%%%%
\item
We use the following bracket notation for iterated integrals: $\int \int (\int \int \dots) \dots$
implies that the inner integrals converge as a double integral and after evaluating them, the outer double integral is absolutely convergent.
%%%%
\end{itemize}
%%%%%%%%%%%%%%%%%%%%%%%%%%%%%%%%%%%%%%%%%%%%%%%
%
%
%
%
%
%
%
%
%
%
%
%
%
%
%%%%%%%%%%%%%%%%%%%%%%%%%%%%%%%%%%%%%%%%%%%%%%%
\subsection{Measures}
\label{measures}
The Lie algebra $\mathfrak{M}$ of $\mathrm{Res}_{E \slash F} \mathrm{GL}_m$ consists of the $m \times m$-matrices $X$ over $E$. 
Let $\mathfrak{M}_{\mathcal{O}}$ be the lattice
of integral matrices in $\mathfrak{M}$.
For any algebraic subgroup $\mathbf{Q}$ of $\mathrm{Res}_{E \slash F} \mathrm{GL}_m$ defined over $F$, let $\mathfrak{q} \subset \mathfrak{M}$
be the Lie algebra of $\mathbf{Q}$. 
The lattice $\mathfrak{q} \cap \mathfrak{M}_{\mathcal{O}}$ of $\mathfrak{q}$ 
gives rise to a gauge form of $\mathbf{Q}$ (determined up to multiplication by an element of $\mathcal{O}_E^\ast$) 
and we use it (together with $\psi$) to define a Haar measure on $Q$ by the recipe of Kneser~\cite{Kn}.
For example, when $n=1$, the measure on $N_{\mM} = \{ \left( \begin{smallmatrix} 1&x\\ &1\end{smallmatrix} \right) :x \in E \} \simeq E$
is the self-dual Haar measure $dx$ on $E$ with respect to $\psi$.
It is written as follows using measure on $F$.
Let $dx_i (i=1, 2)$ and $dy_i (i=1, 2)$  be self-dual Haar measures on $F$ with respect to $\psi_F$, i.e. the measure on $F$ such that 
\[
\int_{F} \int_{F} f(x_i) \psi_F(x_i y_i) \, dx_i \, dy_i = f(0)
\]
for a smooth function $f$ on $F$, provided the integral converges.
We define a Haar measure on $E$ by 
\begin{equation}
\label{measure E to FF}
dz_i = |\tau|^{\frac{1}{2}}dx_{i} dy_i
\quad
\text{with}
\quad
z_i = x_i + \tau y_i.
\end{equation}
Then for a smooth function $g$ on $E$, we have
\[
\int_{E} \int_{E}g(z_1) \psi(z_1 z_2) \, dz_1 \, dz_2 = g(0),
\]
provided the integral converges. Namely, $dz_i$ are self-dual Haar measures with respect to $\psi$.
Further, we note that for $f \in L^{1}(F)$, we have
\[
\int_{F} \int_{F} f(x) \psi(axy) \, dx \, dy
=|a|^{-1} \cdot f(0)
\]
and
\begin{equation}
\label{measure E/F}
\int_{F} \left( \int_{F \backslash E} f(x) \psi_{E}(\tau x y) \, dx \right) \, dy = |\tau|^{-\frac{1}{2}} f(0),
\quad
\int_{F} \left( \int_{F \backslash E} f(x) \psi_{E}(\tau^{-1} x y) \, dx \right) \, dy = |\tau|^{\frac{1}{2}} f(0).
\end{equation}
%%%%%%%%%%%%%%%%%%%%%%%%%%%%%%%%%%%%%%%%%%%%%%%
%
%
%
%
%
%
%
%
%
%
%
%
%
%
%%%%%%%%%%%%%%%%%%%%%%%%%%%%%%%%%%%%%%%%%%%%%%%
\subsection{Weil representation}
%Let $W$ be a symplectic space over $F$ with a symplectic form $\langle \cdot, \cdot \rangle$.
%Let $\mathcal{H} = \mathcal{H}_W$ be the Heisenberg group of $(W, \langle \cdot, \cdot \rangle)$.
%Recall that $\mathcal{H}_W = W \oplus F$ with the product rule
%\[
%(x, t) \cdot (y, s) = (x+y, t+s+ \frac{1}{2} \langle x, y \rangle).
%\]
%Fix a polarization $W = W_+ \oplus W_-$.
%The group $\mathrm{Sp}(W)$ acts on the right on $W$.
%We write a typical element of $\mathrm{Sp}(W)$ as $\left( \begin{smallmatrix} A&B\\ C&D\end{smallmatrix}\right)$
%where $A \in \mathrm{Hom}(W_+, W_+), B \in \mathrm{Hom}(W_+, W_-), C \in \mathrm{Hom}(W_-, W_+)$
%and $D \in \mathrm{Hom}(W_-, W_-)$.
%Let $\widetilde{Sp}(W)$ be the metaplectic two-fold cover of $\mathrm{Sp}(W)$ with respect to the Rao cocycle 
%determined by the splitting.
%Consider the Weil representation $\omega_{\psi}$ of the group $\mathcal{H}_W \rtimes \widetilde{Sp}(W)$
%on $\mathcal{S}(W_+)$.
%Explicitly, for any $\Phi \in \mathcal{S}(W_+)$ and $X \in W_+$, we have (see \cite[(1.5)]{GRS})
%\begin{align}
%\omega_{\psi}^{\Upsilon}(a, 0)\Phi(X) &= \Phi(X+a), \quad a \in W_+,\\
%\omega_{\psi}^{\Upsilon}(b, 0)\Phi(X) &= \psi(\langle \xi, b) \Phi(X), \quad b \in W_-,\\
%\omega_{\psi}^{\Upsilon}(0, t)\Phi(X) &= \psi \left(\frac{1}{2}t \right) \Phi(X), \quad t \in F,\\
%\omega_{\psi}^{\Upsilon} \left( \begin{pmatrix} I&y\\ &I \end{pmatrix}, \varepsilon \right)\Phi(X) &= \varepsilon \psi\Phi(X), \quad a \in W_+,
%\end{align}
%
Let $Y$ be a $2n$-dimensional space over $E$, equipped with a non-degenerate, anti-Hermitian form $(\cdot, \cdot)$.
Assume that it has a maximal isotropic subspace of dimension $n$ over $E$.
We denote its isometry group by $\mathrm{U}(Y) \simeq \mathrm{U}_{2n}$.
We may view $Y$ as symplectic space over $F$ with the symplectic form $\langle \cdot , \cdot \rangle = \mathrm{Tr}_{E \slash L} (\cdot, \cdot)$.
Then it is $4n$ dimensional vector space over $F$, and we denote this space by $Y^\prime$.
Let $\widetilde{\mathrm{Sp}}(Y^\prime)$ be the metaplectic cover of $\mathrm{Sp}(Y^\prime)$ with respect to the cocycle 
given in \cite[p.59]{MVW} (see also \cite[p.455]{GeRo}).
It is clear that $\mathrm{U}(Y) \subset \widetilde{\mathrm{Sp}}(Y^\prime)$, and we know that $\widetilde{\mathrm{Sp}}(Y^\prime)$
splits over $\mathrm{U}(Y)$. Fix a character $\Upsilon$ of $E^\times$ such that $\Upsilon |_{F^\times} = \eta_{E \slash F}$, and 
we choose the splitting as in \cite[p.9]{GRS} corresponding to $\Upsilon$.
Let us write 
\[
Y = Y^+ + Y^-
\]
as a direct sum of two maximal isotropic subspaces of $Y$, which are duality under $(\cdot, \cdot)$.
When we consider $Y^{\pm}$ as subspace of $Y^\prime$, they are isotropic for $\langle \cdot, \cdot \rangle$
,and thus in duality under $\langle \cdot, \cdot \rangle$.
Consider the Weil representation $\omega_{\psi}^{\Upsilon}$ of the group $\mathcal{H}_Y \rtimes \widetilde{\mathrm{Sp}}(Y^\prime)$ on $\Phi \in \mathcal{S}(E^n)$,
where $\mathcal{H}_Y := Y \oplus F$ is the Heisenberg group attached to $Y$ with the multiplication
\[
(w_1, t_1) \cdot (w_2, t_2) = (w_1 +w_2, t_1 +t_2 + \frac{1}{2} \mathrm{Tr}_{E \slash L} ( w_1, w_2)).
\]
Explicitly, for any $\Phi \in \mathcal{S}(Y^+)$ and $X \in Y^+$, the action of $\mathcal{H}_Y$ is given by
\begin{align}
\label{2.1a} \omega_{\psi}^{\Upsilon}(a, 0)\Phi(X) &= \Phi(X+a), \quad a \in Y^+,\\
\label{2.1b} \omega_{\psi}^{\Upsilon}(b, 0)\Phi(X) &= \psi \left( \langle X, b \rangle) \right) \Phi(X), \quad b \in Y^-,\\
\label{2.1c} \omega_{\psi}^{\Upsilon}(0, t)\Phi(X) &= \psi \left(t \right) \Phi(X), \quad t \in F
\end{align}
while the action of $ \widetilde{\mathrm{Sp}}(Y^\prime)$ is (partially) given by 
\begin{align}
\omega_{\psi}^{\Upsilon} \left( \begin{pmatrix} g& \\ &g^\ast \end{pmatrix}, \varepsilon \right)\Phi(X) &= \varepsilon
\gamma_{\psi} \left( N_{E \slash L}(\det g)\right) \Upsilon(\det g) |\det g|^{1 \slash 2} \Phi(X \cdot g)\\
\omega_{\psi}^{\Upsilon} \left( \begin{pmatrix} I&y\\ &I \end{pmatrix}, \varepsilon \right)\Phi(X) &= \varepsilon
\psi \left(\frac{1}{2} \langle X, X \cdot y \rangle) \right)\Phi(X)
\end{align}
We now take $Y = E^{2n}$ with the skew-hermitian form 
\[
\langle (x_1, \dots, x_{2n}), (y_{1}, \dots, y_{2n}) \rangle = \sum_{i=1}^{n} x_i \mc (y_{2n+1-i})- \sum_{i=1}^n y_i \mc(x_{2n+1-i})
\]
and the standard polarization $Y^{+} = \left\{(y_1, \dots, y_n, 0 \dots, 0) \right\}$ and 
$Y^{-} = \left\{(0 \dots, 0, y_1, \dots, y_n) \right\}$ with the standard basis $e_1, \dots, e_{n}, e_{-n}, \dots, e_{-1}$.
For $\Phi \in \mathcal{S}(E^{n})$, we define the Fourier transform
\[
\hat{\Phi}(X) = \int_{E^n} \Phi(X^\prime) \psi \left( \langle X, X^\prime \left( \begin{smallmatrix}&I_n\\ -I_n& \end{smallmatrix} \right) \rangle \right)
\]
Then, reapplied on $\mathcal{S}(E^n)$, the Weil representation satisfies
\begin{align}
\label{2.3a} \omega_{\psi}^{\Upsilon} \left( \varrho^\prime(g) \right)\Phi(X) &= 
\alpha_{\psi}(\varrho^\prime(g)) \Upsilon_{M^\prime}(\varrho^\prime(g)) \nu^{\prime}(\varrho^\prime(g))^{1 \slash 2} \Phi(X \cdot g)\\
\label{2.3b} \omega_{\psi}^{\Upsilon} \left( \begin{pmatrix} I&y\\ &I \end{pmatrix} \right)\Phi(X) &= 
\psi \left(\frac{1}{2} \langle X, X \cdot y \rangle) \right)\Phi(X)\\
\label{2.3c} \omega_{\psi}^{\Upsilon} \left(w_{U^\prime}^\prime \right)\Phi(X) &= \alpha_{\psi}(w_{U^\prime}^\prime) \hat{\Phi}(X)
\end{align}
where $\alpha_{\psi}(\cdot)$ is a certain eight root of unity, which satisfies $\alpha_{\psi}(\cdot)^{-1} = \alpha_{\psi^{-1}}(\cdot)$

Let $V_0 \subset V$ be the unipotent radical of the standard parabolic subgroup of $G$ with Levi 
$\mathrm{GL}_1(E)^{n-1} \times \mathrm{U}_{2n+2}$. The the map 
\[
v \mapsto v_{\mathcal{H}} := \left( (v_{n, n+j})_{j=1, \dots, 2n}, \frac{1}{2} \mathrm{Tr}_{E \slash L}(v_{n, 3n+1}) \right)
\]
gives an isomorphism from $V \slash V_0$ to a Heisenberg group $\mathcal{H}_Y$.
Then we may regard $\omega_{\psi}^{\Upsilon}$ as a representation of $\widetilde{\mathrm{Sp}(Y)} \ltimes V \slash V_0$.
Further, we extend $\omega_{\psi}^\Upsilon$ to a representation  $\omega_{\psi_{N_\mathbb{M}}}^\Upsilon$ of 
$V \rtimes G^\prime$ by setting
\[
\omega_{\psi}^\Upsilon (v g) \Phi = \psi_{V}(v) \omega_{\psi}^\Upsilon (v_{\mathcal{H}}) 
(\omega_{\psi}^\Upsilon(g) \Phi), \quad v \in V, g \in G^\prime.
\]
\subsection{Stable integral}
Many integral studied in this paper do not absolutely converge.
Instead of such integrals, we shall use the notion of stable integrals as in \cite{LMb}.
For the convenience to the reader, we shall recall the notion of stable integral.
Suppose that $F$ is $p$-adic.
Let $U_0$ be a unipotent subgroup over $F$ with a fixed Haar measure on $U_0$.
\begin{Definition}[Definition~2.1 in \cite{LMb}]
Let $f$ be a smooth function on $U_0$. 
We say that $f$ has a stable integral over $U_0$ if there exists $U_1 \in \mathcal{CSGR}(U_0)$ such that for any 
$U_2 \in \mathcal{CSGR}(U_0)$ containing $U_1$ we have
\begin{equation} 
\label{def stable eq}
\int_{U_2} f(u) \, du = \int_{U_1} f(u) \, du. 
\end{equation}
In this case we write $\int_{U_0}^{\rm st} f(u) \, du$ for the common value \eqref{def stable eq} and say that  
$\int_{U_0}^{\rm st} f(u) \, du$ stabilizes at $U_1$. In other words, $\int_{U_0}^{\rm st} f(u) \, du$ is the limit of the net 
$(\int_{U_1} f(u) \, du)_{U_1 \in \mathcal{CSGR}(U_0)}$ with respect to the discrete topology of $\mC$.
\end{Definition}
We also define uniformally stable integral (see \cite[pp.10]{LMb})
\begin{Definition}
Given a family of functions $f_x \in C^{\rm sm}(U_0)$, we say that the integral $\int_{U_0}^{\rm st} f_x(u) \,du$ stabilizes
uniformly in x if $U_1$ as above can be chosen independently of $x$. 
Similarly, if $x$ ranges over a topological space $X$ then we say that $\int_{U_0}^{\rm st} f_x(u) \,du$ stabilizes locally uniformly in $x$ 
if any $y \in X$ admits a neighborhood on which $\int_{U_0}^{\rm st} f_x(u) \,du$ stabilizes uniformly.
\end{Definition}
\section{Statement of main result}
\subsection{Local Fourier-Jacobi transform}
For any $f  \in C(G)$ and $s \in \mC$, define $f_s(g) = f(g)\nu(g)^s$, $g \in G$. 
Let $\pi \in \mathrm{Irr}_{\rm gen} \mM$ with Whittaker model $W^{\psi_{N_\mM}}(\pi)$.
Let $\mathrm{Ind}(W^{\psi_{N_\mM}}(\pi))$ be the space of $G$-smooth left $U$-invariant functions 
$W : G \rightarrow \mC$ such that for all $g \in G$, the function $\delta_{P}^{-\frac{1}{2}}W(mg)$ on $M$ 
belongs to $W^{\psi_{N_\mM}}(\pi)$. 
For any $s \in \mC$ we have a
representation $\mathrm{Ind}(W^{\psi_{N_\mM}}(\pi), s)$ on the space $\mathrm{Ind}(W^{\psi_{N_\mM}}(\pi))$ 
given by $(I(s,g)W)_s(x) = W_s(xg)$, $x, g \in G$. 
It is equivalent to the induced representation of $\pi \otimes \nu^s$ 
from $P$ to $G$. 
The family $W_s$, $s \in \mC$ is a holomorphic section of this family of induced representations.

For any $W \in C^{\rm sm}(N \backslash B, \psi_N)$ and $\Phi \in \mathcal{S}(E^n)$, we define a  function on $G^\prime$ by
\[
A^{\psi, \Upsilon}(W, \Phi, g, s) := \int_{V_\gamma \backslash V} W_s(\gamma v g) \omega_{\psi_{N_\mathbb{M}}^{-1}}^{\Upsilon^{-1}}(vg)
\Phi(\xi_n) \, dv, \quad g \in G^\prime.
\]
where with $\xi_n, \gamma, V$ and $V_\gamma$ are defined in \ref{notation 1}.
Its property was studied by \cite[Lemma~5.1]{LMa}.
In particular, it satisfies 
\begin{equation}
\label{equiv A}
A^{\psi, \Upsilon} \left(I(s, vx)W, \omega_{\psi_{N_{\mathbb{M}}}^{-1}}^{\Upsilon^{-1}}(vx) \Phi, g, s \right)
= A^{\psi, \Upsilon} \left(W, \Phi, gx, s \right)
\end{equation}
for any $g, x \in G^\prime$ and $v \in V$ (see Lapid--Mao~\cite[(5.2)]{LMa}).
Further, the integrand of the definition is compactly supported, and $A^{\psi, \Upsilon}$ gives rise 
to a $V \ltimes G$-intertwining map
\[
A^{\psi, \Upsilon} : C^{sm}(N \backslash G, \psi_N) \otimes \mathcal{S}(E^n)
\rightarrow C^{sm}(N^\prime \backslash G_\circ, \psi_{N_\circ})
\]
where $V \ltimes G_\circ$ acts via $V \ltimes \eta(G_\circ)$ by right translation on $C^{sm}(N^\prime \backslash G_\circ, \psi_{N^\prime})$.

In order to simplify the notation, we introduce the map $A_{\#}^{\psi, \Upsilon}$ as follows.
Let $V_+ \subset V$ be the image under $\ell_M$ of the space of $\mn \times \mn$-matrices whose 
rows are zero except possibly for the last one.
For $c = \ell_M(x) \in V_+$ we denote by $\underline{c} \in E^n$ the last row of $x$.
Then we have
\[
A^{\psi, \Upsilon}(W(\cdot c), \Phi(\cdot +\underline{s}), g) = A^{\psi, \Upsilon}(W, \Phi, g), \quad c \in V_+, g \in G^\prime.
\]
This implies that $A^{\psi, \Upsilon}(W, \Phi, \cdot)$ factors through $W \otimes \Phi \mapsto \Phi \ast W$
where for any function $f \in C^\infty(G)$ we set 
\[
\Phi \ast f (g) = \int_{V_+} f(gc) \Phi(\underline{c}) \, dc.
\]
Then we define the map 
\[
A_{\#}^{\psi, \Upsilon} : C^{\rm sm}(N \backslash G, \psi_N) \rightarrow C^{\rm sm}(N^\prime \backslash G^\prime, \psi_N).
\]
by 
\begin{equation}
\label{A sharp}
A_{\#}^{\psi, \Upsilon}(W, \cdot) = A^{\psi, \Upsilon}(W, \Phi, \cdot)
\end{equation}
for any $\Phi \in \mathcal{S}(E^n)$ such that $\Phi \ast W = W$.
Then $A_{\#}^{\psi, \Upsilon}$ does not have the invariance by $G^\prime$ and $V$,
but it has the following equivariance property.
Set 
\[
\psi_{V_-}(vu) = \psi_{N_M}(v)^{-1} \psi_{U}(u)
\]
Recall that $\psi_{U}(\ell(x)) = \psi(\frac{1}{2}(x_{n, n+1}-x_{2n, 1}))$
\begin{lemma}[cf. Lemma~3.1 in \cite{LMb}]
\label{Lemma 3.1}
For any $v \in V_-$, $p=\varrho^\prime(g) u \in P^\prime$ where $g \in \mM^\prime$ and $u \in U^\prime$,
we have
\[
A_{\#}^{\psi, \Upsilon}(W(\cdot v \eta(p)), g) = \nu^\prime(m)^{\frac{1}{2}} \Upsilon_{M^\prime}(m) \alpha_{\psi^{-1}}(m)^{-1} \psi_{V_-}(v)
A_{\#}^{\psi, \Upsilon}(W, gp)
\]
\end{lemma}
\begin{proof}
This is proved in the same argument as the proof of \cite[Lemma~3.1]{LMb}.
Indeed, for $\Phi \in \mathcal{S}(E^n)$ with sufficiently small support and $\int_{E^n} \Phi(x) \, dx =1$, we can show that 
by \eqref{2.1b}, \eqref{2.1c}, \eqref{2.3a}, \eqref{2.3b},
\[
\Phi^\prime \ast W (\cdot vp) = \psi_{V_-}(v)^{-1} \alpha_{\psi^{-1}}(m) \nu^\prime(m)^{-\frac{1}{2}} \Upsilon_{M^\prime}(m)^{-1} W(\cdot vp).
\]
where $\Phi^\prime = \omega_{\psi^{-1}}^{\Upsilon^{-1}}(vp) \Phi$. 
This identity shows that 
\[
A^{\psi, \Upsilon}(W(\cdot v \eta(p)), \Phi^\prime, g) = 
 \nu^\prime(m)^{-\frac{1}{2}} \Upsilon_{M^\prime}(m)^{-1} \alpha_{\psi^{-1}}(m) \psi_{V_-}(v)^{-1}
A_{\#}^{\psi, \Upsilon}(W(\cdot v \eta(p)), g).
\]
Further, from the definition, we have 
\[
A^{\psi, \Upsilon}(W(\cdot v \eta(p)), \Phi^\prime, g) = A^{\psi, \Upsilon} (W, \Phi, gp)
\]
and the lemma follows from these two identities.
%Note that 
%\[
%\psi \left(\frac{1}{4} \mathrm{Tr}_{E \slash L}(t) \right) = \psi \left(\frac{1}{2} t \right), \quad t \in E.
%\]
\end{proof}
As in \cite[Lemma~3.2]{LMb}, we record the following result, 
which is proved as in the proof of \cite[Lemma~4.5]{LMc} (See also \cite[Lemma~5.1]{LMa}).
\begin{lemma}
\label{Lemma 3.2}
For any $K_0 \in \mathcal{CSGR}(G)$, there exists $\Omega \in \mathcal{CSGR}(V_U)$
such that for any $W \in C(N \backslash G, \psi_N)^{K_0}$ the support of $W(\gamma \cdot)_{V_-}$
is contained in $V_\gamma \eta(w_{U^\prime}^\prime) \Omega \eta(w_{U^\prime}^\prime)^{-1}$.
\end{lemma}
%%%%%%%%%%%%%%%%%%%%%%%%
%
%
%
%
%
%
%
%
%
%
%
%%%%%%%%%%%%%%%%%%%%%%%%%
\subsection{Explicit local descent}
Define the intertwining operator 
\[
M(\pi, s) = M(s) : \mathrm{Ind} \left( \mathbb{W}^{\psi_{N_M}}(\pi), s\right) \rightarrow 
\mathrm{Ind} \left( \mathbb{W}^{\psi_{N_M}}(\pi^\vee), -s \right) 
\]
by (the analytic continuation of)
\[
M(s)W(g) = \nu(g)^s \int_{U} W_s(\varrho(\mathfrak{t}) w_U u g) \, du
\]
where $\mathfrak{t} = \mathrm{diag}(1, -1, \dots, 1, -1)$ is introduced in order to preserve the character $\psi_{N_M}$.
\[
w_U = \begin{pmatrix} &1_{2n}\\ -1_{2n}&\end{pmatrix}.
\]
Then we know that $M(s)$ is holomorphic at $s = \frac{1}{2}$ when $\pi \in \mathrm{Irr}_{\rm gen} \mM$ is of unitary type
by Lapid--Mao~\cite[Proposition~2.1]{LMa}.
By abuse of notation, we will also denote by $M(\pi, s)$ the intertwining operator 
$\mathrm{Ind} \left( \mathbb{W}^{\psi_{N_M}^{-1}}(\pi), s\right) \rightarrow 
\mathrm{Ind} \left( \mathbb{W}^{\psi_{N_M}^{-1}}(\pi^\vee), -s \right)$ defined in the same way.

For simplicity we denote $M_s^\ast W := (M(s)W)_{-s}$ so that 
\[
M_s^\ast W = \int_{U} W_s(\varrho(\mathfrak{t}) w_U u \cdot ) \, du
\]
for $\mathrm{Re}(s) \gg_{\pi} 1$. Set $M^\ast W := M_{\frac{1}{2}}^\ast W$.
%%%%%%%%%%%%
%
%
%
%%%%%%%%%%%%
\begin{definition}[explicit local descent]
Suppose that $\pi \in \mathrm{Irr}_{\rm gen} \mM$ is of unitary type.
The we denote by $\mathcal{D}_\psi^\Upsilon(\pi)$ the space of Whittaker function on $G^\prime$ generated by 
\[
A^{\psi, \Upsilon}\left(M \left( \frac{1}{2}\right)W, \Phi, \cdot,-\frac{1}{2} \right), \quad W \in \mathrm{Ind} (\mathbb{W}^{\psi_{N_M}}(\pi)), \Phi \in \mathcal{S}(E^n).
\]
\end{definition}
We note that from \cite[Proposition~2.1]{LMa} and the remark below \cite[Lemma~5.1]{LMa},
 $\mathcal{D}_\psi^\Upsilon(\pi)$ is not zero.
%%%%%%%%%%%%%%%%%%%%%%%%
%
%
%
%
%
%
%
%
%
%
%
%%%%%%%%%%%%%%%%%%%%%%%%%
\subsection{Good representations}
Let $\pi \in \mathrm{Irr}_{\rm gen} \mM$ and $\sigma \in \mathrm{Irr}_{\psi_{gen, N^\prime}^{-1}} G^\prime$
with the Whittaker model $\mathbb{W}^{\psi_{N^\prime}^{-1}} (\sigma)$.
For any $W^\prime \in \mathbb{W}^{\psi_{N^\prime}^{-1}} (\sigma)$ and $W \in \mathrm{Ind} (\mathbb{W}^{\psi_{N_M}} (\pi))$, 
we define
\[
J(W^\prime, W, s) := \int_{N^\prime \backslash G^\prime} W^\prime(g^\prime)  A_{\#}^{\psi, \Upsilon}(W_s, g^\prime)\, dg^\prime.
\]
By \cite[Proposition~3.1]{BAS}, $J$ converges for $\mathrm{Re}(s) \gg_{\pi, \sigma} 1$ and admits 
a meromorphic continuation in $s$. Moreover, by \cite[Proposition~4.1]{BAS}, for any $s \in \mC$, 
we can choose $W$ and $W^\prime$ such that $J(W^\prime, W, s) \ne 0$.
%\begin{equation}
%\label{equiv J}
%J \left( \sigma^\prime(x)W^\prime, I(s, vx)W, \omega_{\psi_{N_{\mathbb{M}}}^{-1}}^{\Upsilon^{-1}}(vx) \Phi, s \right)
%=J \left(W^\prime, W, \Phi, s \right)
%\end{equation}
%for any $g, x \in G^\prime$ and $v \in V$ (see \cite[p.10]{LMa}) if it converges.
%Here, we write $\sigma^\prime = \mathcal{D}^{\Upsilon^{-1}}_{\psi^{-1}}(\mc(\pi))$.
%%%%%%%%%%%%
%
%
%
%%%%%%%%%%%%
\begin{definition}
\label{def inner sigma}
Let $\pi \in \mathrm{Irr}_{\rm gen, ut} \mM$. We say $\pi$ is \textit{good} if it satisfies the following three assumptions.
\begin{enumerate}
\item $\mathcal{D}_\psi^\Upsilon(\pi)$ and $\mathcal{D}_{\psi^{-1}}^{\Upsilon^{-1}}(\mc(\pi))$ are irreducible.
\item $J(W^\prime, W, s)$ is holomorphic at $s = \frac{1}{2}$ for any $W^\prime \in \mathcal{D}_{\psi^{-1}}^{\Upsilon^{-1}}(\mc(\pi))$ 
and $W \in \mathrm{Ind}(\mathbb{W}^{\psi_{N_M}}(\pi) )$.
\item There exists a non-degenerate $G$-invariant bilinear form $[ \cdot, \cdot ]_\sigma$ 
on $\mathcal{D}_\psi^\Upsilon(\pi) \times \mathcal{D}_{\psi^{-1}}^{\Upsilon^{-1}}(\mc(\pi))$ such that
\[
J \left(W^\prime, W, \frac{1}{2} \right) = \left[W^\prime, A_{\#}^{\psi, \Upsilon}(M^\ast W, \cdot ) \right]_\sigma
\]
\end{enumerate}
\end{definition}
Good representations appear as a local component of certain automorphic representations as follows.
Let $F$ be a number field and $E$ a quadratic extension of $F$.
For an irreducible cuspidal automorphic representation $\pi$ of $\mathrm{GL}_{m}(\mA_E)$, whose central character is trivial on $\mA_F^\times$,
we say that $\pi$ is of unitary type if 
\[
\int_{\mA_F^\times \mathrm{GL}_{m}(F) \backslash \mathrm{GL}_{m}(\mA_F)} \varphi(h) \, dh \ne 0
\]
for some $\varphi \in \pi$.
We write $\mathrm{Cusp}_{\rm uni} \mathrm{GL}_{m}$ for the set of irreducible cuspidal 
automoprhic representations of unitary type.
Consider the set $\mathrm{UCusp} \mM$ of isobaric automorphic representations $\pi_1 \boxplus \cdots \boxplus \pi_k$
of pairwise inequivalent $\pi_i \in \mathrm{Cusp}_{\rm uni} \mathrm{GL}_{n_i}$ for $1 \leq i \leq k$ and $n_1 + \cdots n_k = 2n$.
Then we have the following crucial fact.
\begin{theorem}[Theorem~5.5 in \cite{LMa}]
If $\pi \in \mathrm{UCusp} \mM$, then its local component $\pi_v$ are good.
\end{theorem}
\begin{Remark}
In \cite[Theorem~5.5]{LMa}, they assume two working assumptions 4.2 and 5.2.
The working assumption 5.2 is proved by Ben-Artzi and Soudry \cite{BAS} and 
the working assumption 4.2 is proved by the author \cite{Mo1}.
\end{Remark}
The definition of good representations was introduced and discussed in \cite[5.3]{LMa}.
In particular, if $\pi$ is good and $\sigma^\prime = \mathcal{D}^{\Upsilon^{-1}}_{\psi^{-1}}(\pi)$,
then there exists a constant $c_\pi$ such that for any $W^\prime$
\begin{equation}
\label{ILM 3.7}
\int_{N^\prime}^{st} [\sigma(n)W^\prime, (W^\prime)^{\vee}]_\sigma \psi_{N^\prime}(n) \, dn
= c_\pi W^\prime(e) (W^\prime)^\vee(e).
\end{equation}
More explicitly, for any $W \in \mathrm{Ind}(\mathbb{W}^{\psi_{N_M}}(\pi))$
and $W^{\wedge} \in \mathrm{Ind}(\mathbb{W}^{\psi_{N_M}^{-1}}(\pi)$ we have
\[
\int_{N^\prime}^{st} J(A_{\#}^{\psi^{-1}, \Upsilon^{-1}}(M^\ast W^\wedge , \cdot n), W, \frac{1}{2}) \psi_{N^\prime}(n) \, dn
=c_\pi A_{\#}^{\psi^{-1}, \Upsilon^{-1}}(M^\ast W^\wedge, e)   A_{\#}^{\psi, \Upsilon}(M^\ast W, e).  
\]
In the rest of the paper we will prove the following statement:
\begin{theorem}
\label{main theorem}
For any unitarizable $\pi \in \mathrm{Irr}_{\rm gen, ut} \mM$ which is good we have $c_\pi = \omega_{\pi}(\tau)$.
\end{theorem}
\begin{Remark}
As remarked in Introduction, from the definition, we have $\omega_\pi(\tau) = \omega_\sigma(-1)$.
\end{Remark}
The following special case for our application to formal degree conjecture.
\begin{corollary}
\label{Cor3.6}
Let $\pi$ be an irreducible representation of $\mathrm{GL}_{2n}(E)$ of the form 
$\pi = \tau_1 \times \cdots \times \tau_k$ where $\tau_i$ are mutually inequivalent irreducible discrete series representations
of $\mathrm{GL}_{n_i}(E)$ such that $n=n_1+\cdots + n_k$ and $L(s, \tau_i, \mathrm{As})$ has a pole at $s=0$.
Then $\pi$ is good and $\sigma = \mathcal{D}_{\psi^{-1}}^{\Upsilon^{-1}}(\mc(\pi))$
is discrete series. Moreover, 
\[
\int_{N^\prime}^{st} J(W, W^\prime, \frac{1}{2}) \psi_{N^\prime}(n) \, dn
=\omega_\sigma(-1) W(e)   A_{\#}^{\psi, \Upsilon}(M^\ast W^\prime, e).  
\]
for any $W \in \mathbb{W}^{\psi_{N^\prime}^{-1}}(\sigma)$ and $W^\prime \in \mathrm{Ind}(\mathbb{W}^{\psi_{N_M}}(\pi))$.
\end{corollary}
\begin{proof}
By Theorem~\ref{descent BC local}, we find that $\sigma$ is discrete series representation.
Moreover, by \cite[Corollary~A.6]{ILM} (see also proof of Theorem~\ref{descent BC local}), we see that $\pi$ is good.
Then our assertion follows from Theorem~\ref{main theorem}.
\end{proof}
%%%%%%%%%%%%%%%%%%%%%%%%%%%%%%%%%%%%%%%%%%%%%%%
%
%
%
%
%
%
%
%
%
%
%
%
%
%
%%%%%%%%%%%%%%%%%%%%%%%%%%%%%%%%%%%%%%%%%%%%%%%
%%%%%%%%%%%%%%%%%%%%%%%%%%%%%%%%%%%%%%%%%%%%%%%
%
%
%
%
%
%
%
%
%
%
%
%
%
%
%%%%%%%%%%%%%%%%%%%%%%%%%%%%%%%%%%%%%%%%%%%%%%%
\subsection{A reduction to tempered case}
In this section, we shall reduce Theorem~\ref{main thm}  to the case of tempered good representations.
For our reduction, we shall need lemma concerning analytic properties of local integrals.
In the same argument as the proof of \cite[Lemma~4.6]{LMc},
we can prove the following lemma using  Soudry--Tanay~\cite[Proposition~2.1]{ST14}
\begin{lemma}
\label{lemma4.6}
Suppose that $\pi$ is an irreducible generic admissible representation of $\mathbb{M}$ and $\alpha \in \mR$ with $\alpha > e(\pi)$.
Then for any $W \in \mathrm{Ind} (\mathbb{W}^{\psi_{N_M}} (\pi))$ and $\Phi \in \mathcal{S}(E^n)$,
there exists $c >0$ such that for any $t^\prime = \varrho^\prime(t) \in T^\prime$ with $t= (t_1, \dots, t_n) \in \mathbb{M}^\prime$,
$k \in K^\prime$ and  $s \in \mC$, we have $A^{\psi, \Upsilon}(W, \Phi, tk^\prime, s) = 0$
unless $|t_i| \leq c |t_{i+1}|$ for $i=1, \dots, n$ with $t_{n+1}=1$ in which case
\[
|A^{\psi, \Upsilon}(W, \Phi, \varrho^\prime(t)k^\prime, s)| \ll \delta_{B^\prime}(t^\prime)^{\frac{1}{2}} |\det (t)|^{\mathrm{Re}(s) + \frac{1}{2}-\alpha}.
\]
\end{lemma}
%\begin{Remark}
%In \cite[Lemma~4.6]{LMc}, they state a similar result for irreducible representations.
%It is easy to check that  \cite[Proposition~2.1]{ST14} holds even if $\pi$ is reducible,
%and the same result holds for reducible representations as above.
%\end{Remark}
Similarly as the proof of \cite[Lemma~4.12]{LMc}, the following lemma readily follows from the above lemma.
\begin{lemma}
\label{j conv}
Suppose that $\sigma \in \mathrm{Irr}_{\psi_{N^\prime}^{-1}, gen} G^\prime$ is essentially discrete series 
representation and $\alpha = e(\pi)$. Then there exists $\delta > 0$ such that 
$J(W^\prime, W, \Phi, s)$ is absolutely convergent locally uniformly (hence, holomorphic) 
for $\mathrm{Re}(s) \geq \alpha - \delta - \frac{1}{2}$
where  $W^\prime \in \mathbb{W}^{\psi_{N^\prime}^{-1}} (\sigma)$, $W \in \mathrm{Ind} (\mathbb{W}^{\psi_{N_M}} (\pi))$
and $\Phi \in \mathcal{S}(E^n)$.
Similarly, for any $W \in \mathrm{Ind} (\mathbb{W}^{\psi_{N_M}} (\pi))$, $W^\vee \in \mathrm{Ind} (\mathbb{W}^{\psi_{N_M}^{-1}} (\pi))$,
$\Phi, \Phi^\vee \in \mathcal{S}(E^n)$,
\[
J\left( A^{\psi^{-1}, \Upsilon^{-1}} \left(M \left(\frac{1}{2} \right)\mc(W^\vee), \Phi^\vee, \cdot, -\frac{1}{2} \right), W, \Phi, s \right)
\]
converges absolutely and locally uniformly (and hence, is holomorphic) for $\mathrm{Re} s > 2 \alpha - \frac{1}{2}$.
In particular, if $\pi$ is unitarizable then $J(W^\prime, W, \Phi, s)$ is holomorphic at $s = \frac{1}{2}$ for any
$W^\prime \in \mathcal{D}_{\psi^{-1}}(\pi)$, $W \in \mathrm{Ind} (\mathbb{W}^{\psi_{N_M}} (\pi))$ and $\Phi \in \mathcal{S}(E^n)$.
\end{lemma}
When the base field $F$ is real field, we may prove similar results in the same argument as the proof of \cite[Lemma~4.11]{LMa}.
For later use, we would like to state them here.
\begin{lemma}
\label{4.11 arch}
Suppose that $\pi$ is an irreducible generic admissible representation of $\mathbb{M}$ and $\alpha \in \mR$ with $\alpha > e(\pi)$.
Then for any $W \in \mathrm{Ind} (\mathbb{W}^{\psi_{N_M}} (\pi))$ and $\Phi \in \mathcal{S}(E^n)$,
there exists a compact set $D \subset \mC$ and $m \geq 1$ we have
 \[
|A^{\psi, \Upsilon}(W, \Phi, \varrho^\prime(t)k^\prime, s)| \ll \delta_{B^\prime}(t^\prime)^{\frac{1}{2}} |\det (t)|^{\mathrm{Re}(s) + \frac{1}{2}-\alpha}
(1+|t_i t_{i+1}^{-1}|)^{-m}
\]
for any $t^\prime = \varrho^\prime(t) \in T^\prime$ with $t= (t_1, \dots, t_n) \in \mathbb{M}^\prime$,
$k \in K^\prime$ and  $s \in D$.
\end{lemma}
%\begin{Remark}
%In \cite[Lemma~4.6]{LMc}, they state a similar result for irreducible representations.
%It is easy to check that  \cite[Proposition~2.1]{ST14} holds even if $\pi$ is reducible,
%and the same result holds for reducible representations as above.
%\end{Remark}
\begin{Remark}
\label{remark arch}
In the same argument as the proof of \cite[Lemma~4.12]{LMa} with Lemma~\ref{4.11 arch}, Lemma~\ref{j conv} holds over real field.
\end{Remark}

Let us go back to our situation, namely we suppose that the base field is non-arhchimedean local field.
Then our aim in this section is to prove the following reduction.
\begin{proposition}
\label{first reduction prp}
Suppose that \eqref{main thm eq} holds for any good $\pi \in \mathrm{Irr}_{\rm temp, ut} \mM$ 
such that $\mathcal{D}_{\psi^{-1}}^{\Upsilon^{-1}}(\pi)$ is tempered.
Then Theorem~\ref{main theorem} holds.
\end{proposition}
%%%%%%%%%%%%%%%%%%%%%%%%%%%%%%%%%%%%%%%%%%%%%%%
%
%
%
%
%
%
%
%
%
%
%
%
%
%
%%%%%%%%%%%%%%%%%%%%%%%%%%%%%%%%%%%%%%%%%%%%%%%
For a proof, recall the following classification of generic representations of unitary type by Matringe.
\begin{theorem}[Theorem~5.2 in Matringe~\cite{Ma10}]
\label{mat}
The set $\mathrm{Irr}_{\rm gen, ut}\mM$ consists of the irreducible representations of the form
\[
\pi = \mc(\sigma_1) \times \sigma_1^\vee \times \cdots  \mc(\sigma_k) \times \sigma_k^\vee
\times \delta_1 \times \cdots \times \delta_l 
\]
where $\sigma_1, \dots, \sigma_k$ are essentially square-integrable representations,
and $\delta_1, \dots, \delta_l$ are square-integrable representations of unitary type
(i.e., $L(0, \tau_i, \mathrm{As}) = \infty$ for all $i$).
Here, $\times$ denotes the parabolic induction of $\mathrm{GL}_m$.
\end{theorem}
We shall prove the reduction following the argument in \cite[3.6]{LMa}.
\begin{proof}
Let $\pi \in \mathrm{Irr}_{\rm gen, ut} \mM$ be unitarizable and good.
Then by Theorem~\ref{mat}, we may write 
\[
\pi = \mc(\sigma_1)[a_1] \times \sigma_1^\vee[-a_1] \times \cdots \times \mc(\sigma_k)[a_k] \times \sigma_k^\vee[-a_k]
\times \delta_1 \times \cdots \times \delta_l
\]
where $\sigma_i \in \mathrm{Irr}_{\rm sqr} \mathrm{GL}_{n_i}(E)$, $\delta_i  \in \mathrm{Irr}_{\rm sqr} \mathrm{GL}_{m_j}(E)$ 
and $(a_1, \dots, a_k) \in \mC^k$ is in the domain
\[
\mathfrak{D} = \left\{(s_1, \dots, s_k) \in \mC^k : -\frac{1}{2} < \mathrm{Re} s_i < \frac{1}{2} \text{ for all $i$} \right\}.
\]
Here, we write $\sigma[s]$ for the twist of $\sigma \in \mathrm{Irr} \mathrm{GL}_m(E)$ by $|\det |^s$
with $s \in \mC$.
For $\underline{s} = (s_1, \dots, s_k) \in \mC^k$, let us consider the following representation
\[
\pi(\underline{s}) := \mc(\sigma_1)[s_1] \times \sigma_1^\vee[-s_1] \times \cdots \times \mc(\sigma_k)[s_k] \times \sigma_k^\vee[-s_k]
\times \delta_1 \times \cdots \times \delta_l
\]
which is irreducible for $\underline{s} \in \mathfrak{D}$ by Zelevinsky~\cite[Theorem~9.7]{Ze}.

Take a quadratic extension of number fields $K \slash L$ such that for some finite place $v_0$,
$L_{v_0} = F$ and $K_{v_0} = E$.
Let $\rho \in \mathrm{Irr}_{\rm sqr, gen} \mathrm{U}_{\mathbf{m}}$ be the representation 
corresponding to $\delta_1 \times \cdots \times \delta_l$ under the correspondence established in Theorem~\ref{backward lifting}
with $\mathbf{m} = m_1 + \cdots + m_l$.
Then we put
\[
\sigma(\underline{s}) :=  \mc(\sigma_1)[s_1] \times \sigma_1^\vee[-s_1] \times \cdots \times \mc(\sigma_k)[s_k] \times \sigma_k^\vee[-s_k] \rtimes \rho
\]
%\[
%\pi = \mc(\delta_1)[a_1] \times \delta_1^\vee[-a_1] \times \cdots \mc(\delta_k)[a_k] \times \delta_k^\vee[-a_k]
%\times \tau_1 \times \cdots \times \tau_l
%\]
%corresponding to $\tau_1 \times \cdots \times \tau_l$ under the correspondence in Theorem~\ref{backward lifting}.
%For $\underline{s} \in \mR^k$, we set 
%\[
%\sigma(\underline{s}) =  \mc(\delta_1)[s_1] \times \delta_1^\vee[-s_1] \times \cdots \mc(\delta_k)[s_k] \times \delta_k^\vee[-s_k] \rtimes \rho
%\]
By Konno~\cite[Corollary~4.3]{TK}, for a dense open subset of $\underline{s} \in i \mR^k$, $\sigma(\underline{s}) \in \mathrm{Irr}_{\rm temp} \mathrm{U}_{2n}$.
Further, by \cite[Corollary~A.8]{ILM}, for a dense open subset of $\underline{s} \in i \mR^k$,
$\sigma(\underline{s})$ is a local component of an irreducible cuspidal automorphic representation of $\mathrm{U}_{2n}(\mA_L)$.
From our proof of Theorem~\ref{backward lifting}, we find that 
its base change lift to $\mathrm{GL}_{2n}(\mA_K)$ by Kim--Krishnamurthy~\cite{KK} gives a globalization of $\pi(\underline{s})$,
and $\mathcal{D}_\psi^{\Upsilon^{-1}}(\pi(\underline{s})) = \sigma(\underline{s})$.
In particular, by \cite[Theorem~5.5]{LMa}, $\pi(\underline{s})$ is good for such $\underline{s}$.
Further, it is tempered since $\underline{s} \in i \mR^k$.
We note that $\omega_{\pi(\underline{s})}(\tau) = \omega_{\tau_1}(\tau) \cdots \omega_{\tau_l}(\tau) = \omega_{\pi}(\tau)$.
Hence, from our assumption, for $\underline{s} \in i\mR^k$ given in the above discussion,  we have
\begin{equation}
\label{constant ft}
c_{\pi(\underline{s})} = \omega_{\pi(\underline{s})}(\tau) = \omega_{\pi}(\tau)
\end{equation}

On the other hand, in the same argument as \cite[3.6]{LMe}, Lemma imply that both sides of 
are holomorphic functions of $\underline{s} \in \mathfrak{D}$.
This shows that $c_{\pi(\underline{s})}$ is a meromorphic function on $\underline{s} \in \mathfrak{D}$, and thus 
$c_{\pi(\underline{s})}$ is a constant function of $\underline{s}$ by \eqref{constant ft}.
In particular, we have
\[
c_{\pi} =  \omega_{\pi}(\tau).
\]
%Then $c_{\pi(\underline{s})}$ does not depend 
%On the other hand, as in the proof of \cite[3.6]{LMb}, by Lemma~\ref{j conv} and \cite[Proposition~2.11]{LMe}, we can show that $c_{\pi(\underline{s})}$ is a meromorphic function
%on the domain
%\[
%\mathfrak{D} = \left\{ (s_1, \dots, s_k) \in \mC^k : -\frac{1}{2} < \mathrm{Re}(s_i) < \frac{1}{2} \text{ for all $i$}\right\}.
%\]
%From our assumption, $c_{\pi(\underline{s})} = \omega_{\pi(\underline{s})}(\tau) = \omega_{\pi}(\tau)$ for a dense open subset of $i \mR^k$.
%Since both sides are meromorphic function on $\mathfrak{D}$,
%we should have  $c_{\pi(\underline{s})} = \omega_{\pi}(\tau)$  for any $\underline{s} \in \mathfrak{D}$.
%In particular, Theorem~\ref{main thm} holds for $\pi$.
\end{proof}
In the reminder of the paper we will prove Theorem~\ref{main theorem}, i.e., the Main Identity
\[
\tag{MI}
\int_{N^\prime}^{st} ( \int_{N^\prime \backslash G^\prime} A_{\#}^{\psi, \Upsilon}(W_s, g) A_{\#}^{\psi^{-1}, \Upsilon^{-1}}(M^\ast W^\wedge, gu) \, dg)
\psi_{N^\prime}(u) \, du |_{s= \frac{1}{2}}
= \omega_{\pi}(\tau) A_{\#}^{\psi^{-1}, \Upsilon^{-1}}(M^\ast W^\wedge, e)A_{\#}^{\psi, \Upsilon}(M^\ast W, e)
\]
under the assumptions that $\pi \in \mathrm{Irr}_{\rm temp, ut} \mM$ is good and 
$\sigma:= \mathcal{D}_{\psi^{-1}}^{\Upsilon^{-1}}(\mc(\pi))$ is tempered.
%%%%%%%%%%%%%%%%%%%%%%%%%%%%%%%%%%%%%%%%%%%%%%%
%
%
%
%
%
%
%
%
%
%
%
%
%
%
%%%%%%%%%%%%%%%%%%%%%%%%%%%%%%%%%%%%%%%%%%%%%%%
\section{A bilinear form}
Following \cite{LMa}, we shall study the main identity.
A key ingredient of this section is the stability of the integral defining a Bessel functions,
which is given in Appendix~B.
Indeed, we apply Theorem~\ref{stability} to the function $A_{\#}^{\psi, \Upsilon}(W, g)$
(See also Remark~\ref{stability rmk}).
\begin{proposition}
\label{stable Y def}
Let $K_0 \in \mathcal{CSGR}(G)$.
Then the integral
\[
Y^{\psi, \Upsilon}(W, t)
:=
\int_{N^\prime}^{st} A_{\#}^{\psi, \Upsilon}(W, w_{U^\prime}^\prime w_0^{M^\prime} tn) \psi_{N^\prime}(n)^{-1} \, dn, \quad t \in T^\prime
\]
stabilizes uniformly for $W \in C(N \backslash G, \psi_N)^{K_0}$ and locally uniformly in $t \in T^\prime$.
In particular, $Y^{\psi, \Upsilon}(W_s, t)$ is entire in $s \in \mC$ and if $\pi \in \mathrm{Irr}_{gen} \mM$ and 
$W \in \mathrm{Ind}(\mathbb{W}^{\psi_{N_M}}(\pi))$ then $Y^{\psi, \Upsilon}(M_s^\ast W, t)$
is meromorphic in $s$.
Both $Y^{\psi, \Upsilon}(W_s, t)$ and $Y^{\psi, \Upsilon}(M_s^\ast W, t)$ are locally constant in $t$, uniformly in $s \in \mC$.

Finally, if we assume that $\pi \in \mathrm{Irr}_{meta, gen} \mM$ and that $\pi_\circ = \mathcal{D}_{\psi^{-1}}^{\Upsilon^{-1}}(\pi)$
is irreducible then for any $W^{\wedge} \in \mathrm{Ind}(\mathbb{W}^{\psi_{N_M}^{-1}}(\pi))$ we have
\[
Y^{\psi^{-1}, \Upsilon^{-1}}(M^\ast W^\wedge, t)
= \mathbb{B}_{\pi_\circ}^{\psi_{N_\circ}^{-1}} (w_{U^\prime}^\prime w_0^{M^\prime} t) A_{\#}^{\psi^{-1}, \Upsilon^{-1}}(M^\ast W^\wedge, e)
\]
\end{proposition}
We use Theorem~\ref{stability} in another way as follows,
which is proved in the same argument as \cite[Lemma~5.4]{LMb}. 
%%%%%%%%%%%%%%%
%
%
%
%
%
%
%
%
%
%%%%%%%%%%%%%%%%
\begin{lemma}
\label{Lemma 5.4}
Let $W^\prime \in C^{sm}(N^\prime \backslash G^\prime, \psi_{N^\prime})$ and $(W^\prime)^\wedge \in C^{sm}(N^\prime \backslash G^\prime, \psi_{N^\prime}^{-1})$.
Assume that the function $(t, n) \in T^\prime \times N^\prime \mapsto W^\prime(w_0^\prime t n)$ is compactly supported.
(In particular, $W^\prime \in C_{c}^{sm}(N^\prime \backslash G^\prime, \psi_{N^\prime})$.)  
Then the iterated integral
\[
\int_{N^\prime}^{st} (\int_{N^\prime \backslash G^\prime}
W^\prime(g) (W^\prime)^\wedge(gu) \, dg) \psi_{N^\prime}(u) \, du
\]
is well-defined and is equal to 
\[
\int_{T^\prime}
\left( \int_{N^\prime} \delta_{B^\prime}(t) W^\prime(w_{U^\prime}^\prime w_0^{M^\prime} tn)
 \psi_{N^\prime}(n^{-1}) \, dn \right)
 \left( \int_{N^\prime}^{st} (W^\prime)^\wedge(w_{U^\prime}^\prime w_{0}^{M^\prime} tu) 
 \psi_{N^\prime}(u) \, du \right) \, dt.
\]
\end{lemma}
%%%%%%%%%%%%%%%
%
%
%
%
%
%
%
%
%
%%%%%%%%%%%%%%%%
\subsection{}
\label{def of whittaker space}
In order to apply Lemma~\ref{Lemma 5.4} for $A_{\#}^{\psi, \Upsilon}(W_s, \cdot)$
we make a special choice of $W$.
Consider the $P$-invariant subspace $\mathrm{Ind}(\mathbb{W}^{\psi_{N_M}}(\pi))^\circ$
of $\mathrm{Ind}(\mathbb{W}^{\psi_{N_M}}(\pi))$ consisting of functions supported in the big cell $Pw_U P = Pw_U U$.
Any element of $\mathrm{Ind}(\mathbb{W}^{\psi_{N_M}}(\pi))^\circ$ is a linear combination of functions of the form
\begin{equation}
\label{(5.2)}
W(u^\prime m w_U u) \delta_P(m)^{\frac{1}{2}} W^M(m) \phi(u), m \in M, u, u^\prime \in U
\end{equation}
with $W^M \in \mathbb{W}^{\psi_{N_M}}(\pi)$ and $\phi \in C_c^\infty(U)$.
Let $\eta_{\mM}$ be the embedding $\eta_{\mM}(g) =\left(\begin{smallmatrix} g&\\ &I_n\end{smallmatrix}\right)$
of $\mM^\prime$ into $\mM$. Also let $\eta_M = \varrho \circ \eta_{\mM}$ so that $\eta_M(g) = \left(\begin{smallmatrix} g&&\\ &I_{2n}& \\ &&g^\ast \end{smallmatrix}\right)$
%%%%%%%%%%%%%%%
%
%
%
%
%
%
%
%
%
%%%%%%%%%%%%%%%%
\begin{Definition}
Let $\mathrm{Ind}(\mathbb{W}^{\psi_{N_M}}(\pi))^\circ_\#$ be the linear subspace of $\mathrm{Ind}(\mathbb{W}^{\psi_{N_M}}(\pi))$
generated by $W^s$ as in \eqref{(5.2)} that satisfy the additional property that the function
$(t, n) \mapsto W^M(\eta_M(t w_0^{\mM^\prime} n))$ is compactly supported on $T_{\mM^\prime}^\prime \times N_{\mM^\prime}^\prime$,
or equivalently, that the function $W^M \circ \eta_M$ on $\mM^\prime$ is supported in the big cell 
$B_{\mM^\prime}^\prime w_0^{\mM^\prime} N_{\mM^\prime}^\prime$ and 
its support is compact module $N_{\mM^\prime}^\prime$.
\end{Definition}
Note that this space is non-zero by the proof of \cite[Lemma~6.13]{LMb}.Further, this space is invariant under $\eta(T^\prime) \ltimes N$.
\begin{lemma}
\label{Lemma 5.6}
For any $W \in \mathrm{Ind}(\mathbb{W}^{\psi_{N_M}}(\pi))^\circ_\#$, the function 
$A_{\#}^{\psi, \Upsilon}(W_s, w_{U^\prime}^\prime w_0^{M^\prime}tn)$ is compactly supported in $t \in T^\prime$
and $n \in N^\prime$ uniformly in $s \in \mC$.
\end{lemma}
\begin{proof}
A proof is identical to the proof of \cite[Lemma~5.6]{LMb}.
\end{proof}
%%%%%%%%%%%%%%%
%
%
%
%
%
%
%
%
%
%%%%%%%%%%%%%%%%
\subsection{}
For $W \in \mathrm{Ind}(\mathbb{W}^{\psi_{N_M}}(\pi))_{\#}^\circ$ and $W^\vee \in \mathrm{Ind}(\mathbb{W}^{\psi_{N_M}^{-1}} (\pi^\vee))$,
we define
\[
B(W, W^\vee, s):=
\int_{N^\prime}^{st} \int_{N^\prime \backslash G^\prime} A_{\#}^{\psi, \Upsilon}(W_s, g) A_{\#}^{\psi^{-1}, \Upsilon^{-1}}(W_{-s}^\vee, gu)
\psi_{N^\prime}(u) \, du.
\]
By Lemma~\ref{Lemma 5.6}, we may apply the argument in the proof of \cite[Lemma~5.4]{LMb}.
Then we have
\[
\int_{T^\prime} \int_{N^\prime} \int_{\Omega} \delta_{B^\prime}(t)
A_{\#}^{\psi, \Upsilon}(W_s, w_{U^\prime}^\prime w_0^{M^\prime} tn) 
A_{\#}^{\psi^{-1}, \Upsilon^{-1}}(W_{-s}^\vee, w_{U^\prime}^\prime w_{0}^{M^\prime} tnu)
\psi_{N^\prime}(u) \, du \, dn \, dt
\]
for any sufficiently large $\Omega \in \mathcal{CSGR}(N^\prime)$. This implies that 
$B(W, W^\vee, s)$ is an entire function of $s$ and from Proposition~\ref{stable Y def}, we have
\begin{equation}
\label{5.3}
B(W, W^\vee, s)=
\int_{T^\prime} Y^{\psi, \Upsilon}(W_s, t) Y^{\psi^{-1}, \Upsilon^{-1}}(W_{-s}^\vee, t) \delta_{B^\prime}(t) \, dt
\end{equation}
for any $W \in \mathrm{Ind}(\mathbb{W}^{\psi_{N_M}}(\pi))_{\#}^\circ$ and $W^\vee \in \mathrm{Ind}(\mathbb{W}^{\psi_{N_M}^{-1}} (\pi))$.

Assume that $\pi \in \mathrm{Irr}_{\rm gen, ut} \mM$ and $\sigma = \mathcal{D}_{\psi^{-1}}^{\Upsilon^{-1}}(\mc(\pi))$ is irreducible.
Then for $W \in \mathrm{Ind}(\mathbb{W}^{\psi_{N_M}}(\pi))_{\#}^\circ$ and $W^\vee \in \mathrm{Ind}(\mathbb{W}^{\psi_{N_M}^{-1}} (\pi))$,
\begin{equation}
\label{5.4}
\text{the left-hand side of the Main Identity (MI) is $B \left(W, M\left(\frac{1}{2}\right)W^\vee, \frac{1}{2} \right)$}
\end{equation}
%%%%%%%%%%%%%%%%%%%%%%%%%%%%%%%%%%%%%%%%%%%%%%%
%
%
%
%
%
%
%
%
%
%
%
%
%
%
%%%%%%%%%%%%%%%%%%%%%%%%%%%%%%%%%%%%%%%%%%%%%%%
\section{Further analysis}
Fix an element $\varepsilon_1$ of the form $\ell_M(X)$ where $X \in \mathrm{Mat}_{\mn \times \mn}$
and the last row of $X$ is $-\xi_{\mn}$.
Then we can check that 
\[
\psi_{V_-}(\varepsilon_1^{-1} v \varepsilon_1) = \psi_{V_-}(v) \psi(v_{\mn, 2\mn +1}).
\]
For any $W \in C^{sm}(N \backslash G, \psi_N)$, define
\begin{equation}
\label{6.1}
A_e^\psi(W) :=
\int_{V_\gamma \backslash V_-} W(\gamma v \varepsilon_1) \psi_{V_-}(\varepsilon_1^{-1} v \varepsilon_1)^{-1} \, dv.
\end{equation}
%Recall that $N^{\#} = V_{-} \rtimes \eta(N^\prime)$
%For $v_0 \in \eta_{M}(N_{\mM^\prime}^\prime)$, we have
%\[
%A_{e}^\psi(W(\cdot v_0)) = 
%\]
%since $\eta_{M}(N_{\mM^\prime}^\prime)$ normalizes $V_\gamma$.
%Moreover, for $u_0 = \ell$
\begin{lemma}
\label{Lemma 6.1}
For any $W \in C^{\rm sm}(N \backslash G, \psi_N)$, the integrand in is compactly supported on $V_\gamma \backslash V$
and we have $A_{\#}^{\psi, \Upsilon}(W, e) = A_{e}^\psi(W)$.
\end{lemma}
\begin{proof}
The proof is identical to the proof of \cite[Lemma~6.1]{LMb} because of Lemma~\ref{Lemma 3.2}
and \eqref{2.1a}--\eqref{2.1c}
\end{proof}
\begin{Remark}
The definition of $A_{\#}^{\psi, \Upsilon}(W, \cdot )$ depends on the choice of $\Upsilon$,
and  so does its invariance by Lemma~\ref{Lemma 3.1}.
In this lemma, we evaluate $A_{\#}^{\psi, \Upsilon}(W, \cdot)$ only at $e$,
and we see that this value is independent of the choice of  $\Upsilon$.
Further, from Lemma~\ref{Lemma 3.1}, $W \mapsto A_{\#}^{\psi, \Upsilon}(W, \cdot )$ is $(N^{\#, \psi_{N^{\#}}})$
equivariant, and the same is for $A_{e}^\psi(W)$ by this lemma
(We may check this condition by a direct computation).
\end{Remark}
We now explicate $A_{\#}^{\psi, \Upsilon}(W_s, \cdot)$ on the big cell $N^\prime w_{U^\prime}^\prime P^{\prime}$.
By Lemma~\ref{Lemma 3.1} it is enough to consider the element $w_{U^\prime}^\prime$.
\begin{lemma}
Let $\pi \in \mathrm{Irr}_{\rm gen} M$.
Then for $\mathrm{Re} s \gg_{\pi} 1$ and $W \in \mathrm{Ind}(\mathbb{W}^{\psi_{N_M}}(\pi))$
we have 
\begin{equation}
\label{6.3}
A_{\#}^{\psi, \Upsilon}(W_s, w_{U^\prime}^\prime)
= \alpha_{\psi^{-1}}(w_{U^\prime}^\prime) \int_{V_U} W_s(w_U v) \psi_{U}(v)^{-1} \, dv
=\alpha_{\psi^{-1}}(w_{U^\prime}^\prime) \int_{V_M^\# \backslash V_-} W_s(w_U v) \psi_{V_-}(v)^{-1} \, dv.
\end{equation}
\end{lemma}
\begin{proof}
A proof is identical to the proof of \cite[Lemma~6.3]{LMb}.
\end{proof}
%%%%%%%%%%%%%%%%%%%%%%%%%%%%%%%%%%%%%%%%%%%%%%%
%
%
%
%
%
%
%
%
%
%
%
%
%
%
%%%%%%%%%%%%%%%%%%%%%%%%%%%%%%%%%%%%%%%%%%%%%%%
\subsection{}
We can now explicate $Y^{\psi, \Upsilon}(W_s, t)$ for $\mathrm{Re} s \gg 1$.
\begin{lemma}
\label{Lemma 6.6}
Let $\pi \in \mathrm{Irr}_{\rm gen} M$.
For $\mathrm{Re} s \gg_{\pi} 1$ and any $W \in \mathrm{Ind}(\mathbb{W}^{\psi_{N_M}}(\pi))$, $t \in T^\prime$
we have the identity
\begin{equation}
\label{6.5}
Y^{\psi, \Upsilon}(W_s, t) = \nu^\prime(t)^{n-\frac{1}{2}}\Upsilon_{M^\prime}(t)^{-1}
 \alpha_{\psi^{-1}}(w_{U^\prime}^\prime)
\alpha_{\psi^{-1}}(w_0^{M^\prime} t)
\int_{V_M^\# \backslash N^\#} W_s(w_U \eta(w_0^{M^\prime} t) v) \psi_{N^\#}(v)^{-1} \, dv
\end{equation}
where the right-hand side is absolutely convergent.
\end{lemma}
We first need the following convergence result.
\begin{lemma}[cf. Lemma~6.8]
Let $\pi \in \mathrm{Irr}_{\rm gen} M$.
Then for $\mathrm{Re}\,s \gg_{\pi} 1$ we have
\[
\int_{N_\mM^{t} \cap \mathcal{P}} \int_{U} |W_s(\varrho(n) w_U u g)| \, du \, dn < \infty
\]
and 
\[
\int_{N_\mM^{t} \cap \mathcal{P}^\ast} \int_{U} |W_s(\varrho(n) w_U u g)| \, du \, dn < \infty
\]
for any $W \in \mathrm{Ind}(\mathbb{W}^{\psi_{N_M}}(\pi))$, $g \in G$.
\end{lemma}
\begin{proof}
The proof of this lemma is identical to the proof of \cite[Lemma~6.8]{LMb}.
\end{proof}
%%%%%%%%%%%%%%%%%%%%%%%%%%%%%%%%%%%%%%%%%%%%%%%
%
%
%
%
%
%
%
%
%
%
%
%
%
%
%%%%%%%%%%%%%%%%%%%%%%%%%%%%%%%%%%%%%%%%%%%%%%%
\subsection{}
We now go back to the bilinear form $B$.
The following lemma is proved in the same argument as the proof of \cite[Lemma~6.10]{LMb}.
\begin{lemma}
\label{Lemma 6.10}
For $W \in \mathrm{Ind}(\mathbb{W}^{\psi_{N_M}}(\pi))_{\#}^\circ$
the integrand on the right-hand side of \eqref{6.5} is compactly supported in $t, v$ uniformly in $s$
(i. e. the support in $(t, v)$ is contained in a compact set which is independent of $s$).
In particular, the identity \eqref{6.5} holds for all $s \in \mC$.
\end{lemma}
%%%%%%%%%
%
%
%
%
%%%%%%%%%
\begin{lemma}[cf. Lemma~6.11]
Let $\pi \in \mathrm{Irr}_{\rm gen} M$.
Then for $- \mathrm{Re}\, s \gg 1$ we have
\[
B(W, W^\vee, s)=
\int_{U} \int_{V_M^{\#} \backslash N^{\#}} \int_{N_{\mM^\prime} \backslash \mM^\prime}
W_s(\eta_M(g) w_U v) W_{-s}^\vee(\eta_M(g) w_U u) \delta_P(\eta_M(g))^{-1} |\det g|^{1-n} \psi_{N^\#}(v)^{-1}
\psi_U(u) \, dg \, dv \, du
\]
for any $W \in \mathrm{Ind}(\mathbb{W}^{\psi_{N_M}}(\pi))_{\#}^\circ$, $W^\vee \in \mathrm{Ind}(\mathbb{W}^{\psi_{N_M}^{-1}}(\pi^\vee))$
with the integral being absolutely convergent.
\end{lemma}
\begin{proof}
This is proved in a similar argument as the proof of \cite[Lemma~6.11]{LMb}
except for some minor differences.
For the sake of completeness, we shall repeat their argument in our setting.

Suppose that $-\mathrm{Re}\, s \gg 1$.
Then by \eqref{5.3}, Lemma~\ref{Lemma 6.6} and \ref{Lemma 6.10}
and the fact 
\[
\delta_{B^\prime}(t) = \delta_{B^\prime_{M^\prime}}(t) \nu^\prime(t)^n, \quad t \in T^\prime. 
\]
$B(W, W^\vee, s)$ is equal to 
\[
\int_{T_{\mM^\prime}^\prime} \int_{\eta(N_{M^\prime}^\prime) \ltimes U}  \int_{\eta(N_{M^\prime}^\prime) \ltimes U} 
W_s (\eta_M((w_0^{\mM^\prime} t)^\ast) w_U v_1)W_{-s} (\eta_M((w_0^{\mM^\prime} t)^\ast) w_U v_2)
|\det t|^{3n-1} \delta_{B_{\mM^\prime}^\prime}(t) \psi_{N^\#}(v_1)^{-1} \psi_{N^\#}(v_2) \, dv_1 \, dv_2 \, dt
\]
where the integral is absolutely convergent.
By the change of variable, $B(W, W^\vee, s)$ is equal to 
\[
\int_{T_{\mM^\prime}^\prime} \int_{N_{\mM^\prime}^\prime} \int_{U}  \int_{\eta(N_{M^\prime}^\prime) \ltimes U} 
W_s (\eta_M((w_0^{\mM^\prime} t n)^\ast) w_U u v_1)W_{-s} (\eta_M((w_0^{\mM^\prime} t n)^\ast) w_U u)
|\det t|^{3n-1} \delta_{B_{\mM^\prime}^\prime}(t) \psi_{N^\#}(v_1)^{-1}  \, dv_1 \, du \, dn  \, dt.
\]
Finally, by the Bruhat decomposition and the fact
\[
\delta_{P}(\eta_M(g)) = |\det g|^{2n}, \quad \text{for any $g \in \mM^\prime$}
\]
the lemma readily follows.
\end{proof}

Define when convergent
\[
\{ W, W^\vee \} :=
\int_{N_{\mM^\prime}^\prime \backslash \mM^\prime}
W(\eta_M(g)) W^\vee(\eta_M(g)) \delta_P(\eta_M(g))^{-1} |\det g|^{1-n} \, dg,
\]
which converges for any $(W, W^\vee) \in \mathrm{Ind}(\mathbb{W}^{\psi_{N_M}}(\pi)) \times \mathrm{Ind}(\mathbb{W}^{\psi_{N_M}^{-1}}(\pi^\vee))$ 
when $\pi$ is unitarizable by \cite[Lemma~1.2]{LMd}.
Then we get for any $W \in \mathrm{Ind}(\mathbb{W}^{\psi_{N_M}}(\pi))_{\#}^\circ$, $W^\vee \in \mathrm{Ind}(\mathbb{W}^{\psi_{N_M}^{-1}}(\pi^\vee))$
and when $- \mathrm{Re}\, s \gg 1$
\begin{equation}
\label{6.11}
B(W, W^\vee, s) = 
\int_{U} \int_{V_M^\# \backslash N^\#} \{ W_s(\cdot w_U v), W_{-s}^\vee(\cdot w_U u) \} \psi_{N^\#}(v)^{-1} \psi_U(v) \, dv \, du.
\end{equation}
%%%%%%%%%%%%%%%%%%%%%%%%%%%%%%%%%%%%%%%%%%%%%%%
%
%
%
%
%
%
%
%
%
%
%
%
%
%
%%%%%%%%%%%%%%%%%%%%%%%%%%%%%%%%%%%%%%%%%%%%%%%
\subsection{}
In the same argument as the proof of \cite[Lemma~6.13]{LMb} using
Theorem~\ref{Theorem A.1}, Proposition~\ref{stable Y def} and Lemma~\ref{Lemma 6.10}
instead of \cite[Theorem A.1, Corollary~5.3, Lemma~6.10]{LMb}, 
we may prove the following lemma.
\begin{lemma}
\label{Lemma 6.13}
Assume that $\pi \in \mathrm{Irr}_{\rm gen, ut} M$ and $\sigma = \mathcal{D}^{\Upsilon^{-1}}_{\psi^{-1}}(\pi)$ is irreducible and tempered.
Then the bilinear form $B(W, M(\frac{1}{2})W^\wedge, \frac{1}{2})$ does not identically zero on 
$W \in \mathrm{Ind}(\mathbb{W}^{\psi_{N_M}}(\pi))_{\#}^\circ \times \mathrm{Ind}(\mathbb{W}^{\psi_{N_M}^{-1}}(\mc(\pi)))$
\end{lemma}
By \eqref{5.4}, Lemma~\ref{Lemma 6.1} and Lemma~\ref{Lemma 6.13}, we conclude
\begin{corollary}
\label{cor6.14}
Suppose that $\pi \in \mathrm{Irr}_{\rm ut, temp} M$ is good and $\sigma =\mathcal{D}_{\psi^{-1}}^{\Upsilon^{-1}}(\pi)$
is tempered. Then 
\[
B\left(W, M(\frac{1}{2})W^\wedge, \frac{1}{2} \right) = c_\pi A_{e}^{\psi}(M^\ast W) A_e^{\psi^{-1}}(M^\ast W^\wedge)
\]
for all $W \in \mathrm{Ind}(\mathbb{W}^{\psi_{N_M}}(\pi))_{\#}^\circ$ and $W^\vee \in \mathrm{Ind}(\mathbb{W}^{\psi_{N_M}^{-1}}(\mc(\pi)))$
Moreover, the linear form $A_e^\psi(M^\ast W)$ does not vanish identically on $\mathrm{Ind}(\mathbb{W}^{\psi_{N_M}}(\pi))_{\#}^\circ$.
\end{corollary}
In other words, because of Proposition~\ref{first reduction prp}, Theorem~\ref{main theorem} is reduced to the following statement.
\begin{proposition}
\label{Proposition 6.15}
Assume that $\pi \in \mathrm{Irr}_{\rm ut, temp} M$ is good and $\sigma =\mathcal{D}_{\psi^{-1}}^{\Upsilon^{-1}}(\pi)$
is tempered.
Then for any $W \in \mathrm{Ind}(\mathbb{W}^{\psi_{N_M}}(\pi))_{\#}^\circ$ and $W^\wedge \in \mathrm{Ind}(\mathbb{W}^{\psi_{N_M}^{-1}}(\mc(\pi)))$
we have
\begin{equation}
\label{6-12}
B\left(W, M(\frac{1}{2})W^\wedge, \frac{1}{2} \right) = \omega_\pi(\tau) A_{e}^{\psi}(M^\ast W) A_e^{\psi^{-1}}(M^\ast W^\wedge).
\end{equation}
\end{proposition}
%%%%%%%%%%%%%%%%%%%%%%%%%%%%%%%%%%%%%%%%%%%%%%%
%
%
%
%
%
%
%
%
%
%
%
%
%
%
%%%%%%%%%%%%%%%%%%%%%%%%%%%%%%%%%%%%%%%%%%%%%%%
\section{Application of functional equations}
We define $\underline{B}(W, W^\vee, s)$ to be the right hand side of \eqref{6.11}
whenever the integral defining $\{ \cdot, \cdot \}$ and the double integrals over $V_M^{\#} \backslash N^{\#}$
and $U$ are absolutely convergent.
Clearly for $g \in \mM^\prime$, with $|\det g| =1$
\[
\{ W(\cdot \eta_M(g)), W\vee(\cdot \eta_M(g)) \} = \{ W, W\vee \}.
\]
Then as in \cite[(7.1)]{LMb}, $\underline{B}(W, W^\vee, s)$ is equal to 
\begin{multline}
\label{7.1}
\int_{N_{\mM^\prime}^\prime} \int_U \int_U \{ W_s(\cdot \eta_M(n)w_U v), W_{-s}^\vee(\cdot w_U u) \} \psi_U(v)^{-1} \psi_U(u) 
\psi_{N_{\mM^\prime}^\prime}(n) \, dv \, du \, dn
\, dv \, du \, dn \\
=
\int_{N_{\mM^\prime}^\prime} \int_U \int_U \{ W_s(\cdot  w_U v), W_{-s}^\vee(\cdot \eta_M(n)w_U u) \} \psi_U(v)^{-1} \psi_U(u) 
\psi_{N_{\mM^\prime}^\prime}(n)^{-1} \, dv \, du \, dn
\, dv \, du \, dn \\
=
\int_{V_M^{\#} \backslash N^{\#}} \int_U \{ W_s(\cdot w_U v_1), W_{-s}^\vee(\cdot w_U v_2) \} \psi_U(v_1)^{-1} \psi_{N^\#}(v_2)
\, dv_1 \, dv_2.
\end{multline}
By \eqref{6.11}, for any $\pi \in \mathrm{Irr}_{\rm gen} M$, $W \in \mathrm{Ind}(\mathbb{W}^{\psi_{N_M}}(\pi))_{\#}^\circ$, 
$W^\vee \in \mathrm{Ind}(\mathbb{W}^{\psi_{N_M}^{-1}}(\mc(\pi)))$ and $-\mathrm{Re} \, s \gg 1$ we have
\[
\underline{B}(W, W^\vee, s) = B(W, W^\vee, s).
\]
The following proposition is proved in a similar way as the argument in \cite[Appendix~B]{LMd}, practically word for word.
Hence, we omit its proof.
\begin{proposition}
\label{Proposition 7.1}
Let $\pi \in \mathrm{Irr}_{\rm temp} M$. Then
\begin{enumerate}
\item For $\mathrm{Re}\, s \gg 1$, $\underline{B}(W, W^\vee, s)$ is well-defined for any $W \in  \mathrm{Ind}(\mathbb{W}^{\psi_{N_M}}(\pi))$
$W^\vee \in \mathrm{Ind}(\mathbb{W}^{\psi_{N_M}^{-1}}(\pi^\vee))^\circ$
\item For $-\mathrm{Re}\, s \gg 1$, $\underline{B}(W, W^\vee, s)$ is well-defined for any $W \in  \mathrm{Ind}(\mathbb{W}^{\psi_{N_M}}(\pi))^\circ$
$W^\vee \in \mathrm{Ind}(\mathbb{W}^{\psi_{N_M}^{-1}}(\pi^\vee))$
\item For $-\mathrm{Re}\, s \gg 1$, we have
\[
\underline{B}(W, M(s)W^\wedge, s) = \underline{B}(M(s)W, W^\wedge, -s)
\]
for any $W \in  \mathrm{Ind}(\mathbb{W}^{\psi_{N_M}}(\pi))^\circ$
$W^\wedge \in \mathrm{Ind}(\mathbb{W}^{\psi_{N_M}^{-1}}(\pi))^\circ$.
\end{enumerate}
Recall the definition of the space $W^\vee \in \mathrm{Ind}(\mathbb{W}^{\psi_{N_M}^{-1}}(\pi))^\circ$ in Section~\ref{def of whittaker space}.
\end{proposition}
Combined with the above we get
\begin{corollary}
\label{Corollary 7.2}
For $-\mathrm{Re}\, s \gg 1$ we have
\[
\underline{B}(W, M(s)W^\wedge, s) = \underline{B}(M(s)W, W^\wedge, -s)
\]
for any $W \in  \mathrm{Ind}(\mathbb{W}^{\psi_{N_M}}(\pi))^\circ_\#$
$W^\wedge \in \mathrm{Ind}(\mathbb{W}^{\psi_{N_M}^{-1}}(\pi))^\circ$.
\end{corollary}
%%%%%%%%%%%%%%%%%%%%%%%%%%%%%%%%%%%%%%%%%%
%
%
%
%
%
%
%
%
%
%
%
%
%
%
%%%%%%%%%%%%%%%%%%%%%%%%%%%%%%%%%%%%%%%%%%%%%%%
\subsection{}
Put $\varepsilon_2 = \ell_\mM(e_{1,1}+J)$
where $e_{1,1}$ is the matrix in $\mathrm{Mat}_n$ with $1$ in the upper left corner and zero elsewhere
$\varepsilon_3 = w_{2n, n}^\prime \varepsilon_2$ and $\varepsilon_4$ an arbitrary element of $N_\mM$.
As in \cite[Section~9]{Mo2}, define 
\[
\Delta(t) := |t_1|^{-n} \delta_B^{\frac{1}{2}}(\varrho(t)), \quad t = \mathrm{diag}(t_1, \dots, t_{2n}) \in T_\mM. 
\]
In particular, when $t \in \eta_\mM^\vee(t^\prime)$ with $t^\prime \in T_{\mM^\prime}^\prime$, we have $\Delta(t) = \delta_{B^\prime}^{\frac{1}{2}}(\varrho^\prime(t^\prime))$.

Let $T^{\prime \prime} = \eta_{\mM}^\vee(T_{\mM^\prime}^\prime) \times Z_\mM$.
For any $W \in C^{\rm sm}(N \backslash G, \psi_N)$ and $t \in T^{\prime \prime}$, define
\begin{align}
\label{7.7}
E^\psi(W,t) &:= \Delta(t)^{-1} \int_{\eta_{M}(N_{\mM^\prime}^\prime) \backslash N^{\#}} W(\varrho(t \varepsilon_4 \varepsilon_3) w_U v) \psi_{N^{\#}}(v)^{-1} \, dv\\
\notag &=   \Delta(t)^{-1} \psi_{N_\mM}(t \varepsilon_4 t^{-1}) \int_{\eta_{M}(N_{\mM^\prime}^\prime) \backslash N^{\#}} W(\varrho(t \varepsilon_3) w_U v) \psi_{N^{\#}}(v)^{-1} \, dv
\end{align}
\begin{definition}
Recall the definition of a certain subspace of $\mathbb{W}^{\psi_{N_\mM}^{-1}}(\pi)$ in \cite[Definition~7.5]{LMb}.
\label{Definition 7.5}
Let $\mathbb{W}^{\psi_{N_\mM}^{-1}}(\pi)_{\natural}$ be the subspace of $\mathbb{W}^{\psi_{N_\mM}^{-1}}(\pi)$
consisting of $W$ such that 
\[
W(\cdot \varepsilon_3) |_{\mathcal{P}^\ast} \in C_c^\infty(N_\mM \backslash \mathcal{P}^\ast, \psi_{N_\mM}^{-1})
\text{ and } W(\cdot \varepsilon_3)|_{\eta_{\mM}^\vee(T_{\mM^\prime}^\prime) \ltimes \bar{R}}
\in C_c^\infty(\eta_{\mM}^\vee(T_{\mM^\prime}^\prime) \ltimes \bar{R}).
\]
\end{definition}
%%%%%%%%%%%%%%%%%
%
%
%
%
%
%
%%%%%%%%%%%%%%%%%%
As remarked in \cite[p.33]{LMb}, this space is non-zero, and thus the following space is also non-zero.
\begin{definition}
\label{Definition 7.8}
Let $\mathrm{Ind}(\mathbb{W}^{\psi_{N_M}^{-1}}(\pi))_{\natural}^\circ$ be the linear subspace of 
$\mathrm{Ind}(\mathbb{W}^{\psi_{N_M}^{-1}}(\pi))^\circ$ spanned by the functions which vanish outside $P w_U N$ and 
on the big cell are given by
\[
W(u^\prime m w_U u) = \delta_P^{\frac{1}{2}}(m) W^M(m) \phi(u), \quad m \in M, u, u^\prime \in U
\]
with $\phi \in C_c^\infty(U)$ and $W^M \circ \varrho \in \mathbb{W}^{\psi_{N_\mM}^{-1}}(\pi)_{\natural}$.
\end{definition}
%%%%%%%%%%%%%%%%%
%
%
%
%
%
%
%%%%%%%%%%%%%%%%%%
In the same argument as the proof of \cite[Lemma~7.9]{LMb}, the following lemma is proved.
\begin{lemma}
\label{Lemma 7.9}
Let $\pi \in \mathrm{Irr}_{\rm gen} M$.
For $\mathrm{Re}\, s \gg 1$ the integral \eqref{7.7} defining $E^\psi(W_s, t)$ converges for any $W \in \mathrm{Ind}(\mathbb{W}^{\psi_{N_M}} (\pi))$
and $t \in S$ uniformly for $(s, t)$ in a compact set.
Hence, $E^\psi(W_s, t)$ is holomorphic for $\mathrm{Re}\, s \gg 1$ and continuous in $t$.
If $W^\wedge \in  \mathrm{Ind}(\mathbb{W}^{\psi_{N_M}^{-1}} (\pi))^\circ$
then $E^{\psi^{-1}}(W_s^\wedge, t)$ is entire in $s$ and locally constant in $t$, uniformly in $s$.
If moreover $W^\wedge \in  \mathrm{Ind}(\mathbb{W}^{\psi_{N_M}^{-1}} (\pi))^\circ_\natural$ then
$E^{\psi^{-1}}(W_s^\wedge, t)$  is compactly supported in $t \in S$, uniformly in $s$.
\end{lemma}
%%%%%%%%%%%%%%%%%
%
%
%
%
%
%
%%%%%%%%%%%%%%%%%%
\begin{proposition}
\label{Proposition 7.10}
Let $\pi \in \mathrm{Irr}_{\rm temp} M$.
Then for $\mathrm{Re}\, s \gg 1$ we have
\begin{equation}
\label{7.9}
\underline{B}(W, W^\vee, s) = \int_{\eta_\mM^\vee(T_{\mM^\prime}^\prime)}
E^\psi(W_s, t) E^{\psi^{-1}}(W_{-s}^\vee, t) \frac{dt}{|\det t|}
\end{equation}
for any $W \in  \mathrm{Ind}(\mathbb{W}^{\psi_{N_M}}(\pi))$
$W^\vee \in \mathrm{Ind}(\mathbb{W}^{\psi_{N_M}^{-1}}(\pi))^\circ_\natural$
where the integrand on the right-hand side is continuous and compactly supported.
\end{proposition}
\begin{proof}
This is proved in the same argument as the proof of \cite[Proposition~7.10]{LMb} using \eqref{7.1} and Lemma~\ref{Lemma 7.9}
instead of (7.1) and Lemma~7.9 in \cite{LMb}, respectively.
\end{proof}
Combining Proposition~\ref{Proposition 7.10} with Corollary~\ref{Corollary 7.2} we get
\begin{proposition}
\label{Proposition 7.11}
Let $\pi \in \mathrm{Irr}_{\rm temp} M$.
Then for $-\mathrm{Re} \, s \gg 1$ and any$W \in \mathrm{Ind}(\mathbb{W}^{\psi_{N_M}}(\pi))^\circ_\#$
$W^\vee \in \mathrm{Ind}(\mathbb{W}^{\psi_{N_M}^{-1}}(\pi))^\circ_\natural$, we have
\[
B(W, M(s)W^\vee, s) = \int_{S} E^\psi(M_s^\ast W, t)E^{\psi^{-1}}(W_s^\vee, t) \frac{dt}{|\det t|} 
\]
where the integrand is continuous and compactly supported.
Here, for the simplicity, we denote $S = \eta_\mM^\vee(T_{\mM^\prime}^\prime)$.
\end{proposition}
%%%%%%%%%%%%%%%%%%%%%%%%%%%%%%%%%%%%%%%%%%
%
%
%
%
%
%
%
%
%
%
%
%
%
%
%%%%%%%%%%%%%%%%%%%%%%%%%%%%%%%%%%%%%%%%%%%%%%%
\section{Proof of Proposition~\ref{Proposition 6.15}}
\subsection{}
Let $\mathfrak{d} = \mathrm{diag}(1, -1, \dots, (-1)^{n-1}) \in \mathrm{Mat}_n$. We now fix
\[
\varepsilon_4 = \ell_{\mM}(-\frac{1}{2} \mathfrak{d} w_0^{\mM^\prime}) \in N_\mM.
\]
This element is denoted by $\varepsilon^\prime$ in the beginning of \cite[Section~8]{Mo2} with the parameter
$\mathfrak{a} = -\frac{1}{2}$ in the notation of that paper.
We also fix $\varepsilon_2 = \ell_\mM(\mathfrak{d})$
(and correspondingly $\varepsilon_3 = w_{2n, n}^\prime \varepsilon_2$).
Then we have
\[
\psi_{\bar{U}}(\bar{v}) = \psi_U ((\varrho(\varepsilon_4 \varepsilon_3 ) w_U)^{-1} \bar{v} \varrho(\varepsilon_4 \varepsilon_3 ) w_U)^{-1}
= \psi(\bar{v}_{2n+1, 1}), \quad \bar{v} \in \bar{U}.
\]
We note that in \cite[4.1]{Mo2}, the character $\psi_{\bar{U}}$ is denoted by $\psi_{\bar{U}, F}$.
As in \cite[8.1]{LMb}, we may rewrite (for $\mathrm{Re}\, s \gg 1$)
\begin{multline}
\label{7.7}
E^\psi(W_s, t) = \Delta(t)^{-1} \int_{\eta_{\mM}^\vee(N_{\mM^\prime}^\prime) \backslash N_\mM^\flat} \int_{\bar{U}}
W_s(\varrho(t) \bar{v} \varrho(\varepsilon_4 \varepsilon_3 r) w_U) \psi_{\bar{U}}(\bar{v}) \psi_{N_\mM^\flat}(r)^{-1} \, d\bar{v} \, dr\\
= \Delta(t)^{-1} 
\int_{\bar{R}} \int_{\bar{U}}
W_s(\varrho(t) \bar{v} \varrho(\varepsilon_4 r \varepsilon_3 ) w_U) \psi_{\bar{U}}(\bar{v}) \psi_{\bar{R}}(r) \, d\bar{v} \, dr
\end{multline}
Here, recall $N_{\mM}^\flat = (N_\mM^\#)^\ast$ and $\psi_{N_\mM^\flat}(m) = \psi_{N_\mM^\#}(m^\ast)$,
and define $\psi_{\bar{R}}(r) = \psi_{N_\mM^\flat}(\varepsilon_3^{-1} r \varepsilon_3)^{-1}$,
and 
\[
\bar{R} = \varepsilon_3 (\eta_{\mM}(N_{\mM^\prime}^\prime) \ltimes \ell_\mM(J)) \varepsilon_3^{-1}
= w_{2n, n}^\prime (\eta_{\mM}(N_{\mM^\prime}^\prime) \ltimes \ell_{\mM}(J)) w_{2n, n}^{\prime-1}
= \left\{ \begin{pmatrix} I_n&\\ x&{}^{t} n \end{pmatrix} : x \in J, n \in N_{\mM^\prime}^\prime \right\} \subset {}^{t}N_{\mM} \cap \mathcal{P}^\ast.
\]
We now quote the following pertinent result from \cite{Mo2}.
Let 
\[
T_i := \left\{ \mathrm{diag}(1_{i-1}, z, 1_{2n-i}) : z \in E^\times \right\}
\]
Then we have $S = \prod_{i=1}^n T_i$.
For any $f \in C_c^\infty(S)$ and $g \in C(S)$, we write $f \ast g(\cdot) = \int_Sf(t)g(\cdot t) \, dt$.
\begin{theorem}
\label{Theorem 8.1}
Let $K_i$ be a compact open subgroup of $E^\times$ and $f_{K_i}$ be the characteristic function of $K_i$.
Regard $f_{K_i}$ as a function on $T_i$, and put $f = f_{K_1} \otimes \cdots \otimes f_{K_n} \in \otimes_i C_c^\infty(T_i) = C_c^\infty(S)$.

For any $W \in C^{\rm sm}(N \backslash G, \psi_N)$ which is left invariant under a compact open subgroup of $Z_M$,
the function $f \ast E^\psi(W_s, t)$ extends to an entire function in $s$ which is locally constant in $t$, uniformly in $s$.
Moreover, if $\pi \in \mathrm{Irr}_{\rm ut, temp} \mM$ then
\[
f \ast E^\psi(W_s, t)|_{s=\frac{1}{2}} =
\left\{
\begin{array}{ll}
 \omega_{\pi}(\tau)^n A_e^\psi(M^\ast W) \int_{S_F} f(t^\prime) \, dt^\prime & \text{ if $t_i \in F^\times K_i$}\\
 &\\
 0 & \text{otherwise}
 \end{array}
 \right.
\]
for $t = \mathrm{diag}(1, \dots 1, t_1, \dots, t_{n})$. Here, $S_F = S \cap \mathrm{GL}_{4n}(F)$.
\end{theorem}
Let us explain that this theorem follows from result in \cite{Mo2}.
As in \cite[4.1]{Mo2}, let $\mathcal{Z}$ be the unipotent subgroup of $\mM$ given by
\[
\mathcal{Z} = \{ m \in \mM : m_{i, i} = 1 \, \forall i, m_{i, j} = 0 \text{ if either ($j>i$ and $i+j > 2n$) or ($i>j$ and $i+j \leq 2n+1$)} \}
\]
and let $\psi_{\mathcal{Z}, F}$ be its character 
\[
\psi_{\mathcal{Z}, F}(m) = \psi(m_{1,2} + \cdots + m_{n-1, n} - m_{n+2} - \cdots - m_{2n ,2n-1}).
\]
The group $\varrho(\mathcal{Z})$ stabilizes the character $\psi_{\bar{U}}$.
Let $\mathfrak{E} = \varrho(\mathcal{Z}) \ltimes \bar{U}$.
Also. let $\mathcal{Z}^+ = \mathcal{Z} \cap N_\mM$, $V_\Delta = \mathcal{Z} \cap {}^{t}N_\mM$ and 
\[
N_{\mM, \Delta} = \{ {}^{t}\ell_\mM(X) : {}^{t}X \in J, X_{i, j} = 0 \text{ if $i+j > n+1$} \}.
\]
We have $\mathcal{Z} = \mathcal{Z}^+ \cdot V_\Delta$ and $\bar{R} = V_\Delta \cdot N_{\mM, \Delta}$.

The character $\psi_{\bar{R}}$ is given by $\psi_{\bar{R}}(vn) = \psi_{V_\Delta, F}(v)$, $v \in V_\Delta$,
$n \in N_{\mM, \Delta}$ where $\psi_{v_\Delta, F}(v)$ is defined in \cite[4.1]{Mo2}.
Then the expression in \cite[(10.6)]{Mo2} evaluated at $\varrho(\varepsilon_3) w_U$ is $(f \Delta) \ast E^{\psi}(W_s, t)$,
and its analytic continuation is given by the expression in \cite[(10.5)]{Mo2} which amounts to a finite sum.

Meanwhile, from the definition of $A_e^\psi$ in \eqref{6.1} we have
\[
A_e^\psi(W) =\int_{\varrho(V_\mM^\ast) \backslash \gamma V_- \gamma^{-1}} W(v \gamma \varepsilon_1) 
\psi_{V_-}((\gamma \varepsilon_1) v \gamma \varepsilon_1)^{-1} \, dv.
\]
We can integrate over the group $\bar{U}^{\neg}$ in \cite[(10.2)]{Mo2}, consisting of elements $\bar{u} \in \bar{U}$
such that $\bar{u}_{2n+i, j} = $ whenever $i \geq n$ and $j \leq n$.
Then the character $\bar{u} \mapsto \psi_{V_-}((\gamma \varepsilon_1)^{-1} \bar{u} \gamma \varepsilon_1)$
on $\bar{U}^\neg$ is the character denoted by $\hat{\psi}_{\bar{U}^\Delta}$ in the definition of $\mathcal{T}^\prime$ in \cite[10.2]{Mo2}.
Thus, $A_e^\psi(W)$ is $\mathcal{T}^\prime(W)(\gamma \varepsilon_1)$ in the notation of  \cite[10.2]{Mo2}.
The second part amounts to the first statement of \cite[Corollary~10.8]{Mo2} upon taking
\[
\varepsilon_1 = (\hat{w}\gamma)^{-1} \varrho(\varepsilon_3) w_U = \ell_M((-1)^n \mathfrak{d})
\]
where
\[
\hat{w} := \eta(w_{U^\prime}^\prime) \varrho(w_{2n, n}^\prime) = \begin{pmatrix} &I_n&&\\ &&&w_n\\ -w_n&&&\\ &&I_n&\end{pmatrix}
\]
is as in \cite[4.1]{Mo2}.
This theorem is crucial step, which is a unitary analogue of \cite[Theorem~8.1]{LMb}
and proved in \cite{LMIMRN}.
As in  \cite{LMIMRN}, the main input is that the Langlands quotient of $\mathrm{Ind}(\mathbb{W}^{\psi_{N_M}}(\pi), \frac{1}{2})$
admits a realization in $C^{\rm sm}(H \backslash G)$
where $H = G_{2n} \cap \mathrm{GL}_{4n}(F) = \mathrm{Sp}_{4n}(F)$.
We refer to \cite{Mo2} for more details.
\begin{corollary}
\label{cor8.2}
Suppose that $\pi \in \mathrm{Irr}_{\rm ut, temp}\, \mM$.
Then for any $W \in \mathrm{Ind}(\mathbb{W}^{\psi_{N_M}}(\pi))_{\#}^\circ$, $W^\wedge \in \mathrm{Ind}(\mathbb{W}^{\psi_{N_M}^{-1}}(\pi))_{\natural}^\circ$ 
we have
\begin{equation}
\label{8.4}
B(W, M(\frac{1}{2})W^\wedge, \frac{1}{2}) =\omega_{\pi}(\tau)^n A_e^\psi(M^\ast W) 
\int_{S_F}E^{\psi^{-1}}(W_{\frac{1}{2}}^\wedge, t) \frac{dt}{|\det t|}
\end{equation}
where recall $S_F = S \cap \mathrm{GL}_{4n}(F)$.
\end{corollary}
\begin{proof}
By Lemma~\ref{Lemma 7.9}, for any $W^\wedge \in \mathrm{Ind}(\mathbb{W}^{\psi_{N_M}^{-1}}(\pi))_{\natural}^\circ$
there exist $K_i \in \mathcal{CSGR}(E^\times)$ ($1 \leq i \leq n$) such that $E^{\psi^{-1}}(W_s^\wedge, \cdot) \in C(S)^{K_0}$
for all $s$ and $E^{\psi^{-1}}(W_s^\wedge, \cdot)$ is compactly supported on $S$ uniformly in $s$.
Here, we regard $K_i$ as a subgroup of $S$ and we put $K_0 = \prod_{i=1}^n K_i$.
Suppose $f := f_{K_1} \otimes \cdots f_{K_n} \in C_c^\infty(S)$ and let $f^\vee(t) := f(t^{-1})$.
By Proposition~\ref{Proposition 7.11} for $-\mathrm{Re}\, s \gg 1$ and $W \in \mathrm{Ind}(\mathbb{W}^{\psi_{N_M}}(\pi))_{\#}^\circ$
we have
\begin{equation}
\label{8.5}
B(W, M(s)W^\wedge, s) \int_{S}f(t) \, dt
= \int_{S} E^{\psi}(M_s^\ast W, t) f^\vee \ast E^{\psi^{-1}}(W_{s}^\wedge, t)\frac{dt}{|\det t|}
=\int_{S} f \ast E^{\psi}(M_s^\ast W, t) E^{\psi^{-1}}(W_{s}^\wedge, t)\frac{dt}{|\det t|}.
\end{equation}
Since $B(W, W^\wedge, s)$ is entire function of $s$ for $W \in \mathrm{Ind}(\mathbb{W}^{\psi_{N_M}}(\pi))_{\#}^\circ$, 
$W^\vee \in \mathrm{Ind}(\mathbb{W}^{\psi_{N_M}^{-1}}(\pi))$,
the first part of Theorem~\ref{Theorem 8.1} implies that both sides of \eqref{8.5} are 
meromorphic functions and the identity holds whenever $M(s)$ is holomorphic.
Then by \cite[Proposition~2.1]{LMa}, we may specialize $s = \frac{1}{2}$.
Using the second part of Theorem~\ref{Theorem 8.1}, we find that the right-hand side of \eqref{8.5} is equal to
\begin{multline*}
\omega_{\pi}(\tau)^n A_e^\psi(M^\ast W) 
\left( \int_{S_F} f(t) \, dt \right) \cdot 
\int_{S \cap K_0^\prime} E^{\psi^{-1}}(W_{\frac{1}{2}}^\wedge, t) \frac{dt}{|\det t|}
\\
=
 \omega_{\pi}(\tau)^n A_e^\psi(M^\ast W) 
\left( \int_{S_F} f(t) \, dt \right) \cdot \left( \int_{K_0 \cap \mathrm{GL}_{4n}(F) \backslash K_0}  \, dt \right) 
\int_{S_F} E^{\psi^{-1}}(W_{\frac{1}{2}}^\wedge, t) \frac{dt}{|\det t|}
\end{multline*}
where $K_0^\prime = \prod_{i=1}^n (F^\times K_i)$.
The required formula readily follows since 
\[
\left( \int_{S_F} f(t) \, dt \right) \cdot \left( \int_{K_0 \cap \mathrm{GL}_{4n}(F) \backslash K_0}  \, dt \right) 
= \int_{S} f(t) \, dt.
\]
\end{proof}
%and define $\psi_{\mathcal{Z}}$ of $\mathcal{Z}$ as follows.
%When $n$ is even, we define $\psi_{\mathcal{Z}}(m)$ by
%\begin{multline*}
%\psi(\tau m_{1,2}+ \tau^{-1} m_{2,3}+ \tau m_{3,4} +\cdots + \tau^{-1} m_{n-2, n-1} + \tau m_{n-1, n} \\
%+ \tau m_{n+2, n+1} + \tau^{-1} m_{n+3, n+2} 
%+ \tau m_{n+4, n+3} + \cdots + \tau^{-1} m_{2n-1, 2n-2}
% + \tau m_{2n, 2n-1} ).
%\end{multline*}
%When $n$ is odd, we define $\psi_{\mathcal{Z}}(m)$ by
%\[
%\psi(\tau m_{1,2}+ \tau^{-1} m_{2,3}+ \cdots + \tau m_{n-2, n-1} + \tau^{-1} m_{n-1, n} \\
%+\tau m_{n+2, n+1} + \tau^{-1} m_{n+3, n+2} 
%+\cdots + \tau m_{2n-1, 2n-2}+ \tau^{-1} m_{2n, 2n-1}).
%\]
\subsection{}
It remains to compute the integral on the right-hand side of \eqref{8.4}.
Again this is essentially done in \cite{Mo2} and it relies heavily on the fact that $\pi \in \mathrm{Irr}_{\rm ut} \mM$.

Let 
\[
\mathbb{W}^{\psi_{N_\mM}}(\pi)_{\natural \natural}
= \{ W \in \mathbb{W}^{\psi_{N_\mM}}(\pi) : W |_{\mathcal{P}^\ast} \in C_{c}^\infty(N_\mM \backslash \mathcal{P}^\ast, \psi_{N_\mM})
\text{ and } W|_{\eta_{\mM}^\vee(T_{\mM^\prime}^\prime) \ltimes \mathcal{Z}} \in 
C_c^\infty(\mathcal{Z}^+ \backslash \eta_{\mM}^\vee(T_{\mM^\prime}^\prime) \ltimes \mathcal{Z}, \psi_{\mathcal{Z}}).
\]
\begin{theorem}
Let $\pi \in \mathrm{Irr}_{\rm ut, temp} \mM$. Then
\begin{enumerate}
\item (see \cite[Corollary~4.1]{OO} and \cite[Lemma~3.5]{Mo2}.) The integral 
\[
\mathfrak{P}^{H_\mM}(W) := \int_{(H_\mM \cap N_\mM) \backslash H_\mM \cap \mathcal{P}} W(\tau^\ast p) \, dp
\]
converges and defines a non-zero $H_\mM$-invariant functional on $\mathbb{W}^{\psi_{N_\mM}}(\pi)$.
Here, we define $H_\mM = \mathrm{GL}_{2n}(F)$ and 
$\tau^\ast = \mathrm{diag}(\tau, 1, \tau, 1, \dots ,\tau, 1) \in \mM$.
%%%%%%%
%
%
%
%%%%%%%
\item (\cite[Corollary~10.4]{Mo2})
Put
\[
x(n) = \left\{
 \begin{array}{ll}
n & \text{ if $n$ is even,}\\
&\\
0 & \text{ if $n$ is odd,}
\end{array}
\right.
\]
which is $b^{\prime \prime}(n)$ in the notation of \cite{Mo2}.
Then for any $W \in \mathbb{W}^{\psi_{N_\mM}}(\pi)_{\natural \natural}$ we have
\[
\int_{\mathcal{Z}^+ \backslash \mathcal{Z}} \int_{S_F}
\Delta(t)^{-1} |\det t|^{n-\frac{1}{2}} W(tr) \psi_{\mathcal{Z}, F}(r)^{-1} \, dt \, dr
=\omega_{\pi}(\tau)^n |\tau|^{x(n)} \int_{\mathcal{Z} \cap H_\mM \backslash \mathcal{Z}} \mathfrak{P}^{H_\mM}(\pi(n\tau_\circ ) W) \psi_{\mathcal{Z}}(n)^{-1} \, dn.
\]
where $\tau_\circ \in \mM$ is defined by $\varrho(\tau_\circ) = \hat{w} \varrho(\tau^\ast)^{-1} \hat{w}^{-1}$
and we define $\psi_{\mathcal{Z}}(m) = \psi_{\mathcal{Z}, F}(\tau_\circ^{-1} m \tau_\circ)$.
%%%%%%%
%
%
%
%%%%%%%
\item (\cite[Lemma~4.2]{Mo2})
The integral 
\[
L_W(g):= \int_{(P \cap H) \backslash H} \int_{(H_\mM \cap N_\mM) \backslash H_\mM \cap \mathcal{P}} W(\varrho(\tau^\ast p) hg) |\det p|^{-(n+\frac{1}{2})} \, dp
= \int_{H \cap \bar{U}} \mathfrak{P}^{H_\mM}( (\delta_P^{-\frac{1}{2}}I(\frac{1}{2}, \bar{u}g)W) \circ \varrho) d \bar{u}
\]
converges for any $W \in \mathrm{Ind}(\mathbb{W}^{\psi_{N_M}}(\pi), \frac{1}{2})$ and defines an intertwining map
\[
\mathrm{Ind}(\mathbb{W}^{\psi_{N_M}}(\pi), \frac{1}{2})
\rightarrow C^{\rm sm}(H \backslash G).
\]
%%%%%%%
%
%
%
%%%%%%%
\item (second statement of \cite[Corollary~10.8]{Mo2}) 
Put
\[
y(n) =  \left\{
 \begin{array}{ll}
n & \text{ if $n$ is even,}\\
&\\
n-\frac{1}{2} & \text{ if $n$ is odd,}
\end{array}
\right.
\]
which is $a(n)+n^2$ in the notation of \cite{Mo2}.
Then we have
\begin{equation}
\label{8.6}
A_e^\psi(M^\ast W)
= \omega_{\pi}(\tau) |\tau|^{y(n)} ( \int_{N_{\mM, \Delta}} \int_{H \cap \mathfrak{E} \backslash \mathfrak{E}} L_W(v  \varrho(\tau_\circ \varepsilon_4 u \varepsilon_3) w_U) \psi_{\mathfrak{E}}^{-1}(v) \, dv) \, du
\end{equation}
where we define
\[
\psi_{\mathfrak{E}}(\varrho(m) \bar{u}) = \psi_{\mathcal{Z}, F}(\tau_\circ^{-1} m \tau_\circ) \psi_{\bar{U}}^{-1}(\tau_\circ^{-1} \bar{u} \tau_\circ)
=\psi_{\mathcal{Z}}(m) \psi_{\bar{U}}^{-1}(\tau_\circ^{-1} \bar{u} \tau_\circ).
\]
\end{enumerate}
\end{theorem}
%%%%%%%
%
%
%
%%%%%%%
\begin{corollary}
\label{cor8.5}
Let $\pi \in \mathrm{Irr}_{\rm ut, temp} \mM$.
Then for any $W \in \mathrm{Ind}(\mathbb{W}^{\psi_{N_M}}(\pi))_\natural^\circ$ we have
\begin{equation}
\label{8.7}
\int_{S_F} E^{\psi}(W_{\frac{1}{2}}, t) \, \frac{dt}{|\det t|} = \omega_{\pi}(\tau)^{n+1} A_e^\psi(M^\ast W).
\end{equation}
\end{corollary}
\begin{proof}
We may assume without loss of generality that 
\begin{equation}
\label{8.8}
W_{\frac{1}{2}}(u^\prime \varrho(m) w_U u) = W^\mM(m) |\det m|^{\frac{1}{2}} \delta_P(\varrho(m))^{\frac{1}{2}} \phi(u),
\quad m \in \mM, \, u, u^\prime \in U
\end{equation}
with $W^\mM \in \mathbb{W}^{\psi_{N_M}}(\pi)_{\natural}$ and $\phi \in C_c^\infty(U)$.
We evaluate the left-hand side $I$ of \eqref{8.7} using the last expression in \eqref{7.7}
(where we recall that the integrand is compactly supported by Lemma~\ref{Lemma 7.9}).
Thus,
\[
I = I^\prime \int_{U}\phi(u) \psi_{U}^{-1}(v) \, dv
\]
where 
\[
I^\prime = \int_{S_F} \int_{\bar{R}} \Delta(t)^{-1} |\det t|^{n-\frac{1}{2}} W^\mM(t \varepsilon_4 r \varepsilon_3) \psi_{\bar{R}}(r) \, dr \, dt.
\]
The integrand in $I^\prime$ is compactly supported because $W^\mM \in \mathbb{W}^{\psi_{N_\mM}}(\pi)_{\natural}$.
Note for $r \in V_\Delta = \mathcal{Z} \cap \bar{R}$, $\psi_{\bar{R}}(r)^{-1} = \psi_{\mathcal{Z}, F}(r)$.
We write 
\begin{multline*}
I^\prime = \int_{N_{\mM, \Delta}} ( \int_{V_\Delta} \int_{S_F} \Delta(t)^{-1} |\det t|^{n-\frac{1}{2}} W^\mM(t \varepsilon_4 ru \varepsilon_3) \psi_{\mathcal{Z}, F}(r)^{-1}
\, dt \, dr) \, du\\
= 
\int_{N_{\mM, \Delta}} ( \int_{\mathcal{Z}^+ \backslash \mathcal{Z}} \int_{S_F} \Delta(t)^{-1} |\det t|^{n-\frac{1}{2}} W^\mM(t \varepsilon_4 ru \varepsilon_3) \psi_{\mathcal{Z}, F}(r)^{-1}
\, dt \, dr) \, du\\
=
\int_{N_{\mM, \Delta}} ( \int_{\mathcal{Z}^+ \backslash \mathcal{Z}} \int_{S_F} \Delta(t)^{-1} |\det t|^{n-\frac{1}{2}} W^\mM(tr \varepsilon_4 u \varepsilon_3) \psi_{\mathcal{Z}, F}(r)^{-1}
\, dt \, dr) \, du
\end{multline*}
since $\varepsilon_4$ stabilizes $\psi_{\mathcal{Z}, F}$.
For the double integral in the brackets we apply part 2 of the theorem above to $\pi(\varepsilon_4 u \varepsilon_3) W^\mM$ 
(which is applicable since $W^\mM \in \mathbb{W}^{\psi_{N_\mM}}(\pi)_{\natural}$).
We get
\[
I^\prime = \omega_{\pi}(\tau)^n |\tau|^{x(n)} \int_{N_{\mM, \Delta}}
( \int_{\mathcal{Z} \cap H_\mM \backslash \mathcal{Z}} 
\mathfrak{P}^{H_\mM}(\pi(n\tau_\circ  \varepsilon_4 u \varepsilon_3 ) W^\mM) \psi_{\mathcal{Z}}(n)^{-1} \, dn )\, du
\]
Thus,
\[
I = 
 \omega_{\pi}(\tau)^n |\tau|^{x(n)} \int_{N_{\mM, \Delta}}
( \int_{\mathcal{Z} \cap H_\mM \backslash \mathcal{Z}} \int_U
\mathfrak{P}^{H_\mM}(\pi(n\tau_\circ  \varepsilon_4 u \varepsilon_3 ) W^\mM) \psi_{\mathcal{Z}}(n)^{-1} \phi(v) \psi_{U}(v)^{-1}\, dv \, dn )\, du
\]
From \eqref{8.8},
\[
(\delta_P^{-\frac{1}{2}} I(\frac{1}{2}, \varrho(m) w_U v) W) \circ \varrho
= \phi(v) \delta_P^{\frac{1}{2}}(\varrho(m)) |\det m|^{\frac{1}{2}} \pi (m) W^\mM
\]
for any $v \in U$ and $m \in \mM$.
Thus, $I$ equals
\[
 \omega_{\pi}(\tau)^n |\tau|^{x(n)}  \delta_P^{-\frac{1}{2}}(\varrho(\tau_\circ)) |\det \tau_\circ|^{-\frac{1}{2}}
 \int_{N_{\mM, \Delta}}
( \int_{\mathcal{Z} \cap H_\mM \backslash \mathcal{Z}} \int_U
\mathfrak{P}^{H_\mM}(\delta_P^{-\frac{1}{2}} I(\pi, \varrho(n\tau_\circ  \varepsilon_4 u \varepsilon_3) w_U v) W) \psi_{\mathcal{Z}}(n)^{-1} \psi_{U}(v)^{-1}\, dv \, dn )\, du
\]
Since $\varepsilon_3^{-1} N_{\mM, \Delta} \varepsilon_3 \subset N_\mM^\flat$, the group $\varrho(\varepsilon_3^{-1} N_{\mM, \Delta} \varepsilon_3)$
stabilizes the character $\psi_U(w_U^{-1} \cdot w_U)$ on $\bar{U}$.
Making a change of variable 
\[
v \mapsto (\varrho (n\tau_\circ  \varepsilon_4 u \varepsilon_3) w_U )^{-1} \bar{v} \varrho (n \tau_\circ  \varepsilon_4 u \varepsilon_3) w_U
\]
on $U$ we obtain
\begin{multline*}
I=
 \omega_{\pi}(\tau)^n |\tau|^{x(n)}  \delta_P^{\frac{1}{2}}(\varrho(\tau_\circ)) |\det \tau_\circ|^{-\frac{1}{2}}
 \int_{N_{\mM, \Delta}}
( \int_{\varrho(\mathcal{Z} \cap H_\mM) \backslash \mathfrak{E}}
\mathfrak{P}^{H_\mM}(\delta_P^{-\frac{1}{2}} I(\pi, v \varrho(\tau_\circ  \varepsilon_4 u \varepsilon_3) w_U ) W) \psi_{\mathfrak{E}}(n)^{-1} \, dv  \, du
\\
=
 \omega_{\pi}(\tau)^n |\tau|^{x(n)}  \delta_P^{\frac{1}{2}}(\varrho(\tau_\circ)) |\det \tau_\circ|^{-\frac{1}{2}}
 \int_{N_{\mM, \Delta}}
( \int_{\mathfrak{E} \cap H) \backslash \mathfrak{E}}
\int_{H \cap \bar{U}}
\mathfrak{P}^{H_\mM}(\delta_P^{-\frac{1}{2}} I(\pi, x v \varrho(\tau_\circ  \varepsilon_4 u \varepsilon_3) w_U ) W) \psi_{\mathfrak{E}}(n)^{-1} \,dx  \,dv  \, du
\end{multline*}
From part 3 of the theorem, we get
\[
I = 
 \omega_{\pi}(\tau)^n |\tau|^{x(n)}  \delta_P^{\frac{1}{2}}(\varrho(\tau_\circ)) |\det \tau_\circ|^{-\frac{1}{2}}
 \int_{N_{\mM, \Delta}}
( \int_{\mathfrak{E} \cap H) \backslash \mathfrak{E}}
L_W(v \varrho(\tau_\circ  \varepsilon_4 u \varepsilon_3) w_U )\psi_{\mathfrak{E}}(n)^{-1} \,dx  \,dv  \, du.
\]
From the last part of the theorem, this is equal to
\[
 \omega_{\pi}(\tau)^{n+1}  |\tau|^{x(n)-y(n)}  \delta_P^{\frac{1}{2}}(\varrho(\tau_\circ)) |\det \tau_\circ|^{-\frac{1}{2}}
 A_e^\psi (M^\ast W)
 =  \omega_{\pi}(\tau)^{n+1} A_e^\psi (M^\ast W),
\]
where we used the fact $ |\tau|^{x(n)-y(n)}  \delta_P^{\frac{1}{2}}(\varrho(\tau_\circ)) |\det \tau_\circ|^{-\frac{1}{2}} =1$
(see \cite[Remark~10.1]{Mo2}).
\end{proof}
Let us complete a proof of Proposition~\ref{Proposition 6.15}.
From Corollaries~\ref{cor8.2} and \ref{cor8.5}, we find that \eqref{6-12} holds for any $W \in \mathrm{Ind}(\mathbb{W}^{\psi_{N_M}}(\pi))_{\#}^\circ$ and 
$W^\wedge \in \mathrm{Ind}(\mathbb{W}^{\psi_{N_M}}(\pi))_\natural^\circ$, namely
\[
B\left(W, M(\frac{1}{2})W^\wedge, \frac{1}{2} \right) = \omega_\pi(\tau) A_{e}^{\psi}(M^\ast W) A_e^{\psi^{-1}}(M^\ast W^\wedge).
\]
On the other hand, as in \cite[Section~8C]{LMb}, in a similar argument of the proof of \cite[Lemma~6.13]{LMb} with \eqref{7.7},
we see that the linear map  $\mathrm{Ind}(\mathbb{W}^{\psi_{N_M}}(\pi))_\natural^\circ \rightarrow C_c^\infty(T_{\mM^\prime}^\prime)$
given by $W^\wedge \mapsto E^{\psi}(W_{\frac{1}{2}}^\wedge, \cdot) $ is onto.
Therefore by Corollary~\ref{cor8.5}, the linear form $ A_e^{\psi^{-1}}(M^\ast W^\wedge)$ does not vanish on $\mathrm{Ind}(\mathbb{W}^{\psi_{N_M}}(\pi))_\natural^\circ$.
From Corollary~\ref{cor6.14}, we conclude that  \eqref{6-12} holds for all $W^\wedge$, which complete a proof of Proposition~\ref{Proposition 6.15}.
%%%%%%%%%%%%%%%%%%%%%%%%%%%%%%%%%%%%%%%%%%
%
%
%
%
%
%
%
%
%
%
%
%
%
%
%%%%%%%%%%%%%%%%%%%%%%%%%%%%%%%%%%%%%%%%%%%%%%%
\section{Refined formal degree conjecture and Conjecture~\ref{LM local conjecture}}
\subsection{Non-archimedean case}
In this section, we show an equivalence between a refined formal degree conjecture
and Theorem~\ref{main thm} for discrete series representations of $G^\prime = \mathrm{U}_{2n}$. 
In particular, we obtain a proof of refined formal degree conjecture in our case.
Suppose that a base field $F$ is a non-archimedean local field of characteristic zero.

Let us recall refined formal degree conjecture by Gross-Reeder~\cite{GR} and Gan--Ichino~\cite[14.5]{GI}.
For an irreducible discrete series representation of $G^\prime$, the formal degree $d_\sigma$ for $\sigma$ is the measure 
on $G^\prime$ satisfying 
\[
\int_{G^\prime}(\sigma(g)v_1, v_1^\prime)(\sigma(g^{-1})v_2, v_2^\prime) d_\sigma g = (v_1, v_2^\prime)(v_2, v_1^\prime)
\]
holds for any $v_1, v_1^\prime, v_2, v_2^\prime$ where $(\cdot, \cdot)$ is a $G^\prime$-invariant inner product
on $V_\sigma$.
On the other hand, we denote the measure on $G^\prime$ defined in Section~\ref{measures} by $d_\psi$ in order to clarify the difference of these measures. 

Hiraga, Ichino and Ikeda~\cite{HII} formulated a conjecture on a relationship between these measures $d_\sigma$ and $d_\psi$ in terms of absolute values of special 
values of adjoint gamma factors. Recently, Gan--Ichino~\cite{GI} computed the sign of this special value for Seinberg representation of classical groups
and and using an important observation by Gross--Reeder~\cite{GR} on the sign, Gan--Ichino~\cite[Section~14.5]{GI} conjectured
a refinement of  \cite[Conjecture~1.4]{HII}. Indeed, we prove this refined version.
\begin{theorem}
\label{formal deg thm}
Let $\pi$ be an irreducible representation of $\mathrm{GL}_{2n}(E)$ of the form 
$\pi = \tau_1 \times \cdots \times \tau_k$ where $\tau_i$ are mutually inequivalent irreducible discrete series representations
of $\mathrm{GL}_{n_i}(E)$ such that $n=n_1+\cdots + n_k$ and $L(s, \tau_i, \mathrm{As}^+)$ has a pole at $s=0$.
Write $\sigma = \mathcal{D}_{\psi^{-1}}^{\Upsilon^{-1}}(\mc(\pi))$, which is an irreducible generic discrete series representation of $G^\prime$ 
(See Theorem~\ref{descent BC local}).
Then the formal degree conjecture
\begin{equation}
\label{formal degree identity}
d_\psi =2^k
\lambda(E \slash F, \psi)^n \omega_{\sigma}(-1) \gamma(1, \mathfrak{c}(\pi), \mathrm{As}^-, \psi) d_\sigma.
\end{equation}
holds for in this case.
Here, $\lambda(E \slash F, \psi)$ is the Langlands' $\lambda$-function and the $\gamma$-factor is defined by Langlands-Shahidi method~\cite{Sh90b}.
\end{theorem}
We shall prove this theorem in a very similar argument as the proof of refined formal degree conjecture for metaplectic group by Ichino--Lapid--Mao~\cite{ILM}.

Recall the following functional equation.
Let $W \in \mathbb{W}^{\psi_{N^\prime}^{-1}}(\sigma)$ and $W^\prime \in \mathrm{Ind}(\mathbb{W}^{\psi_{N_M}}(\pi))$.
Let $\gamma(s, \pi, \mathrm{As}^{\pm}, \psi)$ be Asai gamma factors defined by  Langlands--Shahidi method~\cite{Sh90b}.
Then in \cite{Mo3}, we proved that 
\begin{equation}
\label{FQ}
J(W, M(s)W^\prime, -s)
=\lambda(E \slash F, \psi)^{n} \frac{\gamma(s+\frac{1}{2}, \sigma \times (\pi \otimes \Upsilon^{-1}), \psi)}{\gamma(2s, \mc(\pi), \mathrm{As}^+, \psi)} 
J(W, W^\prime,  s)
\end{equation}
and that the the gamma factor $\gamma(s+\frac{1}{2}, \sigma \times (\pi \otimes \Upsilon^{-1}), \psi)$ coincides with the gamma factor 
defined by Langlands--Shahidi method~\cite{Sh90b}. 
We note that the factor $\lambda(E \slash F, \psi)$ comes from the normalization of an intertwining operator (for example, see Anandavardhanan~\cite[(3)]{Ana}).

By Theorem~\ref{descent BC local}, we have
\[
\gamma \left(s+\frac{1}{2}, \sigma \times (\pi \otimes \Upsilon^{-1}), \psi \right)
= \gamma \left(s+\frac{1}{2}, \pi \otimes \mc(\pi), \psi \right).
\]
Moreover, we have
\[
\gamma \left(s+\frac{1}{2}, \pi \otimes \mc(\pi), \psi \right)
= \gamma \left(s+\frac{1}{2}, \mc(\pi), \mathrm{As}^+,\psi \right)\gamma \left(s+\frac{1}{2}, \mc(\pi), \mathrm{As}^-,\psi \right).
\]
Hence, we get
\begin{equation}
\label{ILM 3.6}
\lim_{s \rightarrow \frac{1}{2}} \frac{\gamma(s+\frac{1}{2}, \sigma \times (\pi \otimes \Upsilon^{-1}), \psi)}{\gamma(2s, \mc(\pi), \mathrm{As}^+, \psi)} 
=2^k \gamma  \left(1, \mc(\pi), \mathrm{As}^-,\psi \right)
\end{equation}
We define a non-degenerate $G^\prime$-invariant bilinear form $(\cdot, \cdot)_\sigma$ on 
$\mathbb{W}^{\psi_{N^\prime}^{-1}}(\sigma) \times \mathbb{W}^{\psi_{N^\prime}}(\sigma^\vee)$ by
\[
(W, W^\prime)_\sigma = \int_{N \backslash G} W(g) W^\prime(g) d_\psi(g)
\]
which converges absolutely by \cite[Proposition~3.2]{Del}.
On the other hand, recall that we defined a $G^\prime$-invariant bilinear form $[\cdot, \cdot]_\sigma$ in Definition~\ref{def inner sigma}.
Indeed, it is defined so that
\[
J \left(W, W^\prime, \frac{1}{2} \right) = \left[W, A_{\#}^{\psi, \Upsilon}(W^\prime, \cdot ) \right]_\sigma
\]
for $W \in \mathbb{W}^{\psi_{N^\prime}^{-1}}(\sigma)$ and $W^\prime \in \mathrm{Ind}(\mathbb{W}^{\psi_{N_M}}(\pi))$.
Then by the functional equation \eqref{FQ}, we obtain
\begin{align*}
(W, A_{\#}^{\psi, \Upsilon}(M^\ast W^\prime, \cdot ))_\sigma &= J \left(W, M^\ast W^\prime, -\frac{1}{2} \right)\\
&= 
\lambda(E \slash F, \psi)^{n} \lim_{s \rightarrow \frac{1}{2}}
 \frac{\gamma(s+\frac{1}{2}, \sigma \times (\pi \otimes \Upsilon^{-1}), \psi)}{\gamma(2s, \mc(\pi), \mathrm{As}^+, \psi)} J(W, W^\prime,  \frac{1}{2})
\\
&=
\lambda(E \slash F, \psi)^{n} \lim_{s \rightarrow \frac{1}{2}}
 \frac{\gamma(s+\frac{1}{2}, \pi \otimes \mc(\pi), \psi)}{\gamma(2s, \mc(\pi), \mathrm{As}^+, \psi)} 
[W, A_{\#}^{\psi, \Upsilon}(M^\ast W^\prime, \cdot )]_\sigma
\end{align*}
Hence, by \eqref{ILM 3.6},
\begin{equation}
\label{ILM 3.8}
(W, A_{\#}^{\psi, \Upsilon}(M^\ast W^\prime, \cdot ))_\sigma
=
\lambda(E \slash F, \psi)^{n} 2^k 
\gamma(1, \mc(\pi), \mathrm{As}^-, \psi)
[W, A_{\#}^{\psi, \Upsilon}(M^\ast W^\prime, \cdot )]_\sigma
\end{equation}
Recall that the formal degree $d_\sigma$ is defined so that 
\[
\int_{G^\prime} [\sigma(g)W_1, W_1^\prime]_\sigma [\sigma(g^{-1})W_2, W_2^\prime]_\sigma d_\sigma
= [W_1, W_2^\prime]_\sigma [W_2, W_1^\prime] 
\]
for $W_1, W_2 \in \mathbb{W}^{\psi_{N^\prime}^{-1}}(\sigma)$
and $W_1^\prime, W_2^\prime \in \mathbb{W}^{\psi_{N^\prime}}(\mc(\sigma))$.
Assume that $W_1^\prime = A_{\#}^{\psi, \Upsilon}(M^\ast W^\prime, \cdot )$.
Then by the definition,
\begin{equation}
\label{double int}
\int_{G^\prime} [\sigma(g)W_1, W_1^\prime]_\sigma [\sigma(g^{-1})W_2, W_2^\prime] d_\psi(g)
=\int_{G^\prime} \int_{N^\prime \backslash G^\prime} W_1(xg) A_{\#}^{\psi, \Upsilon}(M^\ast W^\prime, x)
 [\sigma(g^{-1})W_2, W_2^\prime]_\sigma d_\psi(x) d_\psi(g).
\end{equation}
In the same argument as \cite[p.1316--1317]{ILM}, we see that this double integral converges absolutely by Lemma~\ref{lemma4.6}.
Then we can change the order of the integration, and changing the variable $g \mapsto x^{-1}g$, we get
\begin{multline*}
 \int_{N^\prime \backslash G^\prime} \int_{G^\prime} W_1(g) A_{\#}^{\psi, \Upsilon}(M^\ast W^\prime, x)
 [\sigma(g^{-1}x)W_2, W_2^\prime]_\sigma d_\psi(x) d_\psi(g)\\
 =
 \int_{N^\prime \backslash G^\prime} \int_{N^\prime \backslash G^\prime} \int_{N^\prime} 
\psi_{N^\prime}(u)^{-1} W_1(g) A_{\#}^{\psi, \Upsilon}(M^\ast W^\prime, x)
 [\sigma(g^{-1}u^{-1}x)W_2, W_2^\prime]_\sigma d_\psi(x) \, d_\psi(u) \, d_\psi(g)
 \\
=
 \int_{N^\prime \backslash G^\prime} \int_{N^\prime \backslash G^\prime} \int_{N^\prime} 
\psi_{N^\prime}(u) W_1(g) A_{\#}^{\psi, \Upsilon}(M^\ast W^\prime, x)
 [\sigma(ux)W_2, \sigma^\vee(g) W_2^\prime]_\sigma d_\psi(x) \, d_\psi(u) \, d_\psi(g)
\end{multline*}
By \eqref{ILM 3.7} and Corollary~\ref{Cor3.6}, this is equal to
\[
\omega_\sigma(-1)
 \int_{N^\prime \backslash G^\prime} \int_{N^\prime \backslash G^\prime}
W_1(g) A_{\#}^{\psi, \Upsilon}(M^\ast W^\prime, x)
W_2(x) W_2^\prime(g) d_\psi(x)  \, d_\psi(g)
=\omega_\sigma(-1) \cdot (W_1, W_2^\prime)_\sigma [W_2, W_1^\prime]_\sigma
\]
Therefore, Theorem~\ref{formal deg thm} follows from the above computation and \eqref{ILM 3.8}.
\subsubsection{General case}
Let $\mathrm{U}_{2n}^-$ be the even unitary group over $F$ whose discriminant is different from that of $\mathrm{U}_{2n}$
and dimension is $2n$. For the convenience, we write $\mathrm{U}_{2n}^{+} = \mathrm{U}_{2n}$.
In this section, we prove \eqref{formal degree identity} for any discrete series representations of $\mathrm{U}_{2n}^{\pm}$
assuming local Langlands conjecture for these groups.
Indeed, the local Langlands conjecture was established by Mok~\cite{Mok} for $\mathrm{U}_{2n}^+$ 
and Kaletha--Minguez--Shin--White~\cite{KMSW} for $\mathrm{U}_{2n}^{-}$ with the stabilization of the twisted trace 
formula established by Moeglin--Waldspurge~\cite{WMW1,WMW2}
assuming the weighted fundamental lemma for quasi-split groups,
which is proved in Chaudouard--Laumon~\cite{CL}  only in the split case.

Let us briefly recall the local Langlands conjectures.
Fix a splitting such that it gives Whittaker data $(B, \psi_{N^\prime})$
and we denote the $L$-group of $\mathrm{U}_{2n}^{\pm}$ by ${}^{L} \mathrm{U}_{2n} = \mathrm{GL}_{2n}(\mC) \rtimes \mathrm{Gal}(E \slash F)$
with the action $\theta \in \mathrm{Gal}(E \slash F)$ on $\mathrm{GL}_{2n}(\mC)$ defined by $\theta(g) = w_0^{\prime} {}^{t}g^{-1} (w_0^\prime)^{-1}$.
We note that this action preserves our splitting.
We also denote the connected component of ${}^{L} \mathrm{U}_{2n} $ by $\widehat{\mathrm{U}}_{2n}$.
The local Langlands conjecture for $\mathrm{U}_{2n}^{+}$ asserts that there exists a partition
\[
\mathrm{Irr}_{\rm sqr} \mathrm{U}_{2n}^{+} = \coprod \Pi_{\phi}
\]
into $L$-packets, where the disjoint union on the right-hand side runs over conjugacy classes of square integrable $L$-parameters
$\phi : W\!D_F \rightarrow {}^{L}\mathrm{U}_{2n}$.
Here, we say that a continuous homomorphism $\phi :W\!D_F \rightarrow {}^{L}\mathrm{U}_{2n}$ is  an $L$-parameter if $\phi$
is semisimple and $\phi|_{\mathrm{SL}(2, \mC)}$ is algebraic, and that $\phi$ is square-integrable if the centralizer $S_\phi$ of the image of $\phi$
in $\widehat{\mathrm{U}}_{2n}$ is finite, in which $S_\phi$ is an elementary abelian $2$-group.
Moreover, denoting by $\widehat{S}_\phi$ the group of characters of $S_\phi$, there exists an injection 
\[
\Phi_\phi \rightarrow \widehat{S}_\phi, \quad \sigma \mapsto \langle \cdot, \sigma \rangle,
\]
whose image consists of the characters trivial on $\pm I_{2n}$ and which satisfies the suitable character identity.
Furthermore, the Langlands conjecture for even unitary groups asserts that 
 there exists a partition
\[
\mathrm{Irr}_{\rm sqr} \mathrm{U}_{2n}^+ \amalg \mathrm{Irr}_{\rm sqr} \mathrm{U}_{2n}^- = \coprod \Pi_{\phi}
\]
indexed by equivalence classes of square-integrable $L$-parameters $\phi : W\!D_F \rightarrow {}^{L}\mathrm{U}_{2n}$
under the conjugation by $\widehat{\mathrm{U}}_{2n}$, and for each such $\phi$ a bijection 
\[
\Phi_\phi \rightarrow \widehat{S}_\phi, \quad \sigma \mapsto \langle \cdot, \sigma \rangle,
\]
such that 
\[
\Pi_{\phi}^{\pm} := \{ \sigma \in \Pi_\phi : \langle -I_{2n}, \sigma \rangle = \pm 1 \} = \Pi_\phi \cap \mathrm{Irr}_{\rm sqr} \mathrm{U}_{2n}^{\pm}
\]
and the endoscopic character relations hold.
\begin{corollary}
Assume that the local Langlands conjecture holds for $\mathrm{U}_{2n}^{\pm}$. Then
\begin{equation}
\label{formal degree corollary}
d_\psi^{\tilde{G}} =|\mathcal{S}_\phi|
\lambda(E \slash F, \psi)^n \omega_{\sigma}(-1) \gamma(1, \sigma, \mathrm{As}^-, \psi) d_\sigma
\end{equation}
holds for any square-interable $L$-parameter $\phi : WD_F \rightarrow {}^{L} \mathrm{U}_{2n}$
and any $\sigma \in \Pi_\phi$.
\end{corollary}
\begin{proof}
First, we note that in the same argument as \cite[p.1325]{ILM}, we can reduce \eqref{formal degree corollary} in the case of $\mathrm{U}_{2n}^-$
to the case of $\mathrm{U}_{2n}^+$ using the endoscopic relations, \eqref{formal degree corollary} in the case of Steinberg representations in \cite{HII} 
and the proof of \cite[Corollary~9.10]{Sh90b}.
Moreover, in the case of $\mathrm{U}_{2n}^+$, as in the proof of \cite[Corollary~5.1]{ILM}, 
we can show all representation in $\Pi_\phi$ have the same formal degree using the character identity.
Hence, we may assume that  $\langle \cdot, \sigma \rangle$ is trivial, and thus it is generic by \cite[Corollary~9.2.4]{Mok} (See Remark~\ref{atobe}). 
In this case, in a similar way as the proof of \cite[Corollary~9.2.4]{Mok} or \cite[Proposition~8.3.2]{Art},
we can find quadratic extension of number fields $k^\prime \slash k$, a place $v_0$ of $k$, and automorphic representation $\Pi$ of $\mathrm{GL}_{2n}(\mA_{k^\prime})$
and an irreducible globally generic cuspidal automoprhic representation $\Sigma$ of $\mathrm{U}_{2n}(\mA_k)$ such that 
\begin{itemize}
\item $k_{v_0} = F, k^\prime \otimes k_{v_0} = E$
\item $\Pi_{v_0}$ corresponds to the $L$-parameter $\iota \circ \phi$, where $\iota : {}^{L}\mathrm{U}_{2n}$
is the stable base change lift
\item $\Sigma$ weakly lifts to $\Pi$,
\item $\Sigma_{v_0} \in \Pi_\phi$ and $\langle \cdot, \Sigma_{v_0} \rangle$ is trivial, i.e., $\Sigma_{v_0} = \sigma$
\end{itemize}
On the other hand, since the base change lift is strong by \cite{KK}, we obtain
\[
\Pi_{v_0} = \mathrm{BC}(\Sigma_{v_0}).
\]
Hence, by Theorem~\ref{descent BC local}, $\mathrm{BC}(\sigma)$ should correspond to $\iota \circ \phi$.
Then our corollary follows from Theorem~\ref{formal deg thm}.
\end{proof}
\begin{Remark}
\label{atobe}
Atobe~\cite[Theorem~3.1]{At} give a precise proof that $\sigma$ is generic if  $\langle \cdot, \sigma \rangle$ is trivial
using  expected desideratum on the local Langlands conjecture. 
\end{Remark}
\subsection{Archimedean case}
In this section, we prove Conjecture~\ref{LM local conjecture} for discrete series representations of 
$\mathrm{U}_{2n}(\mR)$ as a consequence of the formal degree conjecture.
\begin{lemma}
Let $\phi$ be a square-integrable $L$-parameter of $\mathrm{U}_{2n}(\mR)$.
Then we have
\[
d_\psi =|\mathcal{S}_\phi|
\lambda(E \slash F, \psi)^n \omega_{\sigma}(-1) \gamma(1, \sigma, \mathrm{As}^-, \psi) d_\sigma.
\]
\end{lemma}
\begin{proof}
By \cite[Proposition~2.1]{HII}, we have 
\[
d_\psi =|\mathcal{S}_\phi|
| \gamma(1, \sigma, \mathrm{As}^-, \psi) |d_\sigma.
\]
In this case, we may write $\phi= \oplus \chi_i$ with $\chi_i :W\!D_E \rightarrow \mC^\times$.
Then by a direct computation (or in a similar computation as \cite[Lemma~14.2]{GI}), we see that 
\[
\lambda(E \slash F, \psi)^n \omega_{\sigma}(-1) \gamma(1, \sigma, \mathrm{As}^-, \psi) 
\]
is positive real number with the absolute value $\left| \gamma(1, \sigma, \mathrm{As}^-, \psi) \right|$. Thus, our required identity holds.
\end{proof}
\begin{theorem}
Conjecture~\ref{LM local conjecture} for any generic discrete series representations of $\mathrm{U}_{2n}(\mR)$.
\end{theorem}
\begin{proof}
Suppose $\sigma \in \mathrm{Irr}_{\rm sqr, gen}$.
Then $\mathrm{BC}(\sigma) = \chi_1 \times \cdots \chi_{2n}$ where $\chi_i$ are mutually different unitary characters of $\mC^\times$.
Then we can find irreducible automorphic representations $\Pi_i$ of $\mA_{\mQ(i)}$ such that $\Pi_{i, \infty} = \chi_i$.
Set $\Pi = \Pi_1 \times \cdots \Pi_{2n}$. We see that $\sigma$ is good in the sense of \cite[5.3]{LMa} since $\Pi_{\infty} = \mathrm{BC}(\sigma)$.
In particular, there exists $c_\sigma$ such that
\[
\int_{N^\prime}^{st} J(W, W^\prime, \frac{1}{2}) \psi_{N^\prime}(n) \, dn
=\omega_\sigma(-1) W(e)   A_{\#}^{\psi, \Upsilon}(M^\ast W^\prime, e)
\]
for any $W \in \mathbb{W}^{\psi_{N^\prime}^{-1}}(\sigma)$ and $W^\prime \in \mathrm{Ind}(\mathbb{W}^{\psi_{N_M}}(\pi))$.
Then the same argument as the proof of Theorem~\ref{formal deg thm} gives
\[
c_\pi = \omega_{\sigma}(-1) \Longleftrightarrow
d_\psi =2^k
\lambda(E \slash F, \psi)^n \omega_{\sigma}(-1) \gamma(1, \mathfrak{c}(\pi), \mathrm{As}^-, \psi) d_\sigma,
\]
where we have used the following in the real case.
\begin{enumerate}
\item The integral $J$ converges absolutely uniformly near $s = -\frac{1}{2}$ (cf. Remark~\ref{remark arch}).
\item Any Whittaker function $W \in \mathbb{W}^{\psi_{N^\prime}^{-1}}(\sigma)$ is square-integrable over $N^\prime \backslash G^\prime$
(cf. \cite[Theorem~15.3.4]{Wa2}).
\item Theorem~\ref{backward lifting} $\mathrm{U}_{2n}(\mR)$ readily follows from the Langlands correspondence.
\item We can prove the absolute convergence of the double integral \eqref{double int} in the same argument as \cite[p.1316--1317]{ILM}
using \cite[Theorem~7.2.1]{Wa1}, \cite[Theorem~15.2.4]{Wa2} and Lemma~\ref{4.11 arch}.
\end{enumerate}
\end{proof}
%%%%%%%%%%%%%%%%%%%%%%%%%%%%%%%%%%%%%%%%%%
%
%
%
%
%
%
%
%
%
%
%
%
%
%
%%%%%%%%%%%%%%%%%%%%%%%%%%%%%%%%%%%%%%%%%%%%%%%
\appendix
\section{On the image of local base change lifts for generic representations}
\label{JS analogue}
In this appendix, using a similar argument as Jiang and Soudy~\cite{JS03} and Liu~\cite{Li11},
we shall determine the image of local base change lifts for generic discrete series representations of even unitary groups $\mathrm{U}_{2n}$.
We note that as in the above two papers, we may write down Langlands parameter corresponding to 
these representations. 
We will consider this problem for not only even unitary groups but also odd unitary groups
in our future work.

In this appendix, we will use the same notation in the main body of this paper.
Further, for simplicity, we say that a representation of $\mathrm{U}_{2n}$ is generic if it is $\psi_N$-generic.
For a given representations $\tau_i$ of $\mathrm{GL}_{k_i}(E)$, $i=1, \cdots, a$ and $\rho$ of $\mathrm{U}_{2m}$, we 
denote by 
\[
\tau_1\times \cdots \times \tau_{k_a} \quad
\text{(resp. $\tau_1\times \cdots \times \tau_{k_a} \times \rho$)}
\]
the parabolic induction of $\mathrm{GL}_{k_1+\cdots +k_a}$ (resp. $\mathrm{U}_{2(k_1+\cdots +k_a+m)}$) for the parabolic subgroup
with the Levi part $\mathrm{GL}_{k_1}(E) \times \cdots \times \mathrm{GL}_{k_1}(E)$
(resp. $\mathrm{GL}_{k_1}(E) \times \cdots \times \mathrm{GL}_{k_1}(E) \times \mathrm{U}_{2m}$).

We would like to start our observation with the case of supercuspidal representations.
Let $\Pi^{(sg)}(\mathrm{U}_{2n})$ be the set of all equivalence classes of irreducible supercuspidal 
generic representations of $\mathrm{U}_{2n}$.
Let $\Pi^{(sg)}(\mathrm{GL}_{2n}(E))$ be the set of all equivalence classes of irreducible tempered representations of $\mathrm{GL}_{2n}(E)$
of the form
\[
\tau_1 \times \cdots \times \tau_r 
\]
where $\tau_i$ is an irreducible supercuspidal representation of $\mathrm{GL}_{n_i}(E)$
such that $L(s, \tau_i, \mathrm{As}^+)$ has a pole at $s=0$ and for $i \ne j$,
$\tau_i \not \simeq \tau_j$.
In \cite{KK}, Kim and Krishnamurthy constructed local base change lifts of  generic representations of $\mathrm{U}_{2n}$ explicitly.
Indeed, we have the following result.
\begin{theorem}[Proposition~8.4 in \cite{KK}]
\label{local bc lift}
There is a map $l$ from $\Pi^{(sg)}(\mathrm{U}_{2n})$ to $\Pi^{(sg)}(\mathrm{GL}_{2n}(E))$
which preserves local $\gamma$-factors with $\mathrm{GL}$-twist, namely
\[
\gamma^{Sh}(s, \pi \times \sigma, \psi)=
\lambda(E \slash F, \psi)^{2nk}
\gamma^{RS}(s, l(\pi) \times \sigma, \psi)
\]
for any $\pi \in \Pi^{(sg)}(\mathrm{U}_{2n})$ and any irreducible generic representation $\sigma$ of $\mathrm{GL}_k(E)$ with $k \in \mN$.
Here, the $\gamma$-factor on the right-hand (res. left-hand) side is defined by the Rankin-Selberg method \cite{JPSS}
(resp. Langlands-Shahidi method~\cite{Sh90a} and \cite{Sh90b}).
\end{theorem}
Recently, the author proved the following result.
\begin{theorem}[Corollary~9.2 in \cite{Mo1}]
\label{uniqueness map l}
The map $l$ in Theorem~\ref{local bc lift} is bijective and unique.
\end{theorem}
Let us extend the above result to the case of generic discrete series representations.
Let us recall the construction of generic discrete series representations by Moeglin and Tadi\'{c}~\cite{MT02}.
To explain their construction, let us define some representations of general linear groups.
Let $\rho$ be an irreducible supecuspidal representation of $\mathrm{GL}_k(E)$.
Then for integers $l \geq m > 0$ with the same parity, we write
by $D(l, m, \rho)$ the unique irreducible subrepersentation of 
\[
\mathrm{Ind} (|\det|^{(l-1) \slash 2}\rho \otimes |\det|^{(l-1) \slash 2-1} \rho \otimes \cdots 
\otimes  |\det|^{-(m-1) \slash 2} \rho )
\]
and by $D(l, \rho)$ the unique subrepresentation of 
\[
\mathrm{Ind} (|\det|^{(l-1) \slash 2}\rho \otimes |\det|^{(l-1) \slash 2-1} \rho \otimes \cdots 
\otimes  |\det|^{(l+1) \slash 2- [l \slash 2]} \rho ).
\]
Then $D(l, m, \rho)$ (resp. $D(l, \rho)$) is essentially square integrable representation
of $\mathrm{GL}_{\frac{k(l+m)}{2}}(E)$ (resp. $\mathrm{GL}_{[\frac{l}{2}] k}(E)$).
Further, we write
\[
St(\rho, l) = D(l, l, \rho). 
\]
Then we consider a parabolic induction 
\begin{equation}
\label{explicit const of ds}
D(l_1, m_1, \rho_1) \times \cdots \times D(l_r, m_r, \rho_r) \times 
D(l_{r+1}, \rho_{r+1}) \times \cdots  \times D(l_{t}, \rho_{t}) \times \tau_0
\end{equation}
where $l_i> m_i >0$, $\tau_0$ is an irreducible supercuspidal representation of a smaller even unitary group
and $\rho_i$ are irreducible supercuspidal representation of $\mathrm{GL}_{k_i}(E)$
such that $\rho_i$ is conjugate self-dual.
Further, we shall suppose that 
\begin{center}
$L(s, \rho_i, \mathrm{As}^+)$ has a pole at $s=0$ if and only if $l_i$ is odd.
\end{center}
Since $L(s, \rho_i, \mathrm{As}^+)$ has a pole at $s=0$ if and only if 
$\rho_i$ is $\mathrm{GL}_{k_i}(F)$-distinguished by 
Anandavardhanan-Kable-Tandon~\cite[Corollary~1.5]{AKT04},
the above condition is equivalent to that
\begin{equation}
\label{cond equiv dist}
\text{$\rho_i$ is $\mathrm{GL}_{k_i}(F)$-distinguished if and only if $l_i$ is odd.}
\end{equation}
We know that $\rho_i$ should be $(\mathrm{GL}_{k_i}(F), \eta_{E \slash F})$-distinguished or $\mathrm{GL}_{k_i}(F)$-distinguished
by Kable~\cite[Theorem]{Ka04}.
Moreover if $\rho_i$ is $(\mathrm{GL}_{k_i}(F), \eta_{E \slash F})$-distinguished,
$\rho_i$ is not $\mathrm{GL}_{k_i}(F)$-distinguished, and vice versa by \cite[Corollary~1.6]{AKT04}.
When the above condition holds, the unique generic constituent of 
the above induced representation is discrete series representation.
When it is a representation of $\mathrm{U}_{2n}$, every discrete series representation is obtained in this way.
Let $\Pi^{(dg)}(\mathrm{U}_{2n})$ be the set of irreducible generic discrete series representations of $\mathrm{U}_{2n}$,
namely the set of all irreducible representation obtained in the above way.

Let $\Pi^{(dg)}(\mathrm{GL}_{2n}(E))$ be the set of irreducible representations of $\mathrm{GL}_{2n}(E)$
of the form $\pi = \tau_1 \times \cdots \times \tau_r$
where $\tau_i$ is an irreducible discrete series representation of $\mathrm{GL}_{n_i}(E)$ such that $L(s, \tau_i, \mathrm{As}^+)$
has a pole at $s=0$.
We note that $\tau_i \not \simeq \tau_j$ if $i \ne j$ because of the irreducibility of $\pi$.
Further, we know that $\tau_i$ is conjugate self-dual, in particular it is unitary.
Then by Bernstein-Zelevinsky~\cite{BZ77}, we may write 
\[
\tau_i = \mathrm{St}(\rho_i, a_i)
\]
where $a_i \in \mZ_{\geq 0}$ and $\rho_i$ is an irreducible supercuspidal representation of $\mathrm{GL}_{m_i}(E)$
such that $n_i = a_i m_i$.
We have a necessary and sufficient condition for $\mathrm{St}(\rho_i, a_i)$ to be $\mathrm{GL}_{n_i}(F)$-distinguished.
\begin{theorem}[Corollary~4.2 in \cite{Ma09}]
\label{mat cor4.2}
Let $\rho$ be an irreducible supercuspidal representation of $\mathrm{GL}_k(E)$
and $a \in \mN$.
Then generalized Seinberg representation $\mathrm{St}(\rho, a)$ of $\mathrm{GL}_{ak}(E)$ is $\mathrm{GL}_{ak}(F)$-distinguished 
if and only if $\rho$ is $(\mathrm{GL}_{k}(F), \eta_{E \slash F}^{a-1})$-distinguished.
\end{theorem}
From this theorem, we see that $\rho_i$ is $\eta_{E \slash F}^{a_i-1}$-distinguished for any $i$.
In particular, $\rho_i$ is conjugate self-dual.
Then we prove the following generalization of Theorem~\ref{local bc lift}, \ref{uniqueness map l}
to discrete series representations.
\begin{theorem}
\label{backward lifting}
There is a bijective map $l$ from $\Pi^{(dg)}(\mathrm{U}_{2n})$ to $\Pi^{(dg)}(\mathrm{GL}_{2n}(E))$
satisfying the condition
\[
\gamma^{Sh}(s, \pi \times \sigma, \psi)=
\lambda(E \slash F, \psi)^{2nk}
\gamma^{RS}(s, l(\pi) \times \sigma, \psi)
\]
for any $\pi \in \Pi^{(dg)}(\mathrm{U}_{2n})$ and any irreducible generic representation $\sigma$ of $\mathrm{GL}_k(E)$ with $k \in \mN$.
Further, the above map is unique.
\end{theorem}
\begin{proof}
First, we note that the uniqueness follows from a local converse theorem for generic representations of $\mathrm{GL}_{2n}(E)$
by Henniart~\cite{He93}.
Following Kim-Krishnamurthy~\cite{KK}, we define a map from $\Pi^{(dg)}(\mathrm{U}_{2n})$ to $\Pi^{(dg)}(\mathrm{GL}_{2n}(E))$.
Let $\pi$ be an element of $\Pi^{(dg)}(\mathrm{U}_{2n})$, and write it as \eqref{explicit const of ds}.
Then we define $l(\pi)$ by
\[
\mathrm{St}(\rho_1, l_1) \times \mathrm{St}(\rho_1, m_1) \times \cdots \times \mathrm{St}(\rho_r, l_r) \times \mathrm{St}(\rho_r, m_r)
\times \mathrm{St}(\rho_{r+1}, l_{r+1}) \times \cdots \times \mathrm{St}(\rho_{t}, l_{t}) \times l(\tau_0),
\]
which is irreducible by Bernstein-Zelevinsky~\cite{BZ77}.
First, let us check that this representation is in $\Pi^{(dg)}(\mathrm{GL}_{2n}(E))$.
From Theorem~\ref{local bc lift}, $l(\tau_0)$ should be of the form
$\Pi_{1} \times \cdots \times \Pi_u$ where $\Pi_i$ are mutually distinct irreducible supercuspidal representations
such that $L(s, \Pi_i, \mathrm{As}^+)$ has a pole at $s=0$.
Thus, it suffices to check that $\mathrm{St}(\rho_i, l_i)$ is $\mathrm{GL}(F)$-distinguished.
Indeed, $\mathrm{St}(\rho_i, l_i)$ is $\mathrm{GL}_{l_i k_i}(F)$-distinguished if and only if 
$\rho_i$ is $(\mathrm{GL}(F), \omega_{E \slash F}^{l_i-1})$-distinguished by Theorem~\ref{mat cor4.2}.
From the condition \eqref{cond equiv dist}, 
if $l_i$ is odd (resp. even), $\rho_i$ is $\mathrm{GL}_{k_i}(F)$-distinguished
(resp. $(\mathrm{GL}_{k_i}(F), \omega_{E \slash F}^{l_i-1})$-distinguished).
Thus, $\mathrm{St}(\rho_i, l_i)$ is $\mathrm{GL}_{l_i k_i}(F)$-distinguished.

From the definition of $\Pi^{(dg)}(\mathrm{GL}_{2n}(E))$, the surjectivity of the map $l$ is clear.
Further, from the local converse theorem \cite[Theorem~9.4]{Mo1} for $\mathrm{U}_{2n}$, its injectivity follows.
\end{proof}
Finally, we shall realize the above map by the local descent.
\begin{theorem}
\label{descent BC local}
The map $\pi \mapsto \mathcal{D}_{\psi}^{\Upsilon}(\pi)$ defines a bijection
\[
\mathcal{D}_{\psi}^{\Upsilon} : \Pi^{(dg)}(\mathrm{GL}_{2n}(E)) \rightarrow \Pi^{(dg)}(\mathrm{U}_{2n})
\]
Moreover, if $\pi \in \Pi^{(dg)}(\mathrm{GL}_{2n}(E))$ and $\tilde{\pi} = \mathcal{D}_{\psi}^{\Upsilon}(\pi \otimes \Upsilon)$ then
\[
\gamma(s, \tilde{\pi} \times \tau, \psi) = \lambda(E \slash F)^{2nk} 
 \gamma(s, \pi \times \tau, \psi).
\] 
for any irreducible generic representation $\tau$ of $\mathrm{GL}_{k}(E)$.
\end{theorem}
\begin{proof}
First, we shall prove 
that $\mathcal{D}_\psi^\Upsilon(\pi)$ is irreducible.
From Theorem~\ref{backward lifting}, there is an irreducible discrete series representation $\sigma$
of $G^\prime$ such that $l(\sigma) = \pi$ where $l$ is the map constructed in that theorem.
Indeed, from its definition and explicit description of local base change lift by Kim-Krishnamurthy~\cite{KK},
$l(\sigma) = \mathrm{BC}(\sigma)$.
Let us take number fields $L \slash K$ such that for some place $v_0$ of $K$, $L_{v_0} \simeq E$ and $K_{v_0} \simeq F$.
If necessary, replacing $\psi$ by $\psi^a$ with some $a \in (F^\times)^2$, we may suppose that 
$\psi$ is $v_0$-component of an additive character $\psi_{\mA_K}$ of $\mA_K$.
Then by  \cite[Corollary~A.6]{ILM}, there is $\psi_{\mA_K}$-generic irreducible cuspidal automorphic representation $\Sigma$
of $\mathrm{U}_{2n}(\mA_K)$ such that $\Sigma_{v_0} = \sigma$.
From the explicit construction of base change lifts, we have $\mathrm{BC}(\Sigma)_{v_0} = \mathrm{BC}(\sigma) = \pi$.
Let us take a character $\eta$ of $\mA_L^\times$ such that its restriction to $\mA_E^\times$
is the quadratic character $\omega_{L \slash K}$ corresponding to $L \slash K$
and $\eta_{v_0} = \Upsilon$.
Then by \cite{LMa}, $\mathcal{D}_{\psi_{\mA_K}}^{\eta}(\mathrm{BC}(\Sigma))_{v_0} = \mathcal{D}_{\psi}^{\Upsilon}(\pi)$
is irreducible.

Second of all, we shall prove $\mathrm{BC}(\mathcal{D}_{\psi}^{\Upsilon}(\pi)) \in \Pi^{(dg)}(\mathrm{GL}_{2n})$.
From the above argument, we have a globalization $\Pi := \mathrm{BC}(\Sigma)$ of $\pi$.
Further, from the unramified computation and strong multiplicity one theorem for $\mathrm{GL}_{2n}$, we get
$\mathrm{BC}(\mathcal{D}_{\psi_{\mA_K}}^\eta(\Pi)) = \Pi \otimes \eta^{-1}$
and thus $\mathrm{BC}(\mathcal{D}_{\psi_{\mA_K}}^{\eta}(\Pi))_{v_0} = \pi \otimes \Upsilon^{-1}$.
From the construction of base change lifts, we find that 
$\mathrm{BC}(\mathcal{D}_{\psi_{\mA_K}}^{\eta}(\Pi))_{v_0} = \mathrm{BC}((\mathcal{D}_{\psi_{\mA_K}}^{\eta}(\Pi))_{v_0})$.
From the definition of explicit local descent, we get $(\mathcal{D}_{\psi_{\mA_K}}^{\eta}(\Pi))_{v_0} =\mathcal{D}_{\psi}^{\Upsilon}(\Pi_{v_0})$.
Therefore, 
\begin{equation}
\label{claim2 proof}
\mathrm{BC}(\mathcal{D}_{\psi}^{\Upsilon}(\pi)) = \pi \otimes \Upsilon^{-1} \in \Pi^{(dg)}(\mathrm{GL}_{2n}).
\end{equation}

Third of all, we prove $\mathcal{D}_{\psi}^{\Upsilon}(\pi) \in \Pi^{(dg)}(G^\prime)$.
From Theorem~\ref{backward lifting} and the previous claim, there is $\sigma^\prime \in \Pi^{(dg)}(G^\prime)$ 
such that $l(\sigma^\prime) = \mathrm{BC}(\mathcal{D}_{\psi}^{\Upsilon}(\pi))$.
The base change lift is strong, so that 
$\gamma(s, \mathrm{BC}(\mathcal{D}_{\psi}^{\Upsilon}(\pi)) \times \tau, \psi)
=\lambda(E \slash F)^{n} \gamma(s, \mathcal{D}_{\psi}^{\Upsilon}(\pi) \times \tau, \psi)
$
for any irreducible generic representation $\tau$ of $\mathrm{GL}_i$ with $1 \leq i \leq 2n$.
Again, by Theorem~\ref{backward lifting} 
$
\gamma(s, l(\sigma^\prime) \times \tau, \psi)
=\lambda(E \slash F)^{n} \gamma(s, \sigma^\prime \times \tau, \psi).
$
Hence,
\[
 \gamma(s, \mathcal{D}_{\psi}^{\Upsilon}(\pi) \times \tau, \psi) =  \gamma(s, \sigma^\prime \times \tau, \psi)
\]
and the local converse theorem \cite[Theorem~9.4]{Mo1} implies that 
\[
\mathcal{D}_{\psi}^{\Upsilon}(\pi)  = \sigma \in \Pi^{(dg)}(\mathrm{U}_{2n}).
\]
Further we note that if $\mathcal{D}_{\psi}^{\Upsilon}(\pi_1) = \mathcal{D}_{\psi}^{\Upsilon}(\pi_2)$, then $\pi_1 \simeq \pi_2$
by \eqref{claim2 proof}.

Finally, we shall prove
$l(\pi) = \mathcal{D}_{\psi}^{\Upsilon}(\pi)$.
Since the base change is strong, we have
$\lambda(E \slash F)^{2nk} \gamma(s, \mathrm{BC}(\mathcal{D}_{\psi}^{\Upsilon}(\pi)) \times \tau, \psi)
=\gamma(s, \mathcal{D}_{\psi}^{\Upsilon}(\pi) \times \tau, \psi).$
On the other hand,  we have
$\gamma(s, \mathrm{BC}(\mathcal{D}_{\psi}^{\Upsilon}(\pi)) \times \tau, \psi)
 = \gamma(s, \pi \otimes \Upsilon^{-1} \times \tau, \psi)$ by \eqref{claim2 proof}.
Hence, we get
\[
\lambda(E \slash F)^{2nk}  
\gamma(s, \pi \times \tau, \psi)
=\gamma(s, \mathcal{D}_{\psi}^{\Upsilon}(\pi \otimes \Upsilon) \times \tau, \psi).
\]
Then our assertion follows from the uniqueness in Theorem~\ref{backward lifting}.
\end{proof}

%%%%%%%%%%%%%%%%%%%%%%%%%%%%%%%%%%%%%%%%%%%%%%%
%
%
%
%
%
%
%
%
%
%
%
%
%
%
%%%%%%%%%%%%%%%%%%%%%%%%%%%%%%%%%%%%%%%%%%%%%%%
\section{Stability of certain oscillatory integrals for quasi-split reductive groups and non-vanishing of Bessel functions}
In this section, we shall prove the stability of certain oscillatory integrals for 
a quasi-split reductive group $\mathbf{G}$ over a non-archimedean local field $F$ of characteristic zero.
This is a generalization of \cite{LM13} to quasi-split reductive groups.
\subsection{main results}
Denote $G = \mathbf{G}(F)$.
Fix a Borel subgroup $B$ of $G$,
and we denote its unipotent radical by $N$. 
Let $\psi_N$ be a nondegenerate character of $N$. 
Consider the space $\Omega(N \backslash G, \psi_N)$ of smooth functions on $G$ such that
$f(ng) = \psi_N(n)f(g)$ for all $n \in N, g \in G$. 
Denote the regular representation of $G$ on $\Omega(N \backslash G, \psi_N)$ by $R$. 
For any compact open subgroup $K$ of $G$ let $\Omega(N \backslash G, \psi_N)^K$ denote the subspace of right $K$-invariant functions. 
Let $\overline{w_0} \in G$ be a rerresentative of the longest Weyl element $w_0$. 
Fix a Haar measure on $N$. For any compact open subgroup $N^\prime$ of $N$ and $W \in \Omega(N \backslash G, \psi_N)$
let
\[
R_{N^\prime, \psi_N} W := \frac{1}{\mathrm{vol}(N^\prime)} \int_{N^\prime} (R(n^\prime) W) \psi_N(n^\prime)^{-1} \, dn^\prime
\in \Omega(N \backslash G, \psi_N).
\]
%%%%%%%%%%%%%%%%%%%%%%%%
%
%
%
%
%
%
%
%%%%%%%%%%%%%%%%%%%%%%%%
\begin{theorem}
\label{stability thm1}
For any open subgroup $K$ of $G$, 
there exists an open compact subgroup $N^\prime$ of $N$ such that, 
for any $W \in \Omega(N \backslash G, \psi_N)^K$, $(R_{N^\prime ,\psi_N} W)(\overline{w_0} \cdot)$
is compactly supported on $N$.  

As a consequence. for any $W \in \Omega(N \backslash G, \psi_N)$ and a compact open subgroup $N^\prime$ of $N$ satisfying the above condition, we can define the stable integral
\[
\int_{N}^{st} W(\overline{w_0}n) \psi_{N}(n)^{-1} \, dn:=
\int_{N} (R_{N^\prime, \psi_N} W)(\overline{w_0} n) \psi_N(n)^{-1} \, dn.
\]
\end{theorem}
%%%%%%%%%%%%%%%%%%%%%%%%
%
%
%
%
%
%
%
%%%%%%%%%%%%%%%%%%%%%%%%
More generally, we say that $g \in G$ is relevant (with respect to $\psi_N$), 
if $\psi_N(gng^{-1})= \psi_N(n)$ for all $n \in  N \cap g^{-1}Ng$. 
For any $g \in G_{\rm rel}$ which is the set of relevant elements, we may use the above procedure to regularize the integral.
%%%%%%%%%%%%%%%%%%%%%%%%
%
%
%
%
%
%
%
%%%%%%%%%%%%%%%%%%%%%%%%
\begin{theorem}
\label{stability}
For any open subgroup $K$ of $G$ and $g \in G$, 
there exists an open compact subgroup $N^\prime$ of 
$N$ sucht that, for any $W \in \Omega(N \backslash G, \psi_N)^K$,$(R_{N^\prime, \psi_N}W)(g \cdot)$
is compactly supported on $g^{-1} N g \cap N \backslash N$.
In particular, when $g \in G_{\rm rel}$, we can define
\[
\int_{g^{-1} Ng \cap N \backslash N} W(gn) \psi_N(n)^{-1} \, dn:=
\int_{g^{-1} Ng \cap N \backslash N} (R_{N^\prime, \psi_N} W)(gn) \psi_N(n)^{-1} \, dn
\]
where $N^\prime$ is a compact open subgroup of $N$ given in Theorem~\ref{stability thm1}.
\end{theorem}
\begin{Remark}
\label{stability rmk}
Let $\pi$ be a $\psi_N$-generic irreducible admissible representation of $G$.
Then by the uniqueness of Whittaker functionals, there is a function $\mathbb{B}_\pi^{\psi_N}$ on $G$
such that 
\[
\int_{g^{-1} Ng \cap N \backslash N} W(gn) \psi_N(n)^{-1} \, dn
= \mathbb{B}_\pi^{\psi_N}(g) W(e)
\]
for $g \in G_{\rm rel}$ and $ \mathbb{B}_\pi^{\psi_N}(g) = 0$ for $g \in G \setminus G_{\rm rel}$.
Then function $ \mathbb{B}_\pi^{\psi_N}$ is called the Bessel function attached to $\pi$.
\end{Remark}
%%%%%%%%%%%%%%%%%%%%%%%%
%
%
%
%
%
%
%
%%%%%%%%%%%%%%%%%%%%%%%%
\subsection{Notation}
Suppose that $\mathbf{G}$ is a semisimple quasi-split group defined over $F$ of rank $r$
since main results are easily reduced to the semisimple case.
Write $G = \mathbf{G}(F)$.
Let $B$ be a Borel subgroup of $G$, $T$ a maximal torus contained in $B$,
$A \subset T$ a maximal $F$-split torus of $G$ and $N$ be the unipotent radical of $B$
so that $B = T N$.
Let $\overline{B} = T \overline{N}$ be the opposite Borel subgroup with respect to $T$.
Let $\mathbf{K}$ be a hyper special maximal compact open subgroup of $G$
in good position with respect to $B$.
Then we have the Iwasawa decomposition $G = T N \mathbf{K}$.
Let $X^\ast(Y)$ (resp. $X_\ast(Y)$) be the lattice of rational characters (resp. cocharacters) of a group $Y$.

\subsubsection{Relative roots and weights}
Let $\Phi \subset X^\ast(T)$ be the set of root of $T$ in $\mathfrak{g}: =\mathrm{Lie}(G)$.
Further, let $\Phi_{\rm rel} \subset X^\ast(A)$ be the set of root of $A$ in $\mathfrak{g}$,
$\Phi_{{\rm rel}, +}$ be the subset of positive roots in $\Phi_{\rm rel}$ and $\triangle_{0}$ be the subset of simple roots with respect to $(B, A)$ in $\Phi_{\rm rel}$.
Similarly, let $\triangle_{0}^\vee \subset \Phi_{{\rm rel}, +}^\vee \subset \Phi_{\rm rel}^\vee \subset X_\ast(A)$
be the sets of (simple, or positive) co-roots.
For $\alpha \in X^\ast(A)$, we denote by $\mathfrak{g}_\alpha$ the $\alpha$-eigenspace in $\mathfrak{g}$.
We denote by $\alpha \leftrightarrow \alpha^\vee$ the canonical bijection
between $\Phi_{\rm rel}$ and $\Phi_{\rm rel}^\vee$ (resp. $\Phi_{{\rm rel}, +} \leftrightarrow \Phi_{{\rm rel}, +}^\vee$).
Denote $\mathfrak{a}_T: = X_\ast(T) \otimes \mR$.

For $\alpha \in \Phi_{\rm rel}$, we denote by $N_\alpha$ the unipotent subgroup whose Lie algebra is $\mathfrak{g}_\alpha + \mathfrak{g}_{2 \alpha}$.
For $\alpha \in \triangle_0$, let $P_\alpha$ be the standard parabolic subgroup with respect to $\alpha$.
Then its Levi subgroup $M_\alpha$ has rank 1,
so does the simply connected covering $\tilde{M}_\alpha$.
%Recall that the projection of this covering gives an isomorphism of its unipotent radical $N_\alpha$.
Then, $\tilde{M}_\alpha$ is $\mathrm{SL}_2(F_\alpha)$ or the quasi-split special unitary group 
$\mathrm{SU}(3)_\alpha$ with respect to a quadratic extension $E_\alpha$ of $F_\alpha$ where $F_\alpha$ is a finite extension of $F$.
We say that $\alpha \in \triangle_0$ is of type (I) (resp. type (II)) if 
$\tilde{M}_\alpha$ is isomorphic to $\mathrm{SL}_2(F_\alpha)$ (resp. $\mathrm{SU}(3)_\alpha$).
%We remark that when $\alpha \in \triangle_0$ is of type (II), we can write $\alpha = 2 \alpha^\prime$
%with $\alpha^\prime \in (\Phi_{{\rm rel}, +}) \setminus (\Phi^{\rm red}_+)$
%where $\Phi^{\rm red}_+$ is the set of reduced positive roots.
%Note that $\alpha \mapsto \alpha^\vee$ does not give  a bijection $\triangle_0 \leftrightarrow \triangle_0^\vee$.
%Indeed, 
%\[
%\triangle_0^\vee 
%= \{ \alpha^\vee : \alpha \in \triangle_0 \text{ is of type (I)} \} \cup \{ \frac{1}{2} \alpha^\vee : \alpha \in \triangle_0 \text{ is of type (II)} \}
%\]

Suppose that $\alpha \in \triangle_0$ is of type (II).
Then we may realize $\mathrm{SU}(3)_\alpha$ by
\[
\mathrm{SU}(3)_\alpha = \left\{ g \in \mathrm{SL}(3, E_\alpha) : {}^{t} \bar{g} \begin{pmatrix} &&1\\ &1&\\ 1&&\end{pmatrix} g = \begin{pmatrix} &&1\\ &1&\\ 1&&\end{pmatrix}  \right\}.
\]
Further, we may take the following group as a Borel subgroup of $\mathrm{SU}(3)_\alpha$:
\[
\tilde{B}_\alpha = \tilde{T}_\alpha \tilde{N}_\alpha
\]
where
\[
\tilde{T}_\alpha = \left\{t_\alpha(y):= \begin{pmatrix}
y&&\\
&y \slash \bar{y}&\\
&&\bar{y}^{-1}
\end{pmatrix}
:
 y \in E_\alpha^\times \right\}
\]
and
\[
 \tilde{N}_\alpha = \left\{ x_\alpha(r, m) := 
\begin{pmatrix}
1&r&m\\
&1&-\bar{r}\\
&&1
\end{pmatrix}
:
\text{$x, y \in E_\alpha$ s.t. $N_{E_\alpha \slash F_\alpha}(r) = -\mathrm{Tr}_{E_\alpha \slash F_\alpha}(m)$}
\right\}
\]
Here, $N_{E_\alpha \slash F_\alpha}$ (resp. $\mathrm{Tr}_{E_\alpha \slash F_\alpha}$) is the norm (resp. trace) map  from $E_\alpha$ to $F_\alpha$.
Note that
\[
x_\alpha(r, m)^{-1} = x_\alpha(-r, \bar{m})
\]
Hereafter, we fix a covering map of $M_\alpha$, namely 
\[
1 \rightarrow \pi_1(M_\alpha) \rightarrow \tilde{M}_\alpha \overset{f_{\alpha}}{\longrightarrow} M_\alpha \rightarrow 1.
\]
so  that $f_\alpha(\tilde{T}_\alpha) = T_\alpha:= T \cap M_\alpha$ and $f_\alpha(\{x_{\alpha}(r, m): r, m \in \mO_{E_\alpha}\}) = N_\alpha \cap \mathbf{K}$
where $\mO_{E_\alpha}$ denotes the ring of integers of $E_\alpha$.
Further, we denote the image of $t_\alpha(y)$ and $x_\alpha(r, m)$ by the same symbol.
Then $N_\alpha$ is the unipotent subgroup of $G$ generated by $\{ x_{\alpha}(r, m) : r, m \in E_\alpha \text{ such that } N_{E_\alpha \slash F_\alpha}(r)
 = -\mathrm{Tr}_{E_\alpha \slash F_\alpha}(m) \}$.
For $\alpha \in \triangle_0$, write $N^\alpha$ for the unipotent radical of the parabolic subgroup $B \cup B s_\alpha B$.
Then we may write 
\[
N = N_\alpha \ltimes N^\alpha
\]
Further, for $r, m \in E_\alpha$ such that $N_{E_\alpha \slash F_\alpha}(r) = -\mathrm{Tr}_{E_\alpha \slash F_\alpha}(m)$, we define 
\[
x_{-\alpha}(r, m):= 
\begin{pmatrix}
1&&\\
-\bar{r}&1&\\
m&r&1
\end{pmatrix}
\]
Also, we write the image of this element by the same symbol.
From the definition, we can take the co-character $\alpha^\vee$ so that
\[
%\alpha^\vee = 2 (2 \alpha)^\vee
%\quad
%\text{and}
%\quad
\alpha^\vee(x) = 
 \begin{pmatrix}
x&&\\
&1&\\
&&x^{-1}
\end{pmatrix}
,\quad 
x \in F^\times
\]
Further, we extend $\alpha^\vee$ to $E_\alpha^\times \rightarrow T_\alpha$ by
\[
\alpha^\vee(x)=
 \begin{pmatrix}
x&&\\
&\bar{x} \slash x&\\
&&\bar{x}^{-1}
\end{pmatrix},
\quad
x \in E_\alpha^\times.
\]
More generally, for $\beta = w\alpha$ with $w \in \mathbf{W}$ and $\alpha \in \triangle_0$ of type (II),
we define
\[
x_\beta(r, m) := w x_\alpha(r, m) w^{-1}
\quad
\text{and}
\quad
\beta^\vee := w \cdot (\alpha^\vee) 
\]
Suppose that $\alpha \in \triangle_0$ is of type (I).
Then we choose a parametrization $n_{\alpha} : F_\alpha \rightarrow N_\alpha$ such that 
$n_{\alpha}(\mathcal{O}_{F_\alpha}) = N_\alpha \cap \mathbf{K}$.
Here, $\mathcal{O}_{F_\alpha}$ denotes the ring of integers of $F_\alpha$.
Further, we define
\[
\alpha^\vee(x) = \begin{pmatrix} x&\\ &x^{-1}\end{pmatrix}, \quad x \in F^\times
\]
and we also extend $\alpha^\vee$ to $F_\alpha^\times \rightarrow T_\alpha$ by 
\[
\alpha^\vee(x) = \begin{pmatrix} x&\\ &x^{-1}\end{pmatrix}, \quad x \in F_\alpha^\times.
\]
Similarly as above, for $\beta = w\alpha$ with $w \in \mathbf{W}$ and $\alpha \in \triangle_0$ of type (I),
we define
\[
n_\beta(x) := w x_\alpha(x) w^{-1}
\quad
\text{and}
\quad
\beta^\vee := w \cdot (\alpha^\vee). 
\]
%\color{red}
%For any $\gamma = \sum_{\beta \in \triangle_0^\vee} c_\beta \beta \in \Phi_{{\rm rel}, +}$ with $c_\beta \in \frac{1}{2}\mZ_{\geq 0}$, we define
%\[
%\gamma^\vee :=  \sum_{\beta \in \triangle_0} c_\beta \beta^\vee
%\]
%where $\beta^\vee$ is an extension to $E_\alpha^\times$ or $F_\alpha^\times$ depending on the type of $\beta$.
%Further, for $\gamma \in \Phi_{{\rm rel}, -}$, we define
%\[
%\gamma^\vee := -(-\gamma)^\vee
%\]
%Remark that for $\alpha \in \triangle_0$ of type (II), 
%\[
%s_\alpha \cdot \alpha^\vee(x) \cdot s_\alpha^{-1}
%= (s_\alpha \alpha)^\vee(x), \quad x \in F_\alpha^\times
%\]
%but
%\[
%s_\alpha \cdot \alpha^\vee(x) \cdot s_\alpha^{-1}
%\ne (s_ \alpha \alpha)^\vee(x), \quad x \in E_\alpha^\times
%\]
%Further, we note that the left-hand side is not in $\mathfrak{a}_T$.
%For $\alpha \in \triangle_0$ with $\alpha = \sum_{\alpha_i \in \triangle_0} r_i \alpha_{i} $, define 
%\[
%t_\alpha
%=
%\left\{
%\begin{array}{ll}
%\alpha^\vee & \text{ if $\alpha$ is of type (I)} \\
%&\\
%\alpha^\vee_E & \text{ if $\alpha$ is of type (II)}
%\end{array}
%\right.
%\]
%In general, for $\alpha \in \Phi_{\rm rel}$, we define 
%\[
%t_\alpha
%= \prod_i t_{\alpha_i}^{r_i}.
%\]
%\color{black}
On the other hand, let $\lambda = \sum_{\alpha \in \triangle_0} r_\alpha \alpha \in \mathfrak{a}_0^\ast$, and define a
character a character $|\lambda| : A \rightarrow \mR_+$ by
\[
\lambda(t) = \prod_{\alpha \in \triangle_0} |\alpha(t)|^{r_\alpha}
\]
For $\alpha \in \triangle_0$, we may consider $\alpha$ as a character of $\alpha^\vee(F_\alpha^\times)$ in the trivial way.
Further, when $\alpha$ is of type (II), we may extend a character $|\alpha|$ to a character $|\alpha_0| : T_{\alpha_i} \rightarrow \mR_+$,
and naturally a character of $T$.
Indeed, we define
\[
|\alpha_0| ( \alpha^\vee(x)) = |\alpha| (\alpha^\vee(N_{E_\alpha \slash F_\alpha}(x)))^{1 \slash 2}.
\]
Therefore, $\lambda = \sum_{\alpha \in \triangle_0} r_\alpha \alpha$, we can define $|\lambda_0| : T \rightarrow \mR_+$ by
$|\lambda_0| := \prod_{\alpha \in \triangle_0} |\alpha_0|^{r_\alpha}$.

By the Iwasawa decomposition, we extend $|\lambda_0|$ to a left-$N$ and right-$\mathbf{K}$ invariant function
on $G$. For any compact subset $C \subset G$, there exists a constant $\kappa_C$ (depending also on $\lambda$)
such that 
\begin{equation}
\label{4}
\kappa_C^{-1} |\lambda_0|(g) \leq |\lambda_0|(gh) \leq \kappa_C |\lambda_0|(g)
\end{equation}
for all $g \in G$ and $h \in C$.

Let $H : T \rightarrow \mathfrak{a}_T$ be the Harish-Chandra map given by
\[
\mathrm{exp} \langle \chi, H(t) \rangle = |\chi(t)|_F, \quad {}^{\forall} \chi \in X^\ast(T)_F.
\]
(Note that $X^\ast(T) \otimes \mR \simeq X^\ast(A) \otimes \mR$.)
%%%%%%%%%%%%%
%
%
%
%
%
%
%%%%%%%%%%%%%
\subsubsection{Weyl group, Bruhat order and Bruhat decomposition}
Let $\mathbf{W} = N_G(A) \slash Z_G(A) = N_G(A) \slash T$ be the relative Weyl group of $G$.
We fix once and for all representatives $\bar{w} \in \mathbf{K}$.
Then for any $w_1, w_2 \in \mathbf{W}$, we have $\overline{w_1 w_2} = t \overline{w_1} \, \overline{w_2}$
where $t \in T \cap \mathbf{K}$.
As in \cite{LM13}, we shall denote unspecified element on $T \cap \mathbf{K}$ by $\ast$.
Thus
\[
\overline{w_1 w_2} = \ast \overline{w_1} \, \overline{w_2}
\]
where $\ast$ depends on $w_1$ and $w_2$ and also the choice of $ \overline{w_1}$ and $\overline{w_2}$.
We have the Bruhat decomposition
\[
G = \bigcup_{w \in \mathbf{W}} BwB
\]
Then for any $g \in G$, we may write $g = b \bar{w} n^\prime$
where $b \in B$, $w \in \mathbf{W}$ and $n^\prime \in N_w^- := N \cap {\overline{w}}^{-1} \overline{N}\overline{w}$.
In general, a choice of $\bar{w}$ is not unique.
When we have $g \in N t \bar{w} N$ with $t \in T$ and $w \in \mathbf{W}$,
then we write $\mathbf{w}(g) = w$ and denote by $a(g)= t$.

Let 
\[
\mathcal{C}_+^\ast
= \left\{ \sum_{\alpha \in \Delta_0} c_\alpha \alpha : c_\alpha \in \mZ_{ \geq 0} \right\}.
\]
We write $\leq$ for the Bruhat order on $\mathbf{W}$.
If $w_1 \leq w_2 \in \mathbf{W}$, then for any dominant $\lambda \in \mathfrak{a}_0^\ast$, we have
\[
w_1 \lambda - w_2 \lambda \in \mathcal{C}_+^\ast.
\]
For any $w, w^\prime \in \mathbf{W}$, we have
\[
BwBw^\prime B \subset \bigcup_{w^{\prime \prime} \leq w^\prime} Bw w^{\prime \prime} B 
\]
(cf. \cite[(12)]{LM13}).

Let $S(w)$ be the set of simple roots which appear in a reduced decomposition of $w$. 
(This set does not depend on the reduced decomposition). 
It is the smallest set $S \subset \triangle_\circ$ such that $w$ belongs to the group generated by $\{s_\alpha : \alpha \in S\}$. 
Alternatively, 
\[
S(w) =  \{ \alpha \in \triangle_\circ : w\alpha^\ast \ne \alpha^\ast \}
\]
If $w_1 \leq w_2$, then $S(w_1) \subset S(w_2)$. 
Let $S^\circ(w) = S(w w_0)$.
Thus $S^\circ(w) = \empty$ if and only if $w=w_0$ while $S^\circ(1) = \triangle_0$.
Also, $S^\circ(w_1) \supset S^\circ(w_2)$ if $w_1 \leq w_2$.

For any $S \subset \triangle_0$, let $\Phi_{\rm rel}(S)$ be the set of roots in $\Phi_{\rm rel}$
which are linear combinations of roots in $S$.
This is a root subsystem of $\Phi_{\rm rel}$ which corresponds to the standard Levi subgroup determined by $S$.
We have
\[
\Phi_{\rm rel}(S) = \{ \beta \in \Phi : \langle \alpha^\ast, \beta^\vee \rangle = 0 \text{ for all } \alpha \not \in S \}.
\]
Note that for any $w \in \mathbf{W}$, $w_0 S(w_0 w) = - S^\circ(w)$ and hence
\[
w_0 \Phi_{\rm rel}(S(w_0 w)) = \Phi_{\rm rel}(S^\circ(w)).
\]
For $\beta \in \Phi_{\rm rel}$, we denote by $s_\beta \in \mathbf{W}$ the corresponding reflection.
%By the definition, we have
%\[
%(2 \alpha_0)^\vee(x) = 
% \begin{pmatrix}
%x&&\\
%&1&\\
%&&x^{-1}
%\end{pmatrix}
%\]
%%%%%%%%%%%%%
%
%
%
%
%
%
%%%%%%%%%%%%%
\subsubsection{Spaces of Whittaker functions}
Let $\Omega(N \backslash G, \psi_N)$ denote the space of Whittaker functions on $G$.
It is well-known that for any normal open subgroup $K$ of $\mathbf{K}$, we have
\[
\mathrm{sup} |\alpha_0| (\mathrm{supp}(W)) \ll_K 1,
\] 
for all $W \in \Omega(N \backslash G, \psi_N)^K$
and $\alpha \in \triangle_0$ with an extension to $\alpha_0$,
that is the image of $|\alpha_0|$ on $\mathrm{supp} W$ is bounded above in terms of $K$ only.
As in \cite{LM13}, following \cite[Definition~5.1]{Ba05}, we consider the space 
$\Omega^\circ(N \backslash G, \psi_N)$ consisting of those $W \in \Omega(N \backslash G, \psi_N)$
such that, for any $w \in \mathbb{W}$ and $\alpha \in S^\circ(w)$, we have 
$\mathrm{inf} |\alpha_0| (\mathrm{supp}_{BwB} W) > 0$.
Here, $\alpha_0$ is an extension of $\alpha$ to $T$.

For $\alpha \in \Phi_{\rm rel}$, we denote by $\Phi_\alpha$ the subset of $\Phi$ consisting of roots
whose restriction to $A$ is $\alpha$.
Note that for any $\alpha \in \Phi_{\rm rel}$, there is an extension $\beta \in \Phi$,
however this is not unique in general.
 
For any $w \in \mathbf{W}$ and $\varepsilon > 0$, let
\[
A^\varepsilon(w) = \left\{ t \in T  : 
|\beta^\ast|(t) \geq \varepsilon \text{ for all $\beta \in \Phi_\alpha$ with } \alpha \not \in S^\circ(w) \right\}.
\]
Also, let
\[
 B^\varepsilon(w) = \left\{ t \in T  : |\beta^\ast|(t) \leq 1 \slash \varepsilon \text{ for all $\beta \in \Phi_\alpha$ with } \alpha \in S(w), \,
|\beta^\ast|(t) \geq \varepsilon \text{ for all $\beta \in \Phi_\alpha$ with } \alpha \not \in S^\circ(w) \right\}.
\]
Further, we set $A_s^\varepsilon(w) = A^\varepsilon(w)  \cap A$ and $B_s^\varepsilon(w) = B^\varepsilon(w)  \cap A$.

We say that a set $C \subset G$ is bounded modulo $N$ if there exists a compact set $D \subset G$
such that $C \subset ND$.
Denote by $\Omega^\#(N \backslash G, \psi_N)$ the subspace of those $W \in \Omega(N \backslash G, \psi_N)$
such that for any $w \in \mathbf{W}$ and $\varepsilon > 0$,
$\mathrm{supp}_{N A^\varepsilon(w) w N} W$ is bounded modulo $N$.
%%%%%%%%%%%%%
%
%
%
%
%
%
%%%%%%%%%%%%%
\subsubsection{The main technical statements}
Theorem~\ref{stability thm1}, \ref{stability} readily follows from the following technical results.
\begin{proposition}[cf. Proposition~1 in \cite{LM13} and Theorem~9.5 in \cite{Ba05}]
\label{prp1}
We have $\Omega^\circ(N \backslash G, \psi_N) \subset \Omega^{\#}(N \backslash G, \psi_N)$.
\end{proposition}
\begin{lemma}[cf. Lemma~1 in \cite{LM13}]
\label{lmm1}
If $W \in \Omega^{\#}(N \backslash G, \psi_N)$, then for any $w \in \mathbb{W}$ and $\varepsilon > 0$,
the function $(t, n) \mapsto W(t\bar{w}n)$, $(t, n) \in B^\varepsilon(w) \times N_w^-$ is compactly supported.
\end{lemma}
\begin{proposition}[cf. Proposition~2 in \cite{LM13} and Theorem~7.3 in \cite{Ba05}]
\label{prp2}
For any normal open subgroup $K$ of $\mathbf{K}$, there exists a compact open subgroup $N^\prime$
of $N$ such that $R_{N^\prime, \psi_N} W \in \Omega^\circ(N \backslash G, \psi_N)$ for any $W \in \Omega(N \backslash G, \psi_N)^K$
\end{proposition}
\subsubsection{A lemma on Bruhat orders}
\label{bruhat order}
Suppose that $\alpha \in \triangle_0$ is of type (I).
We may choose $n_{\pm \alpha}$ and $\overline{s_\alpha}$ and so that 
\begin{equation}
\label{(6) type I}
n_\alpha(x)n_{-\alpha}(-x^{-1})=
\alpha^\vee(x) \overline{s_\alpha} n_\alpha(-x)
\end{equation}
and
\begin{equation}
\label{(7) type I}
\overline{s_\alpha}^{-1} n_\alpha(x)  \overline{s_\alpha} 
=
n_{-\alpha}(-x)
= n_{\alpha}(-x^{-1}) \alpha^{\vee}(x^{-1})  \overline{s_\alpha}  n_{\alpha} (-x^{-1}), \quad x \in F_\alpha^\times.
\end{equation}
On the other hand, suppose that $\alpha$ is of type (II).
Then we may take $\overline{s_\alpha}$ so that 
\begin{equation}
\label{(6) type II}
x_{\alpha}(r, m) x_{-\alpha}(-rm^{-1}, \bar{m}^{-1}) = \alpha^\vee(m)  \overline{s_\alpha}  x_{\alpha}(-(r \bar{m}) \slash m, \bar{m})
\end{equation}
and
\begin{equation}
\label{(7) type II}
\overline{s_\alpha} x_{\alpha}(r, m)  \overline{s_\alpha}^{-1}
= x_{-\alpha}(-r, m) = 
x_\alpha(-r \bar{m}^{-1},m^{-1}) \alpha^\vee(\bar{m}^{-1}) \overline{s_\alpha} x_{\alpha}(-r\bar{m}^{-1}, \bar{m}^{-1}).
\end{equation}
for any $r, m \in E_\alpha$ such that $N_{E_\alpha \slash F_\alpha}(r) = - \mathrm{Tr}_{E_\alpha \slash F_\alpha}(m)$.
%Suppose $\alpha_i$ is of type (II).
%Then we define 
%\[
%x_{(w, \alpha_i)}(r, m) = \overline{w_i}^{-1} x_{\alpha_i}(r, m) \overline{w_i}
%\]
%where $\overline{w_i} = \overline{s_{\alpha_{i-1}}} \cdots \overline{s_{\alpha_1}}$.
%
\begin{lemma}
\label{lemma2}
Let $w, w^\prime \in W$ and $\alpha \in \Phi_{{\rm rel}, +}$.
Assume that $w^\prime s_\alpha \leq w$ and $w \alpha \in \Phi_{{\rm rel}, +}$.
Then $w \alpha \in \Phi(S(w {w^{\prime}}^{-1}))$.
\end{lemma}
\begin{proof}
This is just a restatement of \cite[Lemma~2]{LM13} for $\Phi_{\rm rel}$.
\end{proof}
\begin{lemma}
\label{lemma3}
Let $K$ be a compact open normal subgroup of $\mathbf{K}$.
Then for any $g \in G$ and $\alpha \in \triangle_0$, either
\begin{enumerate}
\item there exists $k \in K$ such that $\mw (g \overline{s_\alpha} k) =\mw(g) s_\alpha$ and 
$H(a(g \overline{s_\alpha} k)) = H(a(g))$, or
\item  $\mw(g) \alpha \in \Phi_{{\rm rel}, -} , \mw(g \overline{s_\alpha}) = \mw(g)$ and 
$H(a(g \overline{s_\alpha})) - H(a(g)) = x  (\mw(w)\cdot  (\alpha^\vee))$ with  $x \ll_K 1$.
\end{enumerate}
\end{lemma}
\begin{proof}
We follow the proof of \cite[Lemma~3]{LM13}.
Indeed, when $\alpha$ is of type (I), our lemma is proved in the same way as their lemma.
Now, suppose that $\alpha$ is of type (II).
Further, we may suppose that $g = \overline{w} n$ where $w = \mw(g) \in \mathbf{W}$ and $n \in N_w^{-}$.
When $w \alpha \in \Phi_{{\rm rel}, +}$, the case (1) holds with $k=1$ since $\overline{s_\alpha}^{-1} n  \overline{s_\alpha} \in N$.
Hence, we may assume that $w \alpha \in \Phi_{{\rm rel}, -}$ and let $w^\prime = w s_\alpha$ so that $w^\prime \alpha \in \Phi_{{\rm rel}, +}$.

Write $n = n_1 x_\alpha(r, m) = x_\alpha(r, m) n_2$ where $n_1, n_2 \in N^\alpha$ and $r, m \in E_\alpha$ such that $N_{E_\alpha \slash F_\alpha}(r) 
= -\mathrm{Tr}_{E_\alpha \slash F_\alpha}(m)$.
Suppose that $|m|$ is so small that $x_\alpha(r, m) \in K$. 
Then $k:= \overline{s_\alpha}^{-1} x_\alpha(-r, \bar{m}) \overline{s_\alpha} \in \mathbf{K}$,
and we have
\[
g \overline{s_\alpha} k =\overline{w} n_1  x_\alpha(r, m) \overline{s_\alpha}  k =\overline{w} n_1 \overline{s_\alpha}
=\ast \overline{w^\prime} \overline{s_\alpha}^{-1} n_1 \overline{s_\alpha} \in \ast \overline{w^\prime} N.
\]
Otherwise, we write 
\[
g \overline{s_\alpha} =  \overline{w} x_\alpha(r, m) n_2 \overline{s_\alpha}
= \overline{w} x_\alpha(r, m) \overline{s_\alpha} n_3
\]
where $n_3 =  \overline{s_\alpha}^{-1 }n_2 \overline{s_\alpha} \in N$.
From \eqref{(7) type II}, we have
\begin{align*}
g \overline{s_\alpha} =
\ast \overline{w^\prime}  \overline{s_\alpha}^{-1} x_\alpha(r, m) \overline{s_\alpha} n_3
&=
\ast \overline{w^\prime} x_\alpha(-r \bar{m}^{-1},m^{-1}) \alpha^\vee(\bar{m}^{-1}) \overline{s_\alpha} x_{\alpha}(-r\bar{m}^{-1}, \bar{m}^{-1}) n_3
\\
&=
\ast (\overline{w^\prime} x_\alpha(-r \bar{m}^{-1},m^{-1}) \overline{w^\prime}^{-1}) 
(\overline{w^\prime}  \alpha^\vee(\bar{m}^{-1}) \overline{w^\prime}^{-1}) \overline{w} x_{\alpha}(-r\bar{m}^{-1}, \bar{m}^{-1}). 
\end{align*}
Since $w^\prime \alpha \in \Phi_{{\rm rel}, +}$, $ (\overline{w^\prime} x_\alpha(-r \bar{m}^{-1},m^{-1}) \overline{w^\prime}^{-1})  \in N$.
Further, we know 
\[
\overline{w^\prime}  \alpha^\vee(\bar{m}^{-1}) \overline{w^\prime}^{-1}
= \overline{w} \cdot( \alpha^\vee(m)) 
\]
%and 
%\[
%\overline{w}  t_\alpha(m) \overline{w}^{-1}(m) = t_{w \cdot \alpha}(m), \quad m \in F^\times.
%\]
Thus, we have
\[
H(a(g \overline{s_\alpha})) =  x  (\mw(w)\cdot  (\alpha^\vee))
\]
with $x \ll_K 1$.
\end{proof}
For any subset $S \subset \triangle_0$ and $X >0$, we define
\[
\mathcal{C}(S)_{\geq -X} = \left\{ \sum_{\alpha \in S} c_\alpha \alpha^\vee : c_\alpha \geq -X \quad {}^{\forall} \alpha \in S \right\}.
\]
For $w \in \mathbf{W}$, we fix its reduced decomposition by $w = s_{\alpha_k} \cdots s_{\alpha_1}$ with $\alpha_i \in \triangle_0$.
We shall use the following notational conventions, which implicitly depend on the reduced 
decomposition we chose (see the beginning of \cite[3.2]{LM13}).

For any $ i= 1, \dots, k$, let $w_i = s_{\alpha_{i-1}} \cdots s_{\alpha_1}$ and ${}_i w = s_{\alpha_k} \cdots s_{\alpha_{i+1}}$
so that $w = {}_i w s_{\alpha_i} w_i$.
We also write $\beta_i = w_i^{-1} \alpha_i$ so that $\{ \beta_1, \dots, \beta_{k} \} = \{ \beta \in \Phi_{{\rm rel}, +} : w \beta \in \Phi_{{\rm rel}, -} \}$.
Note that $w \beta_i = - {}_i w \alpha_i \in \Phi_{\rm rel}(S(w))$ for all $i$.

We may also write $N_w^- = {}_i N N_{\beta_i} N_i$ where $N_i = N_{w_i}^-$
and ${}_i N = w_i^{-1} s_{\alpha_i}^{-1} N_{{}_i w}^- s_{\alpha_i} w_i = w_{i+1}^{-1} N_{{}_i w}^- w_{i+1}$. 
Note that $N_1 = {}_k N = 1$, $N_{i+1} = N_{\beta_i} \ltimes N_i$ and ${}_{i-1} N = {}_i N \rtimes N_{\beta_i}$.
Let $n \in N_w^-$. For any $i$, we can write uniquely as 
\[
n =
\left\{
\begin{array}{ll}
 {}_i n  n_{\beta_i}(x_i) n_i & \text{ if $\alpha_i$ if of type (I)}\\
 &\\
  {}_i n n_{\beta_i}(r_i, m_i) n_i& \text{ if $\alpha_i$ if of type (II)}
\end{array}
\right.
\]
where ${}_i n \in {}_i N$ $n_i \in N_i$, $x_i \in F_{\alpha_i}$ and $r_i, m_i \in E_{\alpha_i}$
such that $N_{E_{\alpha_i} \slash F_{\alpha_i}}(r_i) = -\mathrm{Tr}_{E_{\alpha_i} \slash F_{\alpha_i}}(m_i)$.
We have
\begin{equation}
\label{n i+1}
n_{i+1} =
\left\{
\begin{array}{ll}
n_{\beta_i}(x_i) n_i & \text{ if $\alpha_i$ if of type (I)}\\
 &\\
n_{\beta_i}(r_i, m_i) n_i& \text{ if $\alpha_i$ if of type (II)}
\end{array}
\right.
\end{equation}

%%%%%%%%%%%%%%%%%%%%%%%%%%
%
%
%
%
%
%
%
%
%
%
%
%
%
%%%%%%%%%%%%%%%%%%%%%%%%%%
\begin{lemma}
\label{lemma4}
Let $w = {}_i w s_{\alpha_i} w_i$ be as above and let $g = {}_i w n w_i$ with $n \in N$.
Then $\mw(g) < w$.
Moreover, for any $K$ and any $n \in N$, there exists $k \in K$ such that
\begin{enumerate}
\item $\mw(gk) = {}_i w \tilde{w}_i$ for some $\tilde{w}_i \leq w_i$ (depending on $g$), and
\item $H(a(gk)) \in \mathcal{C}(S^\circ(\mw(gk)))_{\geq -X}$ with $X \ll_K 1$.
\end{enumerate}
\end{lemma}
\begin{proof}
This lemma is proved similarly as \cite[Lemma~4]{LM13} using Lemma~\ref{lemma3} instead of \cite[Lemma~3]{LM13}
\end{proof}
%%%%%%%%%%%%%%%%%%%%%%%%%%
%
%
%
%
%
%
%
%
%
%
%
%
%
%%%%%%%%%%%%%%%%%%%%%%%%%%
\begin{lemma}
\label{lmm5}
Let $w$ be as before and suppose that $\alpha_i$ is of type (II).
Let $g = \overline{w} nw_i^{-1} x_{\alpha_i}(r, m) w_i$ with $n \in {}_i N$ and $r, m \in E_{\alpha_i}$ such that 
$N_{E_{\alpha_i} \slash F_{\alpha_i}}(r) = - \mathrm{Tr}_{E_{\alpha_i} \slash F_{\alpha_i}}(m)$.

Then $\mw(g w_i^{-1} x_{-\alpha_i}(-rm^{-1}, \bar{m}^{-1}) w_i) < w = \mw(g)$.
Moreover, given $K$, either $|m| \ll_K 1$ or there exists $k \in K$ such that 
\begin{enumerate}
\item $\mw(gk) < w = \mw(g)$
\item $H(a(gk)) - \nu(m) w w_i^{-1} \cdot( \alpha_i^\vee) \in \mathcal{C}(S^\circ(\mw(gk)))_{-X}$ with $X \ll_K 1$, and
\item $w w_i^{-1} \alpha_i \in \Phi_{\rm rel}(S^\circ(\mw(gk)))$.
\end{enumerate}
\end{lemma}
\begin{proof}
Let $k:= w_i^{-1} x_{-\alpha_i}(-rm^{-1}, \bar{m}^{-1}) w_i$.
If $|m|$ is sufficiently large with respect to $K$, then $k \in K$.
By \eqref{(6) type II}, we have
\[
gk=
\overline{w} n w_i^{-1} \alpha_i^\vee(m)  \overline{s_{\alpha_i}}  x_{\alpha_i}(-(r \bar{m}) \slash m, \bar{m}) w_i,
\]
which we can write as
\[
( \overline{w} w_i^{-1} \alpha_i^\vee(m) w_i  \overline{w}^{-1} ) \overline{w} n^\prime w_i^{-1}  \overline{s_{\alpha_i}}  x_{\alpha_i}(-(r \bar{m}) \slash m, \bar{m}) w_i
\]
where $n^\prime = (w_i^{-1} \alpha_i^\vee(m) w_i)^{-1}  n w_i^{-1} \alpha_i^\vee(m) w_i $.
Changing $n^\prime$ by a conjugate of it by an element of $T \cap \mathbf{K}$, if necessary, we can write this as
\[
\ast( \overline{w} w_i^{-1} \alpha_i^\vee(m) w_i  \overline{w}^{-1} ) \overline{w} n^\prime \overline{s_{\alpha_i} w_i}^{-1}  x_{\alpha_i}(-(r \bar{m}) \slash m, \bar{m}) w_i
=\ast( \overline{w} w_i^{-1} \alpha_i^\vee(m) w_i  \overline{w}^{-1} ) \overline{{}_i w} n^{\prime \prime}  x_{\alpha_i}(-(r \bar{m}) \slash m, \bar{m}) w_i
\]
where $n^{\prime \prime} = \overline{s_{\alpha_i} w_i} n^{\prime} \overline{s_{\alpha_i} w_i}^{-1} \in N_{{}_i w}^{-}$.
Then we conclude that 
\[
gk \in \ast( \overline{w} w_i^{-1} \alpha_i^\vee(m) w_i  \overline{w}^{-1} ) \overline{{}_i w} N \overline{w}_i.
\]
All but the last part directly follows from Lemma~\ref{lemma4}.
Further, the last part follows from Lemma~\ref{lemma2} as in the proof of \cite[Lemma~5]{LM13}.
\end{proof}
Similarly, when $\alpha_i$ is of type (I), we have the following lemma,
which is proved in the same way as the proof of \cite[Lemma~5]{LM13}.
\begin{lemma}
\label{lmm5 I}
Let $w$ be as before and suppose that $\alpha_i$ is of type (I).
Let $g = \overline{w} nw_i^{-1} n_{\alpha_i}(x) w_i$ with $n \in {}_i N$ and $x \in F_{\alpha_i}$.

Then $\mw(g w_i^{-1} x_{-\alpha_i}(-x^{-1}) w_i) < w = \mw(g)$.
Moreover, given $K$, either $|x| \ll_K 1$ or there exists $k \in K$ such that 
\begin{enumerate}
\item $\mw(gk) < w = \mw(g)$
\item $H(a(gk)) - \nu(m) w w_i^{-1} \cdot (\alpha_i^\vee) \in \mathcal{C}(S^\circ(\mw(gk)))_{-X}$ with $X \ll_K 1$, and
\item $w w_i^{-1} \alpha_i \in \Phi_{\rm rel}(S^\circ(\mw(gk)))$.
\end{enumerate}

\end{lemma}
\begin{lemma}
\label{lmm6}
Let $w \in \mathbf{W}$ and let $\Phi_w = \{ \beta \in \Phi_{{\rm rel}, +} : w^{-1} \beta < 0 \}$.
Then for any $n \in N$ we have $H(\overline{w} n) = \sum_{\alpha \in \Phi_w} c_\alpha \alpha^\vee$ with $c_\alpha \leq 0$, for all $\alpha \in \Phi_w$.
Thus, $|\alpha^\ast|(a(\overline{w}n)) \leq 1$ for all $\alpha \in \triangle_0$
with equality if $\alpha \not \in S(w)$.
Moreover, the map $n \mapsto H(\overline{w} n)$ from $N_w^-$ to $\mathfrak{a}_T$ is proper.
\end{lemma}
\begin{proof}
This lemma is proved in the same way as \cite[Lemma~6]{LM13} using the relation
\[
H(\overline{s_\alpha} x_{\alpha}(r, m)) = \mathrm{min}(0, -v(y)) \alpha^\vee
\]
for any $\alpha \in \triangle_0$ of type (I) (resp. type (II))  because of \eqref{(7) type I} (resp. \eqref{(7) type II}).
%If $\alpha$ is of type (I), this is  \cite[Lemma~6]{LM13}.
\end{proof}
The following lemma is proved in the same way as \cite[Lemma~7]{LM13},
which is a special case of Proposition~\ref{prp1}.
\begin{lemma}
\label{lmm7}
For any $W \in \Omega^\circ(N \backslash G, \psi_N)$, $w \in \mathbf{W}$ and $\varepsilon > 0$, $\mathrm{supp}_{A^\varepsilon(w) w} W$ is compact.
\end{lemma}
%%%%%%%%%%%%%%%%%%%%%%%%%%
%
%
%
%
%
%
%
%
%
%
%
%
%
%%%%%%%%%%%%%%%%%%%%%%%%%%
\subsection{Proof of of Proposition~\ref{prp1}}
We follow the argument in \cite[3.4]{LM13} (see also the proof of \cite[Theorem~9.5]{Ba05}).

We prove by induction on $\ell(w)$ that for any $w \in \mathbf{W}$ and $\varepsilon > 0$
\begin{equation}
\label{17}
\text{
for any $W \in \Omega^\circ(N \backslash G, \psi_N)$,
$\mathrm{supp}_{A^\varepsilon(w)wN} W$ is bounded modulo $N$.}
\end{equation}
For $w = 1$, this is a special case of Lemma~\ref{lmm7} above.
Assume now \eqref{17} holds for all $w^\prime$ with $\ell(w^\prime) < \ell(w)$.

Let $w = s_{\alpha_k} \cdots s_{\alpha_1}$ be a reduced decomposition of $w$ and 
let $\beta_1, \dots, \beta_k$, ${}_i w, w_i, N_i, {}_i N$ be as above.
Let 
$n \in N_w^-$ and write 
\[
n =
\left\{
\begin{array}{ll}
 {}_i n  n_{\beta_i}(x) n_i & \text{ if $\alpha_i$ if of type (I)}\\
 &\\
  {}_i n n_{\beta_i}(r, m) n_i& \text{ if $\alpha_i$ if of type (II)}
\end{array}
\right.
\]
To prove \eqref{17}, we will show by induction on $i$ that 
\begin{equation}
\label{18}
\text{for any $W \in \Omega^\circ(N \backslash G, \psi_N)$, $\mathrm{supp}_{NA^{\varepsilon}(w) {}_iN} W$
is bounded modulo $N$.}
\end{equation}
For $i=k$ this follows from Lemma~\ref{lmm7} since ${}_k N = 1$.
For the induction step, assume that \eqref{18} holds for $i$ and we will show it for $i-1$.

Let $W \in \Omega^\circ(N \backslash G, \psi_N)^K$.
Suppose that $\alpha_i$ is of type (I), and 
assume that $W(g) \ne 0$ with $g = t \bar{w} n n_{\beta_i}(x)$, $n \in {}_i N$ and $t \in A^\varepsilon(w)$.
Then we may apply the same argument as the proof of \cite[Proposition~1]{LM13} by Lemma~\ref{lmm5 I},
and thus \eqref{18} holds for $i-1$.
Now, suppose that  $\alpha_i$ is of type (II), and assume that 
 $W(g) \ne 0$ with $g = t \bar{w} n w_i^{-1} n_{\alpha_i}(r, m) w_i$, $n \in {}_i N$  and $t \in A^\varepsilon(w)$.
 %Then let us denote by $\{ a_i \}$ a set of representatives of $T \slash A (T \cap K)$
%and we put $W_i = W(\cdot a_i)$. Here, we note that $\# \left( T \slash A (T \cap K) \right) < \infty$.
 
 By Lemma~\ref{lmm5}, either $|m| \ll_W 1$ or there exists $k \in K$
 such that $w^\prime:= \mathbf{w}(gk) < w$, $w \beta_i \in \Phi_{\rm rel}(S^\circ(w^\prime))$ and 
 if $t^\prime = a(gk)$, then $H(t^\prime)-H(t) - \nu(m) w w_i^{-1} \cdot (\alpha_i^\vee) \in \mathcal{C} (S^\circ(w^\prime))_{-T}$
 with $T \ll_K 1$.
 
 In the former case, we can use the induction hypothesis of \eqref{18}
 for the finitely may translates $R(n^\prime)W$ where $n^\prime$ lies in a suitable compact subgroup 
 of $\overline{w_i}^{-1} N_{\alpha_i} \overline{w_i}$ (depending only on $W$).
 
 In the latter case, 
 \[
 |\alpha^\ast|(t^\prime) = |\alpha^\ast|(t)
 \]
 for all $\alpha \not \in S^\circ(w^\prime)$.
 Hence, $t^\prime \in A^\varepsilon(w^\prime)$ since $t \in A^\varepsilon(w)$. 
 Since $w^\prime < w$, we may apply the inductive assumption, namely \eqref{17} for $w^\prime$.
 Then we find that $gk$ and therefore, $g$ is compactly supported modulo $N$, which finishes the proof of Proposition~\ref{prp1}.
%%%%%%%%%%%%%%%%%%%%%%%%%%
%
%
%
%
%
%
%
%
%
%
%
%
%
%%%%%%%%%%%%%%%%%%%%%%%%%%
\subsection{Proof of Lemma~\ref{lmm1}}
This is proved in the same way as the proof of \cite[Lemma~1]{LM13}
using Lemma~\ref{lmm6} instead of \cite[Lemma~6]{LM13}.
%%%%%%%%%%%%%%%%%%%%%%%%%%
%
%
%
%
%
%
%
%
%
%
%
%
%
%%%%%%%%%%%%%%%%%%%%%%%%%%
\subsection{Proof of Proposition~2}
Let $U_0 = N \cap \mathbf{K}$.
Fix $a \in A$ such that $|\alpha(a)| >1$, for all $\alpha \in \triangle_0$ (and hence for all $\alpha \in \Phi_{{\rm rel}, +}$).
For any $m \geq 1$ define $U_m= a^m U_0 a^{-m}$.
Thus, $U_1 \subset U_2 \subset \cdots$ and $\cup_{m=1}^\infty U_m = N$.
Set
\[
W_m := R_{U_m, \psi_N} W = \frac{1}{\mathrm{vol}(U_m)}\int_{U_m} R(u) W \psi_N(u)^{-1} \, du \in \Omega(N \backslash G, \psi_N).
\]
Clearly,
\begin{equation}
\label{19}
W_m(gu) = \psi_N(u) W_m(g), \quad \text{for all } u \in U_m, g \in G.
\end{equation}
We fix the reduced decomposition $s_{\tilde{\alpha}_l} \cdots s_{\tilde{\alpha}_1}$ of $w_0$.
Note that $N_{w_0}^- = N$
Set 
\[
\tilde{\beta}_i = \tilde{w}_i^{-1} \tilde{\alpha}_i.
\]
If necessary, after renumbering, we may suppose that $\tilde{\alpha}_1, \dots, \tilde{\alpha}_{s}$ are of  type (I) and $\tilde{\alpha}_{s+1}, \dots, \tilde{\alpha}_{l}$ are ot type (II).
\begin{lemma}[cf. Lemma~8 in \cite{LM13}]
\label{lemma8}
Let $n = n_{\tilde{\beta}_l}(x_l) \cdots n_{\tilde{\beta}_m}(x_m) \cdot x_{\beta_{m+1}}(p_{m+1}, q_{m+1}) \cdots x_{\beta_{1}}(p_{1}, q_{1})$.
Then $u \in U_m$ if and only if $|x_i| \leq |\tilde{\beta}_i(a^m)|$ and $|q_{j}| \leq |\tilde{\beta}_i(a^{2m})|$
for all $i, j$.
\end{lemma}
\begin{proof}
Clearly, we have
\[
a^m (N_{\tilde{\beta}_i} \cap \mathbf{K}) a^{-m} = n_{\tilde{\beta}_i}(\tilde{\beta}_i(a^m) \mathcal{O})
\]
for $1 \leq i \leq s$, and 
\[
a^m (N_{\tilde{\beta}_i} \cap \mathbf{K}) a^{-m} = \{ n_{\tilde{\beta}_i}(\tilde{\beta}_i(a^m)p, \tilde{\beta}_i(a^{2m})q) : p, q \in \mathcal{O}_{E_{\alpha_i}}, 
N_{E_{\alpha_i} \slash E_{\alpha_i}}(p) = - \mathrm{Tr}_{E_{\alpha_i} \slash F_{\alpha_i}}(q) \}
\]
for $s+1 \leq i \leq l$.
If $\tilde{\beta}_i(a^{2m})q \in \mathcal{O}_{E_{\alpha_i}}$, then $\tilde{\beta}_i(a^m)p \in \mathcal{O}_{E_{\alpha_i}}$ because of the relation 
$N_{E_{\alpha_i} \slash F_{\alpha_i}}(p) = - \mathrm{Tr}_{E_{\alpha_i} \slash F_{\alpha_i}}(q)$.
Then the "if" direction follows.

Let prove the "only if" direction for $j$.
For the induction step, as well as for the base of the induction, we may assume that $p_1 = q_1 = \cdots =p_{j-1} = q_{j-1} =0$,
i.e. $n \in {}_{i-1} \tilde{N}$.
In this case, we observe that 
\[
x_{\beta_{j}}(p_{j}, q_{j}) = \overline{\tilde{w}_i}^{-1} \pi_{\tilde{\alpha}_i} (\overline{\tilde{w}_i} n \overline{\tilde{w}_i}^{-1}) \overline{\tilde{w}_i},
\]
where $\pi_{\tilde{\alpha}_i} : N \rightarrow N_{\tilde{\alpha}_i}$ is the canonical projection.
Write $n = a^m n^\prime a^{-m}$ where $n^\prime \in {}_{j-1} \tilde{N} \cap \mathbf{K}$.
Then since $\pi_{\tilde{\alpha}_i}$ is equivariant with respect to conjugation by $A$ we get
\[
x_{\beta_{j}}(p_{j}, q_{j})  = \overline{\tilde{w}_i}^{-1} \pi_{\tilde{\alpha}_i} (\overline{\tilde{w}_i} a^m n^\prime a^{-m} \overline{\tilde{w}_i}^{-1}) \overline{\tilde{w}_i}
=a^m  \overline{\tilde{w}_i}^{-1} \pi_{\tilde{\alpha}_i} (\overline{\tilde{w}_i}  n^\prime   \overline{\tilde{w}_i}^{-1}) \overline{\tilde{w}_i} a^{-m},
\]
or,
\[
x_{\beta_{j}}(\tilde{\beta}_i(a^{-m})p_{j}, \tilde{\beta}_i(a^{-2m})q_{j})   = \overline{\tilde{w}_i}^{-1} \pi_{\tilde{\alpha}_i} (\overline{\tilde{w}_i}  n^\prime   \overline{\tilde{w}_i}^{-1}) \overline{\tilde{w}_i} \in \mathbf{K}
\]
Then the required inequality follows.
\end{proof}
%%%%%%%%%%%%%%%%%%%%%%%%%%
%
%
%
%
%
%
%
%
%
%
%
%
%
%%%%%%%%%%%%%%%%%%%%%%%%%%
\begin{lemma}[cf. Lemma~9 in \cite{LM13}]
Let $n = {}_i n x_{\tilde{\beta}_i}(p_i, q_i)$ with ${}_i n  \in {}_i \tilde{N}$ and $p_i, q_i \in E_{\alpha_i}$ such that 
$N_{E_{\alpha_i} \slash F_{\alpha_i}}(p_i) = -\mathrm{Tr}_{E_{\alpha_i} \slash F_{\alpha_i}}(q_i)$.
Assume that $x_{\tilde{\beta}_i}(p_i, q_i) \not \in U_m$.
Then for any $n^\prime \in U_m$, we have $n n^\prime = {}_i \tilde{n} x_{\tilde{\beta}_i}(\tilde{p}_i, \tilde{q}_i) \tilde{n}_i$ where ${}_i \tilde{n} \in {}_i \tilde{N}$
$\tilde{n}_i \in U_m \cap \tilde{N}_i$ and $|q_i| = |\tilde{q}_i|$.
\end{lemma}
\begin{proof}
Write $n^\prime = {}_i n^{\prime} x_{\tilde{\beta}_i}(p_i^\prime, q_i^\prime) n_i^\prime$ with ${}_i n^{\prime}  \in {}_i \tilde{N}$
$p_i, q_i \in E$ such that $N_{E \slash L}(p_i) = - \mathrm{Tr}_{E \slash L}(q_i)$ and $n_i^\prime \in \tilde{N}_i$.
Then
\[
n n^\prime = {}_i n x_{\tilde{\beta}_i}(p_i, q_i) {}_i n^{\prime} x_{\tilde{\beta}_i}(p_i^\prime, q_i^\prime) n_i^\prime
=  {}_i n n^{\prime \prime} x_{\tilde{\beta}_i}(p_i+ p_i^\prime, q_i+q_i^\prime-p_i \overline{p_i^\prime}) n_i^\prime,
\]
where $n^{\prime \prime} = x_{\tilde{\beta}_i}(p_i, q_i)  {}_i n^{\prime} x_{\tilde{\beta}_i}(p_i, q_i)^{-1} \in {}_i \tilde{N}$.
From the definition of $N_i$, we may write $n_i = n_{\tilde{\beta}_{l(i)}}(x_{l(i)}) \cdots n_{\tilde{\beta}_{m(i)}}(x_{m(i)}) \cdot x_{\beta_{m(i)+1}}(p_{m(i)+1}, q_{m(i)+1}) \cdots x_{\beta_{1}}(p_{1}, q_{1})$.
Since $n_{j+1} = n_{j} n_{\beta_j}$ or $n_{j+1} = n_{j} x_{\beta_j}$, we see that $|q_1| \leq |\beta_i(a^{2m})|$ by Lemma~\ref{lemma8}.
Repeatedly using the lemma, we see that $n_i^\prime \in U_m$.
Further, we see that 
\[
|q_i^\prime| \leq |\beta_i(a^{2m})|.
\]
Since  $x_{\tilde{\beta}_i}(p_i, q_i) \not \in U_m$, we have
\[
 |\beta_i(a^{2m})|< |q_i|.
\]
Then it is easy to see that 
\[
|p_i \overline{p_i^\prime}|< |q_i|,
\]
and thus
\[
|q_i+q_i^\prime-p_i \overline{p_i^\prime}| = |q_i|.
\]
The lemma follows.
\end{proof}
%%%%%%%%%%%%%%%%%%%%%%%%%%
%
%
%
%
%
%
%
%
%
%
%
%
%
%%%%%%%%%%%%%%%%%%%%%%%%%%
Now let $w \in \mathbf{W}$ and set $N_w^+ = N \cap \bar{w}^{-1} N\bar{w}$.
Fix a reduced decomposition $s_{\alpha_k} \cdots s_{\alpha_1}$ of $w$ and use the notation of Section~~\ref{bruhat order}.
\begin{corollary}
\label{cor1 (II)}
Let $i=1, \dots, k$ and $n= {}_i n n_{\beta_i}(p_i, q_i)$ with ${}_i N$ and $p_i, q_i \in E_{\alpha_i}$ such that 
$N_{E_{\alpha_i} \slash F_{\alpha_i}}(p_i) =- \mathrm{Tr}_{E_{\alpha_i} \slash F_{\alpha_i}}(q_i)$.
Assume that $n_{\beta_i}(p_i, q_i) \not \in U_m$.
Then for any $n^\prime \in U_m$, we have $n n^\prime = {}_i \tilde{n} n_{\tilde{\beta}_i}(\tilde{p}_i, \tilde{q}_i) \tilde{n}_i$ where ${}_i \tilde{n} \in N_{w}^+ \cdot {}_i \tilde{N}$
$\tilde{n}_i \in U_m$ and $|q_i| = |\tilde{q}_i| > |\beta_i(a^{2m})|$.
\end{corollary}
For the convenience to the reader, we record a similar result in the case of type (I) by Lapid-Mao~\cite{LM13}.
 \begin{corollary}[Corollary~1 in \cite{LM13}]
\label{cor1 (I)}
Let $i=1, \dots, k$ and $n= {}_i n n_{\beta_i}(x_i)$ with ${}_i N$.
Assume that $n_{\beta_i}(x_i) \not \in U_m$.
Then for any $n^\prime \in U_m$, we have $n n^\prime = {}_i \tilde{n} n_{\tilde{\beta}_i}(\tilde{p}_i, \tilde{q}_i) \tilde{n}_i$ where ${}_i \tilde{n} \in N_{w}^+ \cdot {}_i \tilde{N}$
$\tilde{n}_i \in U_m$ and $|x_i| > |\beta_i(a^{m})|$.
\end{corollary}

%%%%%%%%%%%%%%%%%%%%%%%%%%
%
%
%
%
%
%
%
%
%
%
%
%
%
%%%%%%%%%%%%%%%%%%%%%%%%%%
\subsubsection{A special case}
\begin{lemma}
\label{lmm10}
There exists $M$ such that for all $W \in \Omega(N \backslash G, \psi_N)$, $w \in \mathbf{W}$ $m \geq M$ and $\alpha \in S^\circ(w)$
we have $\mathrm{inf} |\alpha_0| (\mathrm{supp}_{Bw} W_m) > 0$.
Here, $\alpha_0$ is an extension of $\alpha$ to $T$
\end{lemma}
\begin{proof}
Suppose that $\alpha \in \triangle_0$ is of type (I) (resp. type (II)). Then we let $\psi_\alpha : F_\alpha \rightarrow \mC$ (resp. $\psi_\alpha:  E_\alpha \rightarrow \mC$) 
be the nontrivial character defined by 
\[
\psi(n_\alpha(x)) \quad
\left(\text{resp. } \psi \left(x_\alpha \left(x, -\frac{N_{E_\alpha \slash F_\alpha}(x)}{2} \right) \right) \right)
\]
Denote by $\mathrm{cond}(\psi_\alpha)$ its conductor, namely the maximal fractional ideal of $\mathcal{O}_{F_\alpha}$ or $\mathcal{O}_{E_\alpha}$
on which $\psi_\alpha$  is trivial.
For any $\alpha, \beta \in \triangle_0$ of the same type, let $c_{\alpha, \beta} \in F_\alpha^\times$ or $E_\alpha^\times$ (depending on the type)
 such that $\psi_{\beta} = \psi_{\alpha}(c_{\alpha, \beta} \cdot)$.

Suppose that $\beta \in \triangle_0$ is of type (II) and there exists $\alpha \in \Phi_{{\rm rel}, +} \setminus \triangle_0$ such that $\beta = w^{-1} \alpha$.
Then we have
\[
\psi_{\beta}(x) W_m(t \overline{w}) = W_m \left(t  \overline{w} x_\beta \left(x, -\frac{N_{E_\alpha \slash F_\alpha}(x)}{2} \right) \right)
= W_m \left(x_\alpha \left(\ast \alpha_0(t)x, -\frac{N_{E_\alpha \slash F_\alpha}(\ast \alpha_0(t)x)}{2} \right)  t  \overline{w} \right) = W_m(t \overline{w})
\]
for any $t \in T$ and all $x \in \beta(a^m) \mathcal{O}_{E_\alpha}$. 
It follows that in fact $W_m(s \overline{w}) = 0$ for all $s \in T$ provided that $\beta(a^m) \not \in \mathrm{cond(\psi_\alpha)}$.
If $\beta \in \triangle_0$ is of type (I) and there exists $\alpha \in \Phi_{{\rm rel}, +} \setminus \triangle_0$ such that $\beta = w^{-1} \alpha$,
then we can prove our claim in a similar way as above or 
as the proof of \cite[Lemma~10]{LM13}.

On the other hand, suppose that these does not exist $\beta \in \triangle_0$ such that $\beta = w^{-1} \alpha$ for some $\alpha \in \Phi_{{\rm rel}, +} \setminus \triangle_0$.
Then using Steinberg \cite[Lemma~89]{St}, in a similar way as \cite[Lemma~10]{LM13}, we find that  if $W_m(s \overline{w}) \ne 0$ and $m$ is sufficiently large, then 
\[
|\alpha_0|(s) = |c_{\alpha, w^{-1} \alpha}|
\]
which concludes our proof.
\end{proof}
\subsubsection{The general case}
To prove Proposition~\ref{prp2}, we will show by induction on $\ell(w)$ that, for any $w \in \mathbf{W}$,
\begin{equation}
\label{20}
\begin{array}{r}
\text{there exists $M$ depending on $K$ such that for any $W \in \Omega(N \backslash G, \psi_N)^K$, $m \geq M$}
\\
\text{and $\alpha \in S^\circ(w)$ we have $\mathrm{inf} |\alpha_0| (\mathrm{supp}_{BwB} W_m) > 0$.}
\end{array}
\end{equation}
The case $w = 1$ follows from Lemma~\ref{lmm10}.
To carry out the induction step assume that \eqref{20}
holds for all $w^\prime < w$.

Fix a reduced decomposition of $w$ and use the notation of Section~\ref{bruhat order}.
For any $m$ and $i= 1, \dots, k$, let
\[
\mathcal{B}_w(i, m) = \{ b\bar{w}n : b \in B, n \in N_w^-, n_i \in U_m, n_{i+1} \not \in U_m \}.
\]
Consider the following auxiliary statement.
\begin{equation}
\label{21}
\begin{array}{r}
\text{There exists $M$ depending on $K$ such that for any $W \in \Omega(N \backslash G, \psi_N)^K$, $m \geq M$}
\\
\text{and $\alpha \in S^\circ(w)$ we have $\mathrm{inf} |\alpha_0| (\mathrm{supp}_{\mathcal{B}_w(i, m)} W_m) > 0$.}
\end{array}
\end{equation}
We will show this statement by induction on $i$.
This will yield \eqref{20} for $w$.
Indeed, we may take $M$ for which \eqref{21} holds for all $i$.
Then, for any $m \geq M$ we have
\[
\mathrm{inf} |\alpha_0| (\mathrm{supp}_{\cup_{i=1}^k \mathcal{B}_w(i, m)} W_m) > 0
\]
On the other hand, the complement of $\cup_{i=1}^k \mathcal{B}_w(i, m)$ in $\mathcal{B}_w$
is $BwU_m$, and by \eqref{19} we have
\[
\mathrm{supp}_{BwU_m} W_m = (\mathrm{supp}_{Bw} W_m)U_m.
\]
Therefore, Lemma~\ref{lmm10}. implies that 
\[
\mathrm{inf} |\alpha_0| (\mathrm{supp}_{BwU_m} W_m) > 0
\]
for all $\alpha \in S^\circ(w)$ as well.

It remains to prove \eqref{21}.
By \eqref{19}, we may replace $\mathcal{B}_w(i, m)$ in \eqref{21} by the set
\[
\mathcal{B}_w^\prime(i, m)
=\{ b\bar{w}n : b \in B, n \in N_w^-, n_i =1, n_{i+1} \not \in U_m \}.
\]
Let $M$ (depending on $K$) such that \eqref{21} holds for all $j < i$
and \eqref{20} holds for all $w^\prime < w$.
We choose $M_1 \geq M$ depending only on $K$ such that 
\[
n_{-\beta_i}(\beta_i(a^{-M_1}) \mathcal{O}) \subset \bigcap_{n \in U_M} nKn^{-1}
\]
when $\alpha_i$ if of type (I), and
\[
\{ n_{-\beta_i}(\tilde{\beta}_i(a^{-M_1})p, \tilde{\beta}_i(a^{-2M_1})q) : p, q \in \mathcal{O}_{E_{\alpha_i}}, 
N_{E_{\alpha_i} \slash E_{\alpha_i}}(p) = - \mathrm{Tr}_{E_{\alpha_i} \slash F_{\alpha_i}}(q) \}
 \subset \bigcap_{n \in U_M} nKn^{-1}
\]
when $\alpha_i$ if of type (II) for $\beta_i = w_i^{-1} \alpha_i$.

Assume that $W \in \Omega(N \backslash G, \psi_N)^K$, $g \in \mathcal{B}_w^\prime(i, m)$,
and $W_m(g) \ne 0$ with $m \geq M_1$.
Write $g = t\bar{w}n$.
Since
\[
W_m(g) = \frac{1}{\mathrm{vol}(U_m)} \int_{U_m} W_M(gn^\prime) \psi_N(n^\prime)^{-1} \, dn^\prime
= 
\]
there exists $n^\prime \in U_m$ such that $W_M(gn^\prime) \ne 0$.
Let $n n^\prime = n_+ \tilde{n}$ where $n_+  \in N_w^+$ and $\tilde{n} \in N_w^-$.
Write 
\[
\tilde{n} = 
\left\{
\begin{array}{ll}
{}_i \tilde{n} n_{\beta_i}(\tilde{x}_i) \tilde{n}_i & \text{ if $\alpha_i$ is of type (I),} \\
&\\
{}_i \tilde{n} n_{\beta_i}(\tilde{r}_i, \tilde{m}_i) \tilde{n}_i  & \text{ if $\alpha_i$ is of type (II).}
\end{array}
\right.
\]
From our assumption on $n$ and $n^\prime$, by Corollary~\ref{cor1 (II)}, \ref{cor1 (I)}
we have $|\tilde{x}_i| > |\beta_i(a)^m|$ and $|\tilde{m}_i| > |\beta_i(a)^{2m}|$.
In particular, $\tilde{n}_{i+1} \not \in U_m$ by \eqref{n i+1}.

Let $j \leq i+1$ be the smallest index for which $\tilde{n}_j \not \in U_M$.
If $j \leq i$, then $\tilde{g} \in \mathcal{B}_w(j-1, M)$.
Thus, we may apply our inductive assumption, and by the choice of $M$
we have $|\alpha_0(\tilde{g})| \geq \delta_1$ for all $\alpha \in S^\circ(w)$
where $\delta_1 >0$ depends only on $W$.
Hence, by  \eqref{4} we also have 
$|\alpha_0(g)| \geq \delta_2$ for a suitable constant $\delta_2 = \delta_2(m, W) > 0$.

Assume that $j=i+1$.
Then $\tilde{n}_i \in U_M$ and therefore
$W_M(g^\prime) \ne 0$ where 
\[
g^\prime = 
\left\{
\begin{array}{ll}
t \bar{w} {}_i \tilde{n} n_{\beta_i}(\tilde{x}_i) & \text{ if $\alpha_i$ is of type (I),} \\
&\\
t \bar{w} {}_i \tilde{n} n_{\beta_i}(\tilde{r}_i, \tilde{m}_i) & \text{ if $\alpha_i$ is of type (II).}
\end{array}
\right.
\]
Here, we note that $g^\prime \in N\tilde{g}  \tilde{n}_i^{-1} $ since $tw(n_+)(tw)^{-1}$.
On the other hand, since $\tilde{n}_{i+1} \not \in U_M$, we get
\[
\left\{
\begin{array}{ll}
x_i^{-1} \in \beta_i(a^{-m}) \mathcal{O}_{F_{\alpha_i}} & \text{ if $\alpha_i$ is of type (I),} \\
&\\
m_i^{-1} \in \beta_i(a^{-m}) \mathcal{O}_{E_{\alpha_i}} & \text{ if $\alpha_i$ is of type (II),}
\end{array}
\right.
\]
and thus $W_M$ is right invariant by 
$n_{-\beta_i}(-\tilde{x}_i^{-1})$ (resp. $n_{-\beta_i}(-rm^{-1}, \bar{m}^{-1})$) if $\alpha_i$ is of type (I) (resp. type (II))
since these belong to $\cap n \in U_M nK n^{-1}$ by the choice of $M_1$.
Hence, $W_M(g^{\prime \prime}) \ne 0$
where $g^{\prime \prime}:= g^{\prime} n_{-\beta_i}(-\tilde{x}_i^{-1})$
(resp. $g^\prime n_{-\beta_i}(-rm^{-1}, \bar{m}^{-1})$)
 if $\alpha_i$ is of type (I) (resp. type (II)).
 By the first part of Lemma~\ref{lmm5}, \ref{lmm5 I}, 
 we have $g^{\prime \prime} \in Bw^{\prime} B$
 with $w^{\prime} < w$.
We conclude from the choice of $M$ that $|\alpha_0|(g^\prime) = |\alpha_0|(g^{\prime \prime}) \geq \delta_3$
for all $\alpha \in S^\circ(w^\prime)$ for a suitable constant $\delta_3 = \delta_3(W, M)$ by our inductive assumption.
In particular, this holds for all $\alpha^\circ(w)$. 
Once again by \eqref{4} we infer that $|\alpha_0|(g) \geq \delta_4$ 
for a suitable constant $\delta_4 = \delta_4(W, m)$. 
This concludes the proof of Proposition~\ref{prp2}.
%%%%%%%%%%%%%%%%%%%%%%%%%%
%
%
%
%
%
%
%
%
%
%
%
%
%
%%%%%%%%%%%%%%%%%%%%%%%%%%
\subsection{Non-vanishing of Bessel functions}
Suppose that $\mathbf{G}$ is quasi-split reductive group over $F$.
We shall use the same notation as above.
Put $B_0 = A \ltimes N$ and $G^\circ = B_0 w_0 B_0$.
\begin{theorem}
\label{Theorem A.1}
For any tempered $\pi \in \mathrm{Irr}_{{\rm gen}, \psi_N} G$,
the function $\mathbb{B}_\pi$ is not identically zero on $G^\circ$.
\end{theorem}
Let $A^d = A \cap G^{\rm der}$. Consider the family
\[
A_n = \{t \in A^d : |\alpha(t)-1| \leq q^{-n} \text{ for all $\alpha \in \triangle_0$} \} \in \mathcal{CSGR}(A^d)
\]
which forms a basis of neighborhoods of $1$ of $A^d$.

Let $U_n$ be the group generated by 
$\langle  n_\alpha(\varpi^{-n} \mathcal{O}_{F_\alpha}), n_{\beta}(\varpi^{-n}\mathcal{O}_{E_\beta}, \varpi^{-2n}\mathcal{O}_{E_\beta}) : \alpha \in \triangle_0^{(I)}, 
\beta  \in \triangle_0^{(II)}\rangle$.
\begin{lemma}
\label{lmmA4}
For any $u \in N$ we have
\[
(\mathrm{vol}(A_n))^{-1}
\int_{A_n} \psi_N(tut^{-1}) \,dt
= \left\{
\begin{array}{ll}
\psi_N(u) & u \in N_n \\
&\\
0 & \text{otherwise}
\end{array}
\right.
\]
\end{lemma}
In the same argument as the proof of \cite[Theorem~A.1]{LMb},
we can prove Theorem~\ref{Theorem A.1} substituting $U_n$ for $N_n$ in the argument \cite[Appendix~A]{LMb}.
%%%%%%%%%%%%%%%%%%%%%%%%%%
%
%
%
%
%
%
%
%
%
%
%
%
%
%%%%%%%%%%%%%%%%%%%%%%%%%%
%%%%%%%%%%%%%%%%%%%%%%%%%%
%
%
%
%
%
%
%
%
%
%
%
%
%
%%%%%%%%%%%%%%%%%%%%%%%%%%
%%%%%%%%%%%%%%%%%%%%%%%%%%%%%%%%%%%%%%%%%%%%%%%
%
%
%
%
%
%
%
%
%
%
%
%
%
%
%%%%%%%%%%%%%%%%%%%%%%%%%%%%%%%%%%%%%%%%%%%%%%%

%%%%%%%%%%%%%%%%%%%%%%%%%%%%%%%%%%%%%%%%%%%%%%%
%
%
%
%
%
%
%
%
%
%
%
%
%
%
%%%%%%%%%%arithmetic of automorphic forms%%%%%%%%%%%%%%%%%%%%%%
\end{document}